\documentclass[10pt,reqno]{amsart}
\usepackage{a4wide}
\usepackage[utf8]{inputenc}
\usepackage{amsmath}
\usepackage{amscd}
\usepackage{amsfonts}
\usepackage{array}
 \usepackage{xcolor} 
\usepackage{amssymb}
\usepackage{amsthm}
\usepackage{subfigure}
\usepackage[mathscr]{eucal}
\usepackage{dsfont}
\usepackage{pstricks}
\usepackage{graphicx}
\usepackage[all]{xy}
\usepackage{comment}

\usepackage{tikz}
\usetikzlibrary{shapes,arrows}
\usetikzlibrary{intersections}
\usetikzlibrary{patterns}

\usepackage[colorlinks=true, linkcolor=purple, citecolor=purple]{hyperref}

\newtheorem{theorem}{Theorem}[section]
\newtheorem{proposition}[theorem]{Proposition}
\newtheorem{corollary}[theorem]{Corollary}
\newtheorem{lemma}[theorem]{Lemma}

\newtheorem{ex}[theorem]{Example}
\theoremstyle{definition}
\newtheorem{definition}[theorem]{Definition}
\theoremstyle{remark}
\newtheorem{remark}[theorem]{Remark}

\newcommand{\nwc}{\newcommand}

\nwc{\eps}{\varepsilon}
\nwc{\ep}{\epsilon}
\nwc{\vareps}{\varepsilon}
\nwc{\Oph}{\operatorname{Op}_\hbar}
\nwc{\la}{\langle}
\nwc{\ra}{\rangle}

\nwc{\mf}{\mathbf} 
\nwc{\blds}{\boldsymbol} 
\nwc{\ml}{\mathcal} 

\nwc{\defeq}{\stackrel{\rm{def}}{=}}

\nwc{\cE}{\ml{E}}
\nwc{\cN}{\ml{N}}
\nwc{\cO}{\ml{O}}
\nwc{\cP}{\ml{P}}
\nwc{\cU}{\ml{U}}
\nwc{\cV}{\ml{V}}
\nwc{\cW}{\ml{W}}
\nwc{\tU}{\widetilde{U}}
\nwc{\IN}{\mathbb{N}}
\nwc{\IR}{\mathbb{R}}
\nwc{\R}{\mathbb{R}}
\nwc{\IZ}{\mathbb{Z}}
\nwc{\Z}{\mathbb{Z}}
\nwc{\N}{\mathbb{N}}
\nwc{\IC}{\mathbb{C}}
\nwc{\C}{\mathbb{C}}
\nwc{\IT}{\mathbb{T}}
\nwc{\T}{\mathbb{T}}
\nwc{\IS}{\mathbb{S}}
\nwc{\tP}{\widetilde{P}}
\nwc{\tPi}{\widetilde{\Pi}}
\nwc{\tV}{\widetilde{V}}
\nwc{\f}{\mathsf{f}}
\nwc{\supp}{\operatorname{supp}}
\nwc{\rest}{\restriction}
\let \d \relax
\nwc{\d}{\partial}
\nwc{\Cor}{\mathscr{C}}
\nwc{\CLam}{\overline{\C}_+^\Lambda}

\nwc{\todo}[1]{$\clubsuit$ {\tt #1}}
\DeclareMathOperator{\vect}{span}
\DeclareMathOperator{\Vol}{Vol}

\DeclareMathOperator{\Hess}{Hess}

\DeclareMathOperator{\Id}{Id}

\DeclareMathOperator{\Crit}{Crit}
\DeclareMathOperator{\GL}{GL}
\DeclareMathOperator{\Ran}{Ran}
\DeclareMathOperator{\dist}{dist}
\DeclareMathOperator{\Sp}{Sp}

\renewcommand{\Im}{\operatorname{Im}}
\renewcommand{\Re}{\operatorname{Re}}

\begin{document}

\title[Poincar\'e series for analytic convex bodies]{Poincar\'e series for analytic convex bodies}

\date{\today}

\author[N.V.~Dang]{Nguyen Viet Dang}
\address{IRMA, Universit\'e de Strasbourg, 7 rue Ren\'e
Descartes, 67084 Strasbourg Cedex, France}
\address{Institut Universitaire de France, Paris, France}
\email{nvdang@unistra.fr}

\author[Y.~Guedes Bonthonneau]{Yannick Guedes Bonthonneau}
\address{D\'epartement de math\'ematiques et applications, \'Ecole normale supérieure, CNRS, 45 rue d’Ulm, 75230 Paris Cedex 05, France} 
\email{yguedesbonthonne@dma.ens.fr}

\author[M.~L\'eautaud]{Matthieu L\'eautaud}
\address{Laboratoire de Math\'ematiques d’Orsay, Universit\'e Paris-Saclay, B\^atiment 307, 91405 Orsay Cedex France} 
\address{Institut Universitaire de France, Paris, France}
\email{matthieu.leautaud@universite-paris-saclay.fr}

\author[G.~Rivi\`ere]{Gabriel Rivi\`ere}
\address{Laboratoire de Math\'ematiques Jean Leray, Universit\'e de Nantes, UMR CNRS 6629, 2 rue de la Houssini\`ere, 44322 Nantes Cedex 03, France}
\address{Institut Universitaire de France, Paris, France}
\email{gabriel.riviere@univ-nantes.fr}

\begin{abstract} 

We study Poincar\'e series associated to strictly convex bodies in the Euclidean space. 
These series are Laplace transforms of the distribution of lengths (measured with the Finsler metric associated to one of the bodies) from one convex body to a lattice. 
Assuming that the convex bodies have analytic boundaries, we prove that the Poincar\'e series, originally defined in the right complex half-plane, continues holomorphically to a conical neighborhood of this set, removing a countable set of cuts and points. The latter correspond to the spectrum of a dual elliptic operator. 
We describe singularities of the Poincar\'e series at each of these branching points.
One of the steps of the proof consists in showing analytic continuation of the resolvent of multiplication operators by a real-valued analytic Morse function on the sphere as a branched holomorphic function, a result of independent interest. \\



\end{abstract}

\maketitle
\tableofcontents

\section{Introduction} 

\subsection{Poincar\'e series: setting and main results}
We consider a smooth strictly convex and compact body $K$ of $\IR^d$ ($d\geq 2$), containing $0$ in its interior ($0 \in \mathring{K}$). All along the article, we use the following  terminology.
 \begin{definition}
 \label{e:def-convex-body}
 We say that the convex compact set $K$ of $\IR^d$ ($d\geq 2$) is 
\begin{itemize}
\item smooth (resp. analytic) if $\partial K$ is a $\mathcal{C}^\infty$ (resp. real-analytic) submanifold of $\R^d$, 
\item strictly convex if $\partial K$ has all its sectional curvatures positive. 
\end{itemize}
\end{definition}
A smooth strictly convex and compact body $K$ defines a Finsler (or Minkowski) pseudo-distance on $\R^d$:
\[
\text{for } x,y\in\IR^d,\quad d_K(x,y):=\inf\{t>0:\ y\in x+tK\}.
\]
If in addition $K=-K$, then the function $d_K$ becomes a genuine distance on $\R^d$.
Given any other compact convex set $K_0\subset \R^d$, a classical topic in harmonic analysis consists in counting (in the large $T$--asymptotics) those lattice points in $2\pi \Z^d \subset \R^d$ that are at $d_K$ distance less than $T$ of the set\footnote{Strictly speaking, these references deal with the case where $K_0$ is reduced to a point but the question still makes sense for more general convex sets and the leading asymptotic remains unchanged.} $K_0$~\cite{vanderCorput20, Herz62b, Randol66, ColindeVerdiere77, Randol84}.
A first standard result in the field provides the rough estimate
\begin{equation}\label{e:countingfunction}
\sharp \left\{x\in 2\pi\IZ^d: d_K(K_0,x)\leq T\right\}\sim \frac{\Vol(K)}{(2\pi)^d}T^d,\quad\text{as}\quad T\rightarrow+\infty.
\end{equation}
This counting problem is naturally associated with the length distribution 
\begin{align}
\label{e:def-TK0}
T_{K_0}(t):=\sum_{x\in 2\pi\IZ^d\setminus K_0}\delta\left(t-d_K(K_0,x)\right) ,
\end{align}
where $\delta$ denotes the Dirac distribution at $0$ in $\R$.
Thanks to~\eqref{e:countingfunction}, one can verify that $T_{K_0}(t)$ defines indeed a tempered distribution in $\mathcal{S}'(\IR)$, supported in $\R_+$. Its Laplace transform
\begin{align}
\label{e:def-PK0}
\mathcal{P}_{K_0}(s):=\sum_{x\in 2\pi\IZ^d\setminus K_0}e^{-sd_K(K_0,x)} , \quad \text{for } s\in \C \text{ with }\Re(s)>0,
\end{align}
 is a so-called \emph{Poincar\'e series}, and it is holomorphic on its domain of definition. This article is devoted to the study of analytic continuation of the Poincaré series $\mathcal{P}_{K_0}(s)$ beyond $\{\Re(s)>0\}$. 
Note that, even if not emphasized by the notation, the distribution $T_{K_0}$ and the function $\mathcal{P}_{K_0}$ do depend on the convex set $K$ used to define the Finsler distance $d_K$.

In order to state our main results, we introduce the support function $h_K$ of the convex set $K$, defined by
\begin{align}
\label{e:defh_K}
h_K(\xi) :=  \sup\{ \xi\cdot x , x \in  K  \} , \quad \xi\in \R^{d}  , 
\end{align}
and we recall that it is homogeneous of degree one (and actually defined on the dual set $(\R^d)^*$ which we identify with $\R^d$ through the Euclidean inner product $(\xi,x)\mapsto \xi\cdot x$). Note also that $-K=\{-x,x\in K\}$ is convex if $K$ is, and we have $h_{-K}(\xi)=h_K(-\xi)$ for all $\xi\in \R^d$, as a consequence of~\eqref{e:defh_K}.
The first main result of this article is the following.
\begin{theorem}
\label{t:maintheo1} 
Let $K\subset \R^d$ ($d\geq 2$) be an analytic strictly convex compact body in the sense of Definition~\ref{e:def-convex-body}, such that $0 \in\mathring{K}$. Suppose that $K_0$ is either a point or an analytic strictly convex compact body.
Then, the following three statements hold.
\begin{enumerate}
\item \label{i:cas-general-thm} There exists $\delta_0>0$ such that $\mathcal{P}_{K_0}$ extends meromorphically to 
\[
 \left\{ \Re(w)> -\delta_0(1+|\Im(w)|) \right\}\setminus \left(\Lambda_K+\R_-\right),
\]
where 
\begin{align}
\label{e:def-lambdak}
\Lambda_K:= \left\{  i h_{ K}(\xi):\ \xi\in\IZ^d\setminus \{0\}\right\}\cup \left\{  -i h_{ -K}(\xi):\ \xi\in\IZ^d\setminus \{0\}\right\}.
\end{align}
\item \label{i:symmetric-case-main-thm} If in addition $K$ is symmetric about zero, then $\Lambda_K= \left\{ \pm i h_{ K}(\xi):\ \xi\in\IZ^d\setminus \{0\}\right\}$ and there is $\delta_0>0$ such that for all $\delta\in (0, \delta_0]$, $\mathcal{P}_{K_0}$ extends meromorphically to the connected set $$
\C \setminus  \bigg( \bigcup_{\xi \in \Z^d\setminus\{0\}} h_K(\xi) \mathscr{C}\bigg), \quad \text{with}\quad \mathscr{C}= [i , i-\delta] \cup [ i-\delta ,-i  -\delta ] \cup  [-i  -\delta , -i] .$$  
\item \label{i:pole-at-zero} The extended function $\mathcal{P}_{K_0}$ admits a single pole located at $s=0$ and the function 
\[
\mathcal{P}_{K_0}(s)-\frac{1}{(2\pi)^d}\sum_{\ell=1}^d\frac{\ell!V_{d-\ell}(K_0,K)}{s^\ell}
\]
is holomorphic near $s=0$ with $V_{d-\ell}(K_0,K)$ defined by 
\begin{align}
\label{e:mixed-volumes}
\text{ for } t>0,\quad \Vol_{\IR^d}\left(K_0+tK\right)=\sum_{\ell=0}^dt^\ell V_{d-\ell}(K_0,K).
\end{align}
\end{enumerate}
\end{theorem}

We refer to Figures~\ref{f:cuts} and~\ref{f:agrafes-triangles} for an illustration of Items~\eqref{i:cas-general-thm} and~\eqref{i:symmetric-case-main-thm} of Theorem~\ref{t:maintheo1} respectively.\\

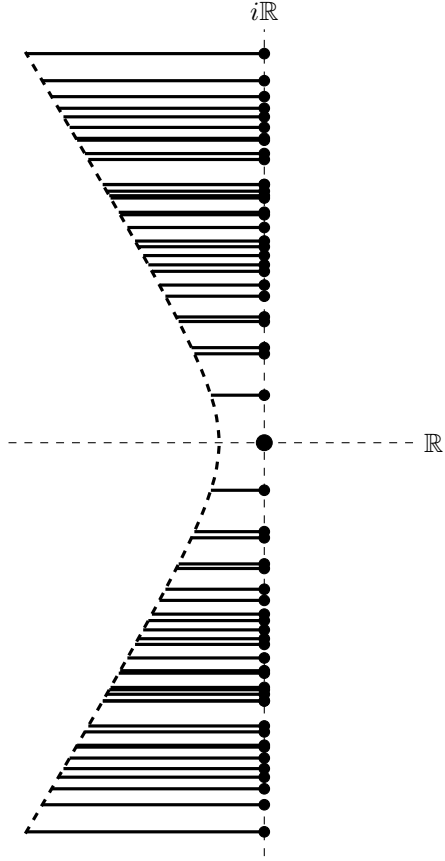
\begin{figure}[!h]
			\centering
			\newcounter{cx3}
			\begin{tikzpicture}		
 			\pgfmathsetmacro{\Si}{1} 
 			\pgfmathsetmacro{\nun}{0.6} 
 			\pgfmathsetmacro{\lambd}{4*\Si} 
 			\pgfmathsetmacro{\asy}{1} 
			\pgfmathsetmacro{\trans}{0.6} 
			
				\draw[dashed] (-\nun*\lambd-1,0) -- (2,0) node[anchor=west]{$\R$};
				\draw[dashed] (0,-\lambd-1.5)--(0,\lambd+1.5) node[above]{$i\R$};
           		  	\draw[very thick,dashed,domain=-5.2*\Si:5.2*\Si,smooth,variable=\t,-] plot ({-{sqrt(\nun*\nun*(\t*\t+1))},\t});

				\filldraw[black] (0,0) circle (3pt) ;  

			\foreach \j in {0,1,...,\lambd} 
			\foreach \k in {0,1,...,\lambd} 	
			{ 
			\setcounter{cx3}{\number\k*\number\k*\asy+\number\j*\number\j} 
				\ifnum\value{cx3}>0
				\ifnum\value{cx3}<\lambd*\lambd+1
				\pgfmathsetmacro{\Aa}{sqrt(\k*(\k-\trans)*\asy+\j*\j)*\Si} 
				\pgfmathsetmacro{\Bb}{sqrt((-\k)*(-\k-\trans)*\asy+\j*\j)*\Si} 

				\filldraw[black] (0,\Aa) circle (2pt) ;  
				\filldraw[black] (0,-\Aa) circle (2pt) ; 
				\draw[very thick]  (0,\Aa) -- (-{sqrt(\nun*\nun*(\Aa*\Aa+1))},\Aa); 
				\draw[very thick]  (0,-\Aa) -- (-{sqrt(\nun*\nun*(\Aa*\Aa+1))},-\Aa); 
				\filldraw[black] (0,\Bb) circle (2pt) ;  
				\filldraw[black] (0,-\Bb) circle (2pt) ; 
				\draw[very thick]  (0,\Bb) -- (-{sqrt(\nun*\nun*(\Bb*\Bb+1))},\Bb); 
				\draw[very thick]  (0,-\Bb) -- (-{sqrt(\nun*\nun*(\Bb*\Bb+1))},-\Bb); 
				\fi
				\fi
				}
 			\end{tikzpicture}
			\caption{Illustration of Item~\eqref{i:cas-general-thm} of Theorem~\ref{t:maintheo1}. The dots (except zero) on the imaginary axis correspond to the set $\Lambda_K$ (defined in~\eqref{e:def-lambdak}) and the cuts are given by $\Lambda_K+\R_-$.
			The (single) pole at zero is described by Item~\eqref{i:pole-at-zero} of Theorem~\ref{t:maintheo1}, whereas the singularity of $\mathcal{P}_{K_0}$ at each point of $\Lambda_K$ is described in Theorem~\ref{t:maintheo2}.}
			\label{f:cuts}
		\end{figure}

	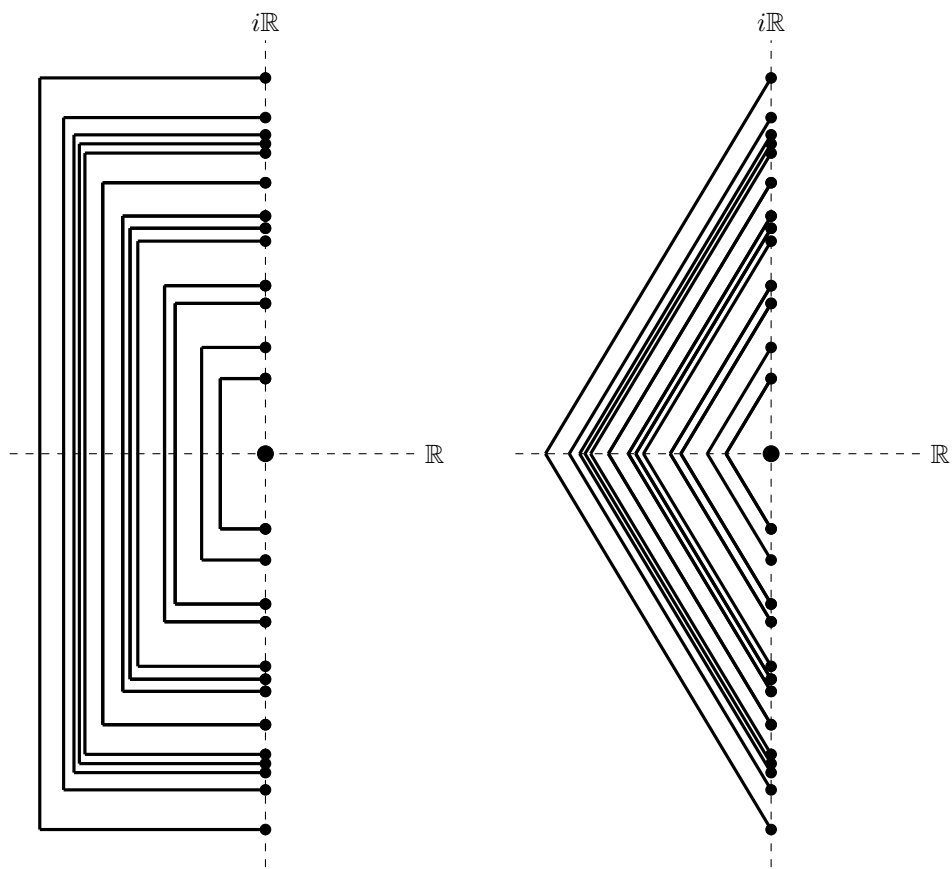
\begin{figure}[!h]
			\centering
			\newcounter{cx}
			\begin{tikzpicture}		
 			\pgfmathsetmacro{\Si}{1} 
 			\pgfmathsetmacro{\nun}{0.6} 
 			\pgfmathsetmacro{\lambd}{4*\Si} 

				\draw[dashed] (-\nun*\lambd-1,0) -- (2,0) node[anchor=west]{$\R$};
				\draw[dashed] (0,-\lambd-1.5)--(0,\lambd+1.5) node[above]{$i\R$};
				\filldraw[black] (0,0) circle (3pt) ;  

			\foreach \j in {0,1,...,\lambd} 
			\foreach \k in {0,1,...,\lambd} 	
			{ 
			\setcounter{cx}{\number\k*\number\k+\number\j*\number\j} 
				\ifnum\value{cx}<\lambd*\lambd
				\filldraw[black] (0,{sqrt(\k*\k+\j*\j)*\Si}) circle (2pt) ; 
				\filldraw[black] (0,-{sqrt(\k*\k+\j*\j)*\Si}) circle (2pt); 
				\draw[very thick] (-{sqrt(\k*\k+\j*\j)*\Si*\nun},{sqrt(\k*\k+\j*\j)*\Si}) -- (0,{sqrt(\k*\k+\j*\j)*\Si}) ; 
				\draw[very thick]  (-{sqrt(\k*\k+\j*\j)*\Si*\nun},-{sqrt(\k*\k+\j*\j)*\Si}) -- (-{sqrt(\k*\k+\j*\j)*\Si*\nun},{sqrt(\k*\k+\j*\j)*\Si})  ;
				\draw[very thick]  (0,-{sqrt(\k*\k+\j*\j)*\Si}) -- (-{sqrt(\k*\k+\j*\j)*\Si*\nun},-{sqrt(\k*\k+\j*\j)*\Si})  ;
				\fi
				}	

 			\end{tikzpicture}
\quad \quad 
			\newcounter{cx2}
			\begin{tikzpicture}		
 			\pgfmathsetmacro{\Si}{1} 
 			\pgfmathsetmacro{\nun}{0.6} 
 			\pgfmathsetmacro{\lambd}{4*\Si} 

				\draw[dashed] (-\nun*\lambd-1,0) -- (2,0) node[anchor=west]{$\R$};
				\draw[dashed] (0,-\lambd-1.5)--(0,\lambd+1.5) node[above]{$i\R$};
				\filldraw[black] (0,0) circle (3pt) ;  

			\foreach \j in {0,1,...,\lambd} 
			\foreach \k in {0,1,...,\lambd} 	
			{ 
			\setcounter{cx2}{\number\k*\number\k+\number\j*\number\j} 
				\ifnum\value{cx2}<\lambd*\lambd
				\filldraw[black] (0,{sqrt(\k*\k+\j*\j)*\Si}) circle (2pt) ; 
				\filldraw[black] (0,-{sqrt(\k*\k+\j*\j)*\Si}) circle (2pt); 
				\draw[very thick] (-{sqrt(\k*\k+\j*\j)*\Si*\nun},0) -- (0,{sqrt(\k*\k+\j*\j)*\Si}) ; 
				\draw[very thick]  (0,-{sqrt(\k*\k+\j*\j)*\Si}) -- (-{sqrt(\k*\k+\j*\j)*\Si*\nun},0)  ;
				\fi
				}	
 			\end{tikzpicture}
			\caption{Illustration of Item~\eqref{i:symmetric-case-main-thm} of Theorem~\ref{t:maintheo1} in case $K=-K$, with two different choices of cut shapes $\mathscr{C}$. The dots (except zero) on the imaginary axis correspond to the set $\Lambda_K$ and each curve is $h_K(\xi) \mathscr{C}$ for some $\xi \in \Z^d\setminus\{0\}$. The function $\mathcal{P}_{K_0}$ extends meromorphically to $\C \setminus  \bigg( \bigcup_{\xi \in \Z^d\setminus\{0\}} h_K(\xi) \mathscr{C}\bigg)$. 
			The (single) pole at zero is described by Item~\eqref{i:pole-at-zero} of Theorem~\ref{t:maintheo1}, whereas the singularity of $\mathcal{P}_{K_0}$ at each point of $\Lambda_K$ is described in Theorem~\ref{t:maintheo2}.}
			\label{f:agrafes-triangles}
		\end{figure}

Several remarks are in order. 
That $\Vol_{\IR^d}\left(K_0+tK\right)$ is a polynomial of degree $d$, namely~\eqref{e:mixed-volumes}, is a classical fact in convex geometry~\cite[Chapter~5.1]{Schneider14} (see also~\cite[Remark~4.5]{DangLeautaudRiviere22}), and its coefficients $V_{d-\ell}(K_0,K)$ are called the mixed volumes of $K_0,K$. In Item~\ref{i:pole-at-zero} of Theorem~\ref{t:maintheo1}, they appear as residues of the Poincar\'e series $\mathcal{P}_{K_0}(s)$ at zero. It is also worth noticing that the singularities of the Poincar\'e series appear (up to multiplication by $\pm i$) at the eigenvalues of the elliptic operators $h_{\pm K}(D_x)$ (where $D_x = -i\partial_x$), viewed as a translation invariant selfadjoint pseudodifferential operator of order $1$ on the torus $\mathbb{T}^d=\mathbb{R}^d/2\pi\mathbb{Z}^d$. This is a manifestation of quantum--classical correspondence. We refer to Section~\ref{ss:classical-quantum} below for a more detailed discussion of this fact in relation with trace formulas and duality for convex sets. The shape of the ``cuts'' (namely $\Lambda_K+\R_-$ in the general statement of Item~\ref{i:cas-general-thm} and the union of the $h_K(\xi) \mathscr{C}$ in the symmetric case of Item~\ref{i:symmetric-case-main-thm}) is somewhat arbitrary.
We refer to Theorem~\ref{t:continuation-zeta} and Lemma~\ref{l:geometry-of-cuts} for more general choices of the cuts.

 In~\cite[Th.~9.6]{DangLeautaudRiviere22}, it was shown that, if $K,K_0$ are (only) smooth (i.e. $\mathcal{C}^\infty$), then the function $\mathcal{P}_{K_0}$ extends in a $\mathcal{C}^\infty$ way up to $\{\Re(s)=0\}\setminus(\{0\}\cup\Lambda_K)$. Theorem~\ref{t:maintheo1}  shows that analyticity allows to continue $\mathcal{P}_{K_0}$ beyond the imaginary axis modulo the forbidden lines $\Lambda_K+\R_-$ or the forbidden cuts $\Lambda_K \mathscr{C}$ in the symmetric case. Recall also from~\cite[\S9.5]{DangLeautaudRiviere22}  that, in the case where $K_0$ is reduced to a point $x_0$ not lying in $2\pi\Z^d$ and where $K=B_d(0,1)$ is the Euclidean unit ball in $\R^d$, an explicit computation shows\footnote{Note that the factor $4\pi^2$ in the formula~$(9.17)$ from this reference is incorrect due to a wrong factor $2\pi$ in Eq.~($6.3$) therein; this error has no other logical consequence.}
 \begin{equation}\label{e:laplacianformula}
 \mathcal{P}_{x_0}(s)=d!\text{Vol}_{\R^d}(B_d(0,1))s\left(s^2-\Delta_{\T^d}\right)^{-\frac{d+1}{2}}\left(x_0,0\right),
 \end{equation}
 where $\left(s^2-\Delta_{\T^d}\right)^{-\frac{d+1}{2}}\left(x,y\right)$ is the Schwartz kernel of the resolvent (to the power $\frac{d+1}{2}$) of the standard Euclidean Laplacian on $\T^d$. From this formula and from the spectral theory of the Laplacian, one directly obtains Theorem~\ref{t:maintheo1} when $K=B_d(0,1)$ and $K_0=\{x_0\}$: the conclusion is even stronger in that case. For more general $K$ and $K_0$, such a formula does not hold true a priori and, as in~\cite{DangLeautaudRiviere22}, we will need to proceed in a different way to prove the holomorphic continuation statement of Theorem~\ref{t:maintheo1}.

An analysis of the singularities occurring at points of $\Lambda_K$ allows for more precise extension properties of the function $\mathcal{P}_{K_0}$, as stated by the following result.
\begin{theorem}
\label{t:maintheo2} 
Under the assumptions of Theorem~\ref{t:maintheo1}, there exists $\delta_0>0$ such that, for any $\lambda\in\Lambda_K$, the following holds. 
There exist $\delta_\lambda>0$ and {\em holomorphic} functions $f_\lambda,g_\lambda$ on 
$$
\Omega_\lambda:=\{w:\ \Re(w)> -\delta_0(1+|\Im(w)|) \ \text{and}\ |\operatorname{Im}(w-\lambda)|<\delta_\lambda\}
$$
such that
 $$
\mathcal{P}_{K_0}(s) = \frac{1}{(s-\lambda)^{\frac{d+1}{2}}}f_\lambda(s) + g_\lambda(s), \quad \text{if $d$ is even}, \quad s \in \Omega_\lambda , 
 $$
 or
 $$
 \mathcal{P}_{K_0}(s) = \frac{1}{(s-\lambda)^{\frac{d+1}{2}}}f_\lambda(s) + \ln(s-\lambda) g_\lambda(s), \quad \text{if $d$ is odd}, \quad s \in \Omega_\lambda.
 $$

\end{theorem}
Here, $\ln$ denotes the principal value of the logarithm on $\C\setminus\R_-$ and $z^\gamma$ is understood as $\exp(\gamma\ln z)$ when $\gamma$ is not an integer. The logarithmic or squareroot singularities near the points in $\Lambda_K$ were already identified in the $\mathcal{C}^\infty$ setting treated in~\cite{DangLeautaudRiviere22}. Here, the additional analyticity assumption on the convex sets, allows to continue the function $\mathcal{P}_{K_0}$ deep in the complex plane. 
The precise knowledge of the singularities allows to extend $\mathcal{P}_{K_0}$ either as a certain multivalued holomorphic function or as a holomorphic function on an appropriate Riemann surface. See Appendix~\ref{a:branched} for a reminder on these notions.

\bigskip
The study of Poincar\'e series is a classical topic in hyperbolic geometry, see~\cite{BroiseParkkonenPaulin2019}. In that setting, the critical axis is given by $\{\Re(s)= h_{\text{top}}\}$, 
where $h_{\text{top}}>0$ is the topological entropy of the underlying dynamical system, or equivalently the critical exponent of the associated (non commutative) fundamental group. Poincar\'e series are a priori holomorphic on the half plane $\{\Re(s)> h_{\text{top}}\}$ and, to the best of our knowledge, the first results about their continuation are due to Huber~\cite{Huber56, Huber59} -- see also~\cite{Delsarte42}. In these references, Huber proves the meromorphic continuation of these series beyond the critical axis. More precisely, he works with \emph{hyperbolic} metrics and the role of $K$ is played by the unit ball in the universal cover of the hyperbolic manifold. Using tools from harmonic analysis on Lie groups and the connection with the spectral theory of the Laplacian in the spirit of formula~\eqref{e:laplacianformula}, the \emph{meromorphic} extension of Poincar\'e series can then be obtained provided that $K_0$ satisfies nice enough convexity properties (\emph{e.g.} if $K_0$ is reduced to a point). See for instance~\cite{Guillope94}. This result was recently extended to manifolds of \emph{variable} negative curvature by two of the authors~\cite{DangRiviere20d} using the spectral analysis of Anosov dynamical systems. See also~\cite{Chaubet22, ChaubetGuedesBonthonneau24, BenardChaubetDangSchick23, DangMehmeti24, Anantharaman24} for recent developments in other geometric settings.

In the present article, the geometric setting is related to integrable dynamical systems on the unit tangent bundle $S\mathbb{T}^d=\mathbb{T}^d\times\IS^{d-1}$ of the torus for which the topological entropy is $0$. Theorems~\ref{t:maintheo1} and~\ref{t:maintheo2} show that the behaviour on (and beyond) the critical axis  $\{\Re(s)=0\}$ becomes more complicated than in the hyperbolic setting. Indeed, rather than meromorphic continuation to the complex plane, we obtain a branched continuation of the corresponding Poincar\'e series with singularities located on an infinite discrete subset of $\{\Re(s)=0\}$. As in~\cite{DangLeautaudRiviere22}, the proofs of these results will be related to the spectral analysis of the underlying dynamical system as we shall discuss in Section~\ref{s:sketch} below. Finally, we remark that stochastic perturbations of the geodesic flow on $\T^2\times \IS^1$ were considered by Dyatlov and Zworski in~\cite[\S1]{DyatlovZworski15} and that the discrete $L^2$-spectrum of the associated stochastically perturbed  operator seems to accumulate on curves similar to the ones appearing in Figure~\ref{f:agrafes-triangles} (see Figure 3 in that reference). It would be interesting to understand why (or if) a particular choice of viscosity term in these stochastic results selects a specific family of cut-shapes in the spectral analysis of the present paper.

\subsection{Spectral theory of multiplication operators}
\label{s:multiplicatop}
As will appear in Section~\ref{s:sketch} below, the proofs of Theorems~\ref{t:maintheo1} and~\ref{t:maintheo2} rely on a fine understanding of the spectral theory of certain (families of) multiplication operators on $L^2(\IS^{d-1})$.
We present here one of the results needed in the proofs of Theorems~\ref{t:maintheo1} and~\ref{t:maintheo2}, which, we believe, is of independent interest.

A continuous function $\f:\IS^{d-1}\to \R$ naturally defines a bounded selfadjoint multiplication operator on $L^2(\IS^{d-1})$ by 
\begin{align}
\label{e:multiplication-f}
\mathbf{m}_\f:u\in L^2(\IS^{d-1})\mapsto \f\times u\in L^2(\IS^{d-1}), 
\end{align}
satisfying
$$
 \mathbf{m}_\f \in \mathcal{L}\big(L^2(\IS^{d-1}) \big) , \quad  \| \mathbf{m}_\f \|_{ \mathcal{L}\big(L^2(\IS^{d-1}) \big) } = \|\f\|_{L^\infty(\IS^{d-1})}\quad \mathbf{m}_\f^*=\mathbf{m}_\f,
$$ 
where the $L^2$ space is considered with respect to the canonical volume measure on $\mathbb{S}^{d-1}$. The analysis of multiplication operators is a cornerstone topic in spectral theory since, according to the multiplicative spectral theorem~\cite[Theorem VIII.4 p. 260]{RS80}, every (resp. bounded) selfadjoint operator is unitarily equivalent to a  multiplication operator (resp. by a bounded function) with respect to the spectral measure.

The spectrum of the operator $\mathbf{m}_\f$ is given by $\Sp(\mathbf{m}_\f) = [\min\f,\max\f]$. We prove in Theorem~\ref{t:maintheo-multiplication} that (assuming in addition that $\f$ is real-analytic, with only two critical points that are nondegenerate) changing the spaces of definition of the operator, one can continue the resolvent $(z-\mathbf{m}_\f)^{-1}$, originally defined (and holomorphic) on $\C \setminus [\min\f,\max\f]$, across the continuous spectrum, away from the critical values of $\f$. We also describe the singularities of the extended resolvent $(z-\mathbf{m}_\f)^{-1}$ near the critical values $\min\f$ and $\max\f$.

\begin{remark}[Conventions for the logarithm and square-root]
\label{r:log-definition}
All along the article, $\ln$ denotes the {\em principal } value of the logarithm on $\C\setminus\R_-$. Then, the function $$\ln_\theta(z) :=  \ln (e^{i\theta} z)-i\theta$$ defines the determination on $\C\setminus e^{-i\theta}\R_- = \C\setminus e^{i(\pi-\theta)}\R_+$ which is real on on $\R_+^*$. Accordingly, $\sqrt{\cdot}$ denotes the {\em principal } value of the square root on $\C\setminus\R_-$, defined by $\sqrt{z} = \exp\left(\frac12 \ln(z)\right)$ and we denote by
$$_\theta\sqrt{z}  :=  \exp\left(\frac12 \ln_\theta(z)\right) = e^{-i\frac{\theta}{2}}\sqrt{e^{i\theta}z}
$$ the determination  of the square root on $\C\setminus e^{-i\theta}\R_-$ which is real-valued on the real line.
Notice that $\ln_{\theta}$ and $\ln$ coincide on $\R_+^*$, and hence $\ln_{\theta}(z)=\ln(z)$ for all $z$ in the connected component of $\C\setminus(\R_-\cup e^{-i\theta}\R_-)$ containing $\R_+^*$. Consequently, we also have $_\theta\sqrt{z}  =\sqrt{z}$ on this set.
 \end{remark}

In the next statement, we denote by $\mathcal{A}_R(\IS^{d-1})$ a Banach space of real-analytic functions introduced formally in~\S\ref{s:analytic-norms} (see Definition~\ref{d:def-norms}), and by $\mathcal{A}_R(\IS^{d-1})'$ its dual space of ultradistributions. The parameter $R>0$ plays the role of a radius of analyticity.
\begin{theorem}
\label{t:maintheo-multiplication} Let $0<R<\frac{1}{2\sqrt{d-1}}$ and let $\f:\mathbb{S}^{d-1}\rightarrow \R$ be an analytic Morse function in $\mathcal{A}_R(\IS^{d-1})$ having only two critical points. 
Then, there exists an open neighborhood $\mathcal{U}$ of $[\min\f,\max\f]$ in $\C$ such that the holomorphic function
$$
(z-\mathbf{m}_\f)^{-1}:\mathcal{A}_R(\IS^{d-1})\mapsto  \mathcal{A}_R(\IS^{d-1})',\quad z\in\mathcal{U}\cap\{\Im(s)>0\}
$$
extends holomorphically to 
$$
\widetilde{\mathcal{U}}:=\mathcal{U}\setminus \left((\min\f+i\R_-)\cup(\max\f+i\R_-)\right).
$$
Moreover, for all $z \in \widetilde{\mathcal{U}}$ we have
$$
\left(z-\mathbf{m}_\f\right)^{-1}=\mathcal{H}(z)\mathrm{F}_{d}(z)+\mathcal{R}(z),
$$
where $\mathcal{H}(z),\mathcal{R}(z):\mathcal{A}_R(\IS^{d-1})\mapsto  \mathcal{A}_R(\IS^{d-1})'$ are holomorphic functions on $\mathcal{U}$ and where, using the notation of Remark~\ref{r:log-definition},
  \begin{align*}
  \mathrm{F}_{d}(z) & = \big(z -\min(\f)\big)^{\frac{d-3}{2}} \big(z-\max(\f)\big)^{\frac{d-3}{2}} \big(  \ln_{-\frac{\pi}{2}}(z-\min(\f)) - \ln_{-\frac{\pi}{2}}(z-\max(\f)) \big)   , \\
  & \quad \text{ for }d\geq 3 \text{ odd}, \\
  \mathrm{F}_{d}(z)&  =  \left(z -\min(\f)\right)^{\frac{d-3}{2}} \left(z-\max(\f)\right)^{\frac{d-3}{2}}    , \quad \text{ for }d\geq 2 \text{ even} ,
  \end{align*}
  where in the last line we have written $\left(\zeta\right)^{\frac{d-3}{2}} = {}_{-\frac{\pi}{2}}\sqrt{\zeta}^{d-3}$.
\end{theorem}

Compared to results on multiplication operators~\cite{ Georgescu,MantoiuPascu96,JezequelWang2025} or more general pseudodifferential operators of order $0$~\cite{GalkowskiZworski19a,GalkowskiZworski19}, we gain a precise understanding of the branched analytic structure near singularities thanks to the simple structure of our operator. 
The restriction $0<R<\frac{1}{2\sqrt{d-1}}$  in the space $\mathcal{A}_R(\IS^{d-1})$ is here only to ensure that the spaces are well-defined and nonempty. 
Note that the fact that $\f$ has only two critical points actually implies that the underlying manifold is diffeomorphic to the sphere.
Equivalently, such results can be formulated using the terminology of multivalued holomorphic functions described in Appendix~\ref{a:branched}. With this point of view, we can say that $(z-\mathbf{m}_\f)^{-1}$ has a multivalued holomorphic extension to $\mathcal{U}\setminus\{\min\f,\max\f\}$ that belongs to the Nilsson class.

\begin{remark} Note that in~\cite{DangLeautaudRiviere22}, under the assumption that $\f \in \mathcal{C}^\infty(\IS^{d-1};\R)$ 
a limiting absorption principle was proved through a method closer to the so-called Mourre's commutator method~\cite{Mourre81, Georgescu, MantoiuPascu96} (even if not explicitly stated like this). More precisely, this method shows that, if $\f$ is a smooth Morse function with only two critical points, then the holomorphic function
$$
(z-\mathbf{m}_\f)^{-1}:H^{\frac{1}{2}+\epsilon}(\IS^{d-1})\mapsto  H^{-\frac{1}{2}-\epsilon}(\IS^{d-1}),\quad z\in\{\Im(s)>0\}
$$
extends in a continuous way to the interval $(\min\f,\max\f)$. Here, analyticity of $\f$ allows to continue $(z-\mathbf{m}_\f)^{-1}$ beyond the real axis.
\end{remark}

\subsection{Remarks on Laplacians on noncompact domains and trace formulae}
\subsubsection{Related results in the spectral analysis of Laplacians on noncompact domains} 

The results of the present article are also related to the spectral analysis of Laplace type operators on noncompact domains. A typical situation is the following, see e.g.~\cite{DyatlovZworski19}. Let $\mathcal{O}$ be a bounded open set with smooth boundary and denote by $\mathcal{U}=\mathbb{R}^d\setminus\mathcal{O}$. One can then consider the Laplacian $-\Delta:H^2(\mathcal{U})\cap H^1_0(\mathcal{U})\subset L^2(\mathcal{U})\rightarrow L^2(\mathcal{U})$ which an unbounded selfadjoint operator with spectrum equal to $\R_+$. In particular, the resolvent
$$
(-\Delta-\lambda^2)^{-1}:L^2(\mathcal{U})\rightarrow L^2(\mathcal{U})
$$
is well-defined and holomorphic on $\{\Im(\lambda)>0\}$. The theory of resonances (see e.g.~\cite[Th.~4.4]{DyatlovZworski19}) shows that, when $d$ is odd, 
$$
(-\Delta-\lambda^2)^{-1}:L^2_{\text{comp}}(\mathcal{U})\rightarrow L^2_{\text{loc}}(\mathcal{U})
$$
extends meromorphically to $\mathbb{C}$ while, for $d$ even, it extends to $\mathbb{C}\setminus i\R_-$ (and more precisely to the logarithmic cover of $\C$). The study of this resolvent is related to the study of the decay of the local energy for the wave equation. In odd dimension, under nontrapping (or hyperbolic trapping) conditions, one can for instance prove the existence of a spectral gap and of exponential decay for the local energy. In even dimension, one cannot expect better than polynomial decay due to the presence of the logarithmic singularity at $0$. Similarly, the squareroot/logarithmic singularities that one observes for the vector field $V$ defined in~\eqref{e:def-V} can be thought of as a reminiscence of the weak polynomial mixing properties of such vector fields -- see e.g.~\cite[\S6]{DangLeautaudRiviere22} and~\cite{DangLeautaudRiviereJEDP}. The fact that the Poincar\'e series $\mathcal{P}_{K_0}$ have infinitely many singularities can also be observed in the context of scattering resonances for the Laplacian if one considers waveguide situations, see e.g.~\cite{ChristiansenDatchev} and the references therein. More precisely, if one replaces $\mathcal{U}$ by $\mathcal{U}\times M$ where $M$ is a compact manifold (say the circle), then one ends up with infinitely scattering problems after using the spectral decomposition of the Laplacian (along the variables in $M$). This leads to infinitely many logarithmic singularities when $d$ is even. This phenomenon is similar to what we obtain here. 
For more details, related results and references on scattering theory, we refer to the survey~\cite{Zworski17} or to the recent  book~\cite{DyatlovZworski19}.

\subsubsection{Duality in convex geometry and trace formulae}\label{ss:classical-quantum}
As explained in~\cite[\S3.2.1]{DangLeautaudRiviere22}, the Hamiltonian vector field $X_{h_{\pm K}}|_{\{h_{\pm K}=1\}}$ induced by the pseudodifferential elliptic operator $h_{\pm K}(D_x)$ is smoothly conjugated to a contact vector field on the unit tangent bundle of $\mathbb{T}^d$. The analysis in this reference shows that the Poincar\'e series  $\mathcal{P}_{K_0}$ can in fact be expressed in terms of the resolvent of this vector field (see e.g.~\eqref{e:key-integration-formula2} below) whose singularities turn out to be in correspondance with the spectrum of $h_{\pm K}(D_x)$. This can be understood as an instance, in the case of certain completely integrable systems, of the classical-quantum correspondance obtained for contact Anosov flows in~\cite{DyatlovFaureGuillarmou2015, FaureTsujii17, FaureTsujii17b, FaureTsujii21}.

We first discuss the case where $K_0=\{0\}$ in which  the duality results of~\cite{DangLeautaudRiviere22}  have a simpler and clearer interpretation. For $x\in\mathbb{R}^d$, one has~\cite[Lemma~3.3]{DangLeautaudRiviere22}
\[
d_K(0,x)=\inf\left\{t>0:\ \frac{x}{t}\in K\right\}=\inf\left\{t>0:\ h_{K^\circ}\left(\frac{x}{t} \right)\leq 1\right\}=h_{K^\circ}(x),
\]
where $K^\circ :=\left\{\xi\in\mathbb{R}^d:\ \xi\cdot x\leq 1\ \text{for all } x\in K\right\}$ is the convex set dual to $K$ and $h_{K^\circ}$ denotes the support function of the convex set $K^\circ$, defined in~\eqref{e:defh_K} (with $K$ replaced by $K^\circ$). In particular, the distribution
\[
T_{0}(t) :=\sum_{\lambda\in \Sp(h_{K^\circ}(D_\xi))}\delta(t-\lambda), \quad D_\xi = -i\partial_\xi ,
\]
is an element of $\mathcal{S}'(\R)$ thanks to~\eqref{e:countingfunction}. Hence, one can consider its Fourier transform
\begin{equation}
\label{eq:Fourier-Trace}
\widehat{T}_0(\tau):=\sum_{\lambda\in \Sp(h_{K^\circ}(D_\xi))}e^{-i\tau\lambda}=\mathcal{P}_{\{0\}}(i\tau)\ \in\ \mathcal{S}'(\R).
\end{equation}
Theorem~$9.6$ from~\cite{DangLeautaudRiviere22} shows that $\widehat{T}_0(\tau)$ is smooth on $\R \setminus\big(-i\Lambda_K\cup\{0\}\big)$ when $\partial K$ is smooth while Theorem~\ref{t:maintheo1} proves that it is real analytic on the same subset as soon as $\partial K$ is analytic. It is also interesting to formulate these results in terms of the propagator $e^{-i\tau h_{K^\circ}(D_\xi)}$ of the Schr\"odinger equation (or rather the half-wave equation since the symbol $x \mapsto h_{K^\circ}(x)$ is homogeneous of degree one) 
\[
i\partial_\tau u=h_{K^\circ}(D_\xi) u,\quad (\tau,\xi)\in\R\times\R^d/\IZ^d, \quad D_\xi = -i\partial_\xi .
\]
Notice that $\widehat{T}_0(\tau)$ is the distributional trace of the group $(e^{-i\tau h_{K^\circ}(D_\xi)})_{\tau \in \R}$. Thus~\cite[Theorem~9.6]{DangLeautaudRiviere22} yields
\begin{align}
\label{e:formule-des-traces}
\text{singular support}\left(\text{Tr}\left(e^{-i\tau h_{K^\circ}(D_\xi)}\right)\right)\subset \Sp(h_{K}(D_x))\cup \Sp(-h_{-K}(D_x)),
\end{align}
and in addition describes precisely the distribution $\widehat T_0$ near its singularities in terms of standard distributions $F(\tau-\lambda\pm i0)$ for some explicit functions $F$. Note that even when $K$ is symmetric, contrary to the case of the round ball, there is no explicit relation between $\mathcal{P}_{x_0}(s)$ and the kernel of $(s^2+h_K(D_x)^2)^{\tfrac{d+1}{2}}$, even though their singular locus is the same. Here $h_{K}(D_x)$ is considered as a selfadjoint elliptic translation-invariant operator on $L^2(\T^d)$ and  
\[
\Sp(h_{K}(D_x)) = \{ h_{K}(\xi) , \xi \in \Z^d \} .
\]
Recalling that $K^{\circ\circ}=K$, one obtains a converse statement for $\text{Tr}\left(e^{-i\tau h_{K}(D_x)}\right)$. Hence, the singularities of the distributional trace $\text{Tr}\left(e^{-i\tau h_{K}(D_x)}\right)$ defined from $K$ are determined from the spectrum of the Schr\"odinger operator defined from its dual convex set $K^\circ$ and vice versa. 
Under this form,~\eqref{e:formule-des-traces} may be understood as a way to write the trace formula~\cite{ColindeVerdiere73a, ColindeVerdiere73b, Chazarain74, DuistermaatGuillemin75} from the point of view of convex geometry. Theorem~$9.6$ in~\cite{DangLeautaudRiviere22} shows that this property on the singular support holds true more generally when one considers the more geometric distribution 
\begin{equation}
\label{eq:with-a-K_0}
\widehat{T}_{K_0}(\tau)=\mathcal{P}_{K_0}(i\tau)=\sum_{\xi\in\Z^d}e^{-i\tau\inf\{h_{K^\circ}(\xi-x):\ x\in K_0\}},
\end{equation}
associated with any smooth strictly convex subset $K_0$ for which there is no obvious spectral interpretation. This is due to the connection of these zeta functions/distributions with the spectral analysis of the operator $V$ discussed above. Let us also add that while Formula \eqref{eq:Fourier-Trace} may suggest to use a Poisson summation formula to study $\mathcal{P}_{x}(s)$, it does not seem to be good direction to attack the general $K_0$ case, considering \eqref{eq:with-a-K_0}.

\subsection{Sketch of the proofs, dynamics and spectral theory}
\label{s:sketch}
In the present section, we briefly describe the proofs of Theorems~\ref{t:maintheo1},~\ref{t:maintheo2} and their links with dynamics and spectral theory of multiplication operators, namely Theorem~\ref{t:maintheo-multiplication}.
As noticed in~\cite{DangLeautaudRiviere22}, the study of the Poincar\'e series $\mathcal{P}_{K_0}$ is related to a natural dynamical system on $\mathbb{T}^d\times\IS^{d-1}$. The latter allows to rewrite $\mathcal{P}_{K_0}$ in terms of a certain spectral resolvent of a vector field. Such a dynamical reinterpretation of a geometric counting problem dates back to the seminal works of Margulis~\cite{Margulis69, Margulis04}, who dealt with this type of counting problems on manifolds with variable negative curvature and established the connection with the ergodic properties of the measure of maximal entropy. As explained in~\cite{DangLeautaudRiviere22}, this point of view can also be adopted in the present set-up and the dynamical system involved is given by 
\begin{align*}
e^{tV}:(x,\theta)\in \mathbb{T}^d\times\IS^{d-1}\mapsto (x+t\mathbf{v}(\theta),\theta)\in \mathbb{T}^d\times\IS^{d-1},
\end{align*}
where $\theta\in\mathbb{S}^{d-1}\mapsto \mathbf{v}(\theta)\in\IS^{d-1}$ is the parametrization of $\partial K$ by its outward normal. One can show (see~\cite[\S3.2.2]{DangLeautaudRiviere22}) that $e^{tV}$ is the pullback by the appropriate scaling of impulsions to $\T^d\times\IS^{d-1}$ of the Hamiltonian flow generated by $h_K$ on the energy layer $\{h_K=1\}$. Then, following~\cite{DangRiviere20d} and setting 
\begin{align}
\label{e:def-V}
V:=\mathbf{v}(\theta)\cdot\partial_x, \quad \text{ acting on } \mathbb{T}^d\times\IS^{d-1} ,
\end{align} 
the strategy in~\cite[\S9]{DangLeautaudRiviere22} (see also~\cite[Section~3.3]{DangLeautaudRiviereJEDP}) consists first in interpreting the length distribution $T_{K_0}$ in~\eqref{e:def-TK0} as a correlation function (rather a ``correlation distribution'') between two De Rham currents (carried by $\d K_0$ and $0$ in $\T^d$ respectively). Second, the Poincar\'e series $\mathcal{P}_{K_0}$ being the Laplace transform of the Distribution  $T_{K_0}$, it can be expressed in terms of the resolvent $(V+s)^{-1}$ acting on certain anisotropic spaces of De Rham currents. The $\mathcal{C}^\infty$-extension results obtained in~\cite{DangLeautaudRiviere22} can hence be interpreted as a form of limiting absorption principle for the first order differential operator $V$ on the imaginary axis. In this article, we will follow a similar strategy except that we will take advantage of the analytic regularity to go beyond the $\mathcal{C}^\infty$ continuation on the imaginary axis (away from the set $\Lambda_K$ defined in~\eqref{e:def-lambdak}).

In order to clarify the spectral results we are aiming at, let us discuss the spectral properties of the differential operator $V$ acting on functions (rather than De Rham currents as needed for Poincar\'e series). First, one can verify that 
$$
(V+s)^{-1}:L^{2}(\T^d\times\IS^{d-1})\rightarrow L^2(\T^d\times\IS^{d-1})
$$
defines a holomorphic family of bounded operator outside $\{\Re(s)=0\}$. This follows for instance from the fact the selfadjoint operator $-iV:\text{Dom}(-iV)\subset L^2\rightarrow L^2$ has spectrum exactly equal to $\mathbb{R}$. This result holds in fact for any volume preserving vector field which has a nonperiodic orbit~\cite{Guillemin77}. A convenient feature of working on the torus is that we can decompose the resolvent $(V+s)^{-1}$ according to the Fourier modes $e^{i\xi\cdot x}$, $\xi \in \Z^d$. Thus, analyzing the properties of the resolvent $(V+s)^{-1}$ amounts to analyze the resolvent of the (family, indexed by the Fourier mode $\xi \in \Z^d$, of) multiplication operator(s) by the function $\theta\mapsto i \xi\cdot \mathbf{v}(\theta)$, namely, 
$$
\text{ for } \xi\in\Z^d,\qquad \mathbf{m}_{i\xi\cdot \mathbf{v}}:u\in L^2(\IS^{d-1})\mapsto i(\xi\cdot \mathbf{v})\times u\ \in\ L^2(\IS^{d-1}), 
$$
that is to say $\mathbf{m}_{i\xi\cdot \mathbf{v}}= i \mathbf{m}_{\xi\cdot \mathbf{v}} = i \mathbf{m}_{\f}$ with the notation~\eqref{e:multiplication-f}, and with $\f(\theta)=\xi\cdot \mathbf{v}(\theta)$.
Equivalently, one can decompose $(V+s)^{-1}$ as
$$
(V+s)^{-1}\left(\widehat{u}_\xi\right)_{\xi\in\mathbb{Z}^d}=\left((\mathbf{m}_{i\xi\cdot \mathbf{v}}+s)^{-1}\widehat{u}_\xi\right)_{\xi\in\Z^d},
$$
where we have written $u=\sum_{\xi\in\Z^d}\widehat{u}_\xi(\theta)e^{i\xi\cdot x}$ and identified $u$ and $\left(\widehat{u}_\xi\right)_{\xi\in\mathbb{Z}^d}$. 
From the strict convexity assumption on $K$ (in the sense of Definition~\ref{e:def-convex-body}), one can in addition verify that, for $\xi\neq 0$, the function $\theta \mapsto \f(\theta)=\xi\cdot \mathbf{v}(\theta)$ is a Morse function with only two critical points given by $\min \f = -\xi\cdot \mathbf{v}(\frac{\xi}{|\xi|}) = -h_{-K}(\xi)<0$ and $\max \f = \xi\cdot \mathbf{v}(\frac{\xi}{|\xi|}) = h_{K}(\xi)>0$. 
As a consequence of the discussion on multiplication operators at the beginning of Section~\ref{s:multiplicatop}, the associated resolvent
$$
(\mathbf{m}_{i\xi\cdot \mathbf{v}}+s)^{-1}: L^2(\IS^{d-1})\mapsto L^2(\IS^{d-1})
$$
defines a holomorphic function on $\mathbb{C}\setminus i[-h_{-K}(\xi),h_K(\xi)]$. Extending the Poincaré series $\mathcal{P}_{K_0}$ (starting from $\Re(s)>0$) hence reduces to extending $(V+s)^{-1}$ and thus to:
\begin{enumerate}
\item \label{i:goal-1} extend all resolvents $(\mathbf{m}_{i\xi\cdot \mathbf{v}}+s)^{-1}$ starting from the region $\{\Re(s)>0\}$ across the slit $i[-h_{-K}(\xi),h_K(\xi)]$, in appropriate spaces, and 
\item \label{i:goal-2}  be able to make sums over the parameter  $\xi \in \Z^d$ converge.
\end{enumerate}

Recalling that $h_K$ is homogeneous of degree one, the interval $i[-h_{-K}(\xi),h_K(\xi)]$ increases with $|\xi|$.
The strategy we develop consists in splitting the sum over $\xi\in \Z^d$ into 
\begin{itemize}
\item large frequencies, i.e. large values of $|\xi|$, for which on the one hand, the summation issue of Goal~\eqref{i:goal-2} is taken care of by means of careful uniform estimates in terms of $|\xi|$, but on the other hand, only continuation (hence goal~\eqref{i:goal-1}) into the {\em interior }of the interval $i[-h_{-K}(\xi),h_K(\xi)]$ (that is, away from the critical values $\pm i h_K(\pm \xi)$) is required; 
\item small frequencies, i.e. a finite number of remaining values of $\xi$, for which on the one hand no uniform estimates are needed, but on the other hand, continuation (hence goal~\eqref{i:goal-1})  across  $i[-h_{-K}(\xi),h_K(\xi)]$ up to the critical values $\pm i h_K(\pm \xi)$ has to be precisely described. This is where Theorem~\ref{t:maintheo-multiplication} is applied.
\end{itemize}

In both situations (i.e. working on high and low frequencies), to extend the resolvent $(\mathbf{m}_{i\xi\cdot \mathbf{v}}+s)^{-1}$, we rely on ideas reminiscent from microlocal complex deformation methods of Helffer and Sj\"ostrand~\cite{HelfferSjostrand86} even if we do not formulate things in the exact same manner. See also~\cite{GalkowskiZworski19,GalkowskiZworski19a,JezequelWang2025} for recent applications of this mehod for models of operators close to ours by Galkowski--Zworski and J\'ez\'equel--Wang in the context of fluid mechanics. More precisely, in the ``low frequency situation'', when proving Theorem~\ref{t:maintheo-multiplication}, we perform an explicit analysis of the operator at hand to determine the precise shape of the singularities. Similar ideas were already used when studying resonances of Schr\"odinger operators with continuous spectrum. Here, they are applied in the simple setting of multiplication operators and, after summation, allow to homorphically continue $(V+s)^{-1}$ (in appropriate spaces). This constitutes the basic brick for our analysis. In order to reach the geometric application to the Poincar\'e series $\mathcal{P}_{K_0}$, several additional ingredients are required including (i) a uniform control on the size of the neighborhood $\mathcal{U}$ in the statement of Theorem~\ref{t:maintheo-multiplication} in terms of $\xi$ when $\f=\xi \cdot \mathbf{v}$, and 
(ii) the understanding of powers $(z-\mathbf{m}_\f)^{-\ell}$ of the resolvent in order to deal with the case of De Rham currents. Then, in the ``large frequency'' situation, away from the critical values of $\f$, we use the oscillatory structure of our integrals to ensure convergence in $\xi$ trough a fixed open interval $(-N_0,N_0)$. 

Finally, we emphasize that the idea of using complex deformation methods for studying the spectral properties of integrable systems on the torus can be traced back to the works of Degond on the linear Vlasov equation~\cite{Degond86}. In that reference, the analysis also boils down to the description of the resolvent for compact perturbations of multiplication operators. One of the main differences with the setting of the present article is that, due to the geometric applications to Poincar\'e series, we need to deal with multiplication operators on the sphere. This ultimately leads to the existence of logarithmic/squareroot singularities due to the presence of critical points for $\f$. In view of dealing with convergence to equilibrium for the (linear) Vlasov equation, Degond rather worked on the full space $\T^d\times\R^d$ and ended up with studying multiplication operators on the noncompact space $\R^d$ for functions $\f$ having no critical points. The absence of such points allowed him to prove, using analytic weights, the existence of a more standard resonance spectrum as in the case of Schr\"odinger type operators~\cite{HelfferSjostrand86}.

 \subsection{Organization of the article}

In Section~\ref{s:dynamics}, we begin by a short review of the results from~\cite{DangLeautaudRiviere22} that will be used in this article. As in this reference, this allows us to rewrite Poincar\'e series through Fourier series as an infinite sum of oscillatory integrals over $\mathbb{S}^{d-1}$. 
These oscillatory integrals will later be interpreted as powers of resolvent for multiplication operators on $\mathbb{S}^{d-1}$, amenable to statements like Theorem~\ref{t:maintheo-multiplication}. In particular, we split this sum into two parts: one involving the high Fourier modes (``high frequencies'') and another one involving only a finite number of modes (``low frequencies''). After that, we formulate our two main technical results: Proposition~\ref{p:low-freq} for low frequencies and Proposition~\ref{p:highfrequency} for high frequencies. The proofs of these two results is delayed to Sections~\ref{s:low-freq} and~\ref{s:high-freq} respectively. Section~\ref{s:dynamics} ends with the proof of Theorems~\ref{t:maintheo1} and~\ref{t:maintheo2} building on these two technical results.

Section~\ref{s:analytic-norms} is devoted to the description of the analytic norms that are used all along the article, along with their calculus properties. These are more or less standard tools that we review here for the sake of completeness and self-containedness.

Section~\ref{s:low-freq} deals with the analysis of a finite number of Fourier modes and thus of Proposition~\ref{p:low-freq}. We first focus on one multiplication operator and prove a version of Theorem~\ref{t:maintheo-multiplication} with parameters (see Theorem~\ref{t:multiplicationoperators}), using complex deformation arguments.

Section~\ref{s:high-freq} deals with the analysis of the high Fourier modes using complex deformation methods in the high frequency limit by optimizing the results of Section~\ref{s:low-freq} carefully.

The article also contains two appendices. Appendix~\ref{a:branched} collects some material on the notion of multivalued holomorphic functions, allowing to reformulate Theorems~\ref{t:maintheo1}--\ref{t:maintheo2} and~\ref{t:maintheo-multiplication} as continuation of objects as multivalued holomorphic functions.
Finally, Appendix~\ref{ss:analyticMorse} provides with a quantitative version of the analytic Morse Lemma needed in Section~\ref{s:low-freq} when analyzing the behaviour near critical values.

\subsection*{Acknowledgements}

We thank Tanya~Christiansen, Kiril~Datchev, Alix~Deleporte, Malo~J\'ez\'equel and Julien~Royer for useful discussions related to various aspects of the article. NVD, ML and GR are partially supported by the Institut Universitaire de France. YGB, ML and GR are partially supported by the Agence Nationale de la Recherche through the grant ADYCT (ANR-20-CE40-0017). GR also acknowledges partial support from the Centre Henri Lebesgue (ANR-11-LABX-0020-01).

\section{Rewriting Poincar\'e series as a sum of complex integrals}
\label{s:dynamics}

 In this preliminary section, we review the constructions from~\cite[\S8]{DangLeautaudRiviere22} in order to express Poincar\'e series for convex bodies in terms of infinite sums of certain oscillatory integrals. This decomposition is presented in Lemma~\ref{l:decomp-zeta-I} below.  These oscillatory integrals will then be analyzed in the upcoming sections using methods from complex analysis. 
In the present Section~\ref{s:dynamics}, we begin by recalling a few basic properties of strictly convex subsets of $\R^d$ in \S\ref{ss:convex}. In~\S\ref{ss:dynamics}, we explain how to interpret our counting problem in a dynamical manner and to rewrite Poincar\'e series by involving the vector field $V$ in~\eqref{e:def-V}. Next, \S\ref{ss:oscillatory} is devoted to the rewriting of Poinvar\'e series in terms of a sum of oscillatory/complex integrals using this alternate representation of Poincar\'e series. 
 We conclude this section by \S\ref{s:main-tech-res}, in which we formulate our main two technical results, namely Propositions~\ref{p:low-freq} and Proposition~\ref{p:highfrequency}, and apply them to provide a proof of Theorems~\ref{t:maintheo1} and~\ref{t:maintheo2}.

\subsection{Parametrization of convex sets}
\label{ss:convex}
If $K$ is a smooth strictly convex compact body (in the sense of Definition~\ref{e:def-convex-body}), then the \emph{Gauss map}
$$
G: (x,\theta)\in N_+(K)\mapsto \theta\in\IS^{d-1}
$$
is a smooth diffeomorphism, where $N_+(K)$ is the outward normal bundle to $K$ -- see~\cite[\S3.1]{DangLeautaudRiviere22} for more details.
We can write 
\begin{align}
\label{e:inverse-gauss}
G^{-1}(\theta)=(\mathbf{v}(\theta),\theta) ,\quad\text{with}\quad\mathbf{v}:\IS^{d-1}\to \d K,
\end{align}
being a smooth diffeomorphism. 
That is to say, $\mathbf{v}$ is the parametrization of $\partial K$ by its outward normal. 
We shall see in \S\ref{ss:analyticfunction} below (once spaces of analytic functions will be properly defined) that if $K$ is in addition analytic, then $\mathbf{v}$ is a real-analytic diffeomorphism and has all its components in a standard analytic space.
Note that the assumption $0 \in \mathring{K}$ translates as (see~\cite[Remark~3.2 p~942]{DangLeautaudRiviere22} for a proof)
\begin{equation}
\label{e:asspt-hyper-importante}
\theta \cdot \mathbf{v}(\theta)>0,\quad  \text{ for all }\theta\in \mathbb{S}^{d-1} .
\end{equation}
From this point forward, we will always assume that $K$ is strictly convex in the sense of Definition~\ref{e:def-convex-body} (so that $\mathbf{v}$ is well defined) and that~\eqref{e:asspt-hyper-importante} is satisfied.

In this parametrization, the support function $h_K$ (resp. $h_{-K}$) of the convex set $K$ (resp. $-K$), defined in~\eqref{e:defh_K} reads 
\begin{align}
\label{e:hK-fct-v}
h_{K}(\xi) = \xi\cdot\mathbf{v}\left(\frac{\xi}{|\xi|}\right) , \quad h_{-K}(\xi)= h_K(-\xi) = -\xi\cdot\mathbf{v}\left(-\frac{\xi}{|\xi|}\right),\quad \text{for all}\quad \xi\in\IR^d\setminus \{0\} ,
\end{align}
see~\cite[Section~3.2]{DangLeautaudRiviere22}. We conclude this paragraph with the following geometric lemma that will be used later on in our arguments.
\begin{lemma}\label{l:geometric-lemma}
Let $K\subset \R^d$ ($d\geq 2$) be a strictly convex compact body such that $0 \in\mathring{K}$,  and let $\mathbf{v}$ be the associated parametrization of $\d K$ according to~\eqref{e:inverse-gauss}.  Then there exists $c_0>0$ such that, for every $\theta\in \IS^{d-1}$ and for every $\omega\in\IS^{d-1}$, one has
\[ 
\left|\nabla_{\IS^{d-1}}(\omega\cdot \mathbf{v}(\theta))\right|^2\geq c_0d_{\IS^{d-1}}(\theta,\{\pm \omega\})^2,
\]
 where $\nabla_{\IS^{d-1}}$ denotes the gradient with respect to $\theta$, and
\[ 
|\omega\cdot\mathbf{v}(\theta)-\omega\cdot\mathbf{v}(\pm\omega)|\leq c_0^{-1}d_{\IS^{d-1}}(\theta,\pm\omega)^2.
\]
\end{lemma}

\begin{proof} 
Let $\omega\in\IS^{d-1}$ and set $\phi_{\omega}(\theta):=\mathbf{v}(\theta)\cdot\omega,$ for every $\theta\in\IS^{d-1}$. 
According to~\eqref{e:inverse-gauss}, this defines a map $\tilde{\phi}_\omega:\mathbf{v}\in\partial K\mapsto \mathbf{v}\cdot\omega$. Since $\partial K$ is compact, proving the lemma amounts to prove that $\mathbf{v}(\pm \omega)$ are the unique critical points of $\tilde{\phi}_\omega$, and they are non-degenerate. The uniformity follows from the compactness of the parameter space $\omega\in\IS^{d-1}$ and $\mathbf{v}\in \d K$.

Observe  that the critical points of the map $\tilde{\phi}_\omega$ correspond to the points $\mathbf{v}$ in $\partial K$ such that $T_{\mathbf{v}}\partial K$ is orthogonal to $\omega.$ In other words, this map has exactly two critical points given by $\mathbf{v}(\pm\omega)\in\d K$. Let us now check that they are non-degenerate. Denote by $N_\mathbf{v}$ the unit outer normal to $\partial K$, and also one of its smooth extensions to a neighborhood of $\partial K$ in $\R^d$. We will need some basic concepts from the theory of isometric immersions to finish the proof, see \cite{Submanifolds-and-holonomy} for a presentation. The shape operator is an endomorphism of $T\partial K$ given by
\[
A : u \in T(\partial K) \mapsto - pr(d_u N_\mathbf{v}), 
\]
where $pr$ is the orthogonal projection on $T(\partial K)$~\cite[Def.~4.17]{Lee09}. The Gauss equation relates the curvature of $\partial K$ with the shape operator, so that the strict convexity assumption means that $A$ is negative definite. Now, since $\tilde{\phi}_\omega$ is the restriction to $\partial K$ of a linear function on $\R^d$, its hessian at the critical point is given by 
\[
\nabla^2_{\partial K} \tilde{\phi}_\omega(X,X) = \langle A(X),X\rangle 
\]
The strict convexity property thus translates into the non-degeneracy of the critical points. 
%
\end{proof}

\subsection{Lifting to the tangent bundle}
\label{ss:dynamics}

In order to describe the analytical properties of the function $\mathcal{P}_{K_0}$, it is convenient to lift the associated counting problem~\eqref{e:countingfunction} into a dynamical question on the unit tangent bundle $S\IT^d=\IT^d\times\IS^{d-1}$ of the torus. Doing so for counting lattice points relative to some subset with appropriate convexity assumptions is a classical idea that can be traced back to Margulis' PhD thesis~\cite{Margulis69, Margulis04} in the context of hyperbolic geometry. See also~\cite{BroiseParkkonenPaulin2019} for a recent book on these questions. Here, following~\cite{DangRiviere20d, DangLeautaudRiviere22}, we will rewrite the Poincar\'e series $\mathcal{P}_{K_0}$ in terms of De Rham currents and of the flow 
\begin{align}
\label{e:def-varphi-t}
\varphi^t=e^{tV}:(x,\theta)\in S\IT^d\mapsto(x+t\mathbf{v}(\theta),\theta)\in S\IT^d,
\end{align}
where $\mathbf{v}:\theta\in\IS^{d-1}\mapsto \mathbf{v}(\theta)\in\partial K$ is the parametrization of $\partial K$ by its outward normal $\theta\in\IS^{d-1}$. Note that the corresponding vector field $V=\mathbf{v}(\theta)\cdot\partial_x$ already appeared in~\eqref{e:def-V}. Before considering the analytical issues raised by the study of Poincar\'e series, let us briefly recall how this question is handled in~\cite{DangLeautaudRiviere22} in the $\mathcal{C}^\infty$ case. The key point is that, for any smooth function $\eta$ that is compactly supported in $(0,\infty)$, one has~\cite[Cor.~4.18]{DangLeautaudRiviere22}
\begin{equation}\label{e:key-integration-formula}
\sum_{\xi\in 2\pi\mathbb{Z}^d\setminus K_0}\eta(d_K(K_0,\xi))=(-1)^{d-1}\int_{S\IT^d}[\Sigma_1]\wedge\iota_V\left( \int_{\IR}\eta(t)\varphi^{-t*}([\Sigma_2]) dt \right),
\end{equation}
where $[\Sigma_1]$ and $[\Sigma_2]$ are the currents of integration on the submanifolds
\[
\Sigma_2:=N_+(K_0)=\{(x_{K_0}(\theta),\theta):\,\theta\in\IS^{d-1}\},\quad\text{and}\quad
\Sigma_1=S_{[0]}\IT^d=\{[0]\}\times\IS^{d-1},
\]
where $\theta\in\mathbb{S}^{d-1}\mapsto x_{K_0}(\theta)\in\partial K_0$ is the parametrization of $\partial K_0$ by its outward normal.
\begin{remark}
In the case of the convex $K$ defining the Finsler metric, we used the convention $\mathbf{v}=x_K$.
\end{remark}
Here both submanifolds are oriented with the canonical volume form $\Vol_{\IS^{d-1}}(\theta,d\theta)$. Recall from~\cite[Eq.~(4.6)]{DangLeautaudRiviere22} that these currents of integration can be written as
\[
[\Sigma_1](x,\theta,dx,d\theta)=\delta_0^{\IT^d}(x)dx^1\wedge\ldots \wedge dx^d,
\]
and
\begin{align}
\label{e:Sigma-current}
[\Sigma_2](x,\theta,dx,d\theta):=\delta_0^{\IT^d}(x-x_{K_0}(\theta))\bigwedge_{j=1}^dd\left(x^j-x_{K_0}^j(\theta)\right),
\end{align}
with 
\begin{align}
\label{e:def-delta}
\delta_0^{\IT^d}(x)=\sum_{\xi\in\IZ^d}\frac{e^{i\xi\cdot x}}{(2\pi)^d}.
\end{align}
According to Lemma 9.5 in~\cite{DangLeautaudRiviere22}, formula~\eqref{e:key-integration-formula} still makes sense for $\eta(t) = \chi_0(t)e^{-st}$ where $\Re(s)>0$ and $\chi_0$ is a smooth nondecreasing function on $\IR$ which is equal to $0$ on $(-\infty,0]$ and to $1$ on $[T_0,\infty)$, where $T_0$ is chosen small enough so that if $x\notin K_0$ and $x\in 2\pi \IZ^d$, $d_K(x,K_0)\geq T_0$. That is to say, with $\mathcal{P}_{K_0}$ defined in~\eqref{e:def-PK0}, for $\Re(s)>0$, one has
\begin{equation}\label{e:key-integration-formula2}
 \mathcal{P}_{K_0}(s)=(-1)^{d-1}\int_{S\IT^d}[\Sigma_1]\wedge\iota_V\left( \int_{\IR}\chi_0(t)e^{-st}\varphi^{-t*}([\Sigma_2]) dt \right).
\end{equation}
Hence, studying the analytical properties of the function $\mathcal{P}_{K_0}$ amounts to studying the analytical properties of the right-hand side of~\eqref{e:key-integration-formula2}.
\begin{remark}\label{r:support-chi0} In fact, up to removing a finite number of terms in the sum defining $\mathcal{P}_{K_0}$ (hence an entire function) and as we are only interested in the holomorphic properties of $\mathcal{P}_{K_0}$, we may without loss of generality replace $\chi_0$ by some $\chi_1$ satisfying
\begin{equation}\label{e:supportchi0}
\chi_1\in\mathcal{C}^\infty(\IR_+,[0,1]),\ \text{supp}(\chi_1)\subset (\kappa_1,\infty),\ \text{and}\ \text{supp}(1-\chi_1)\subset (-\infty,2\kappa_1),
\end{equation}
with $\kappa_1>0$ as large as we want. In particular, $\text{supp}(\chi_1')\subset(\kappa_1,2\kappa_1)$. Equivalently, for such a function $\chi_1$, the right hand side of~\eqref{e:key-integration-formula2} is equal to $\mathcal{P}_{K_0}(s)$ up to adding an entire function. Such a function will henceforth be generically denoted by $\mathsf{H}$. 
\end{remark}

\subsection{Decomposing Poincar\'e series as a sum of integrals}
\label{ss:oscillatory}

The goal of this section is to reduce the analysis to the study of certain oscillatory integrals, relying on the fact that we only need to understand the analytical properties of the right-hand side of~\eqref{e:key-integration-formula2}. For this, let us set
\begin{equation}\label{eq:def-Omega}
\Omega(t,\theta,d\theta):=\sum_{j=1}^d(-1)^{j+1}\mathbf{v}^{j}(\theta)\bigwedge_{k\neq j}d\left(t\mathbf{v}^k(\theta)+x_{K_0}^k(\theta)\right)=\left(\sum_{\ell=0}^{d-1} \Omega_\ell(\theta)t^\ell\right)\Vol_{\IS^{d-1}}(\theta,d\theta).
\end{equation}
This defines a volume form, polynomial in $t$, and the $\Omega_\ell(\theta)$ denotes its coefficient of degree $\ell$.
In the spirit of \cite[\S4]{DangLeautaudRiviere22} we find by direct computation that
\[
(-1)^{d-1}\int_{\IT^d} e^{i\xi\cdot x} \Big[dx^1\wedge\dots \wedge dx^d \wedge\iota_V \varphi^{-t*}([\Sigma_2])\Big] = \Omega(t,\theta,d\theta) e^{i\xi\cdot(t\mathbf{v}(\theta)+x_{K_0}(\theta))}.
\]
Expanding the integration currents $[\Sigma_i]$ in Fourier series in $x$, using Fourier inversion together with formul\ae~\eqref{e:Sigma-current}--\eqref{e:def-delta}, we may rewrite \eqref{e:key-integration-formula2} for $\Re(s)>0$ as (see \cite[\S4]{DangLeautaudRiviere22})
\begin{equation}\label{e:poincare-fourier}
 {\mathcal{P}}_{K_0}(s)
 = \mathsf{H}(s)+\frac{1}{(2\pi)^d}\sum_{\xi\in\IZ^d}\int_{\IR}\chi_1(t)e^{-st}\left( \int_{\IS^{d-1}}e^{i\xi\cdot\left(t\mathbf{v}(\theta)+x_{K_0}(\theta)\right)}\Omega(t,\theta,d\theta)\right) dt .
\end{equation}
\begin{remark}
Note that the fact that the sum over $\xi$ in~\eqref{e:poincare-fourier} is convergent comes from an integration by parts arguments present in the proof of Theorem 8.8 of \cite{DangLeautaudRiviere22}. 
We recall briefly that one uses, for $\theta$ close $\pm \xi/|\xi|$, integration by parts in time thanks to the identity
\[
\frac{1}{s-i\xi\cdot\mathbf{v}(\theta)}\frac{d}{dt}\left(e^{-t\left(s-i\xi\cdot\mathbf{v}(\theta)\right)}\right)=e^{-t\left(s-i\xi\cdot\mathbf{v}(\theta)\right)}.
\]
For points $\theta$ away from $\pm \xi/|\xi|$, the integration by parts is in the $\theta$ variable, using the identity 
\[
\frac{1}{it}\frac{\nabla_{\IS^{d-1}} \phi_t}{|\nabla_{\IS^{d-1}}\phi_t(\theta)|^2}\cdot \nabla_{\IS^{d-1}}\left(e^{it\phi_t(\theta)}\right)=e^{it\phi_t(\theta)},\quad\text{with}\ \phi_t(\theta):= \xi\cdot(\mathbf{v}(\theta)+ \frac{1}{t}x_{K_0}(\theta)),
\]
and the fact that $\kappa_1>0$ (and thus $t$) is large enough.
These integration by parts arguments will be used again in the proof of Lemma~\ref{l:decomp-zeta-I} below.
\end{remark}

To formulate the main decomposition formula, we introduce the standard form of integral
\begin{equation}\label{eq:def-I(stuff)}
I_{\lambda,\omega}^{(k)}(z,F,\phi) := \int_{\IS^{d-1}} \frac{e^{i\lambda\phi(\theta)}}{(z- i\omega\cdot \mathbf{v}(\theta))^{k+1}}F(\theta) d\Vol_{\IS^{d-1}}(\theta)
\end{equation}
Here, $\lambda\geq 1, \omega\in\IS^{d-1}, k \in \N$ and $F:\IS^{d-1}\rightarrow \IC$, $\phi:\IS^{d-1}\rightarrow \IR$ are assumed to be smooth functions.

\begin{lemma}
\label{l:decomp-zeta-I}
Assuming that $\chi_1$ satisfies~\eqref{e:supportchi0},
there are numbers $\mathsf{E}^{(\ell)}\in \mathbb{R}$ and for $N\geq 0$ an entire function $\mathsf{H}_N$ such that we have, for $\Re(s)>0$,
 \begin{align}
\label{e:decomposition-zeta-N}
\mathcal{P}_{K_0}(s)
 =\mathsf{H}_N(s) + \sum_{\ell=0}^{d-1}\frac{\mathsf{E}^{(\ell)}}{s^{\ell+1}} + \mathcal{P}_{K_0}^{<N}(s) + \mathcal{P}_{K_0}^{\geq N}(s).
\end{align}
where for some constant $C>0$ (that does not depend on $N$),
\begin{align}
\label{e:expo-bound}
|\mathsf{H}_N(s)| &\leq C(1+ N^d e^{- C\Re(s)}),\\
\intertext{and,}
\mathcal{P}_{K_0}^{<N}(s) & :=\frac{1}{(2\pi)^d}\sum_{0<|\xi|<N}\sum_{\ell=0}^{d-1} \frac{\ell !}{|\xi|^{\ell+1}}I_{|\xi|,\frac{\xi}{|\xi|}}^{(\ell)}\left(\frac{s}{|\xi|},\Omega_\ell, \frac{\xi}{|\xi|}\cdot x_{K_0}\right), \nonumber \\
\label{e:def-zeta-sup}
\mathcal{P}_{K_0}^{\geq N}(s) & :=\frac{1}{(2\pi)^d}\sum_{|\xi|\geq N}\sum_{\ell=0}^{d-1}\sum_{j=0}^\ell \frac{\ell!}{j!|\xi|^{1+\ell-j}} \int_{\IR}\chi_1'(t) t^j e^{-st}I_{|\xi|,\frac{\xi}{|\xi|}}^{(\ell-j)}\left(\frac{s}{|\xi|}, \Omega_\ell, \frac{\xi}{|\xi|}\cdot(t\mathbf{v}+x_{K_0})\right) dt .
\end{align}
\end{lemma}
Notice that $\mathsf{E}^{(\ell)} \in \mathbb{R}$ can be interpreted in terms of mixed volumes of the two convex sets $K$ and $K_0$ as it is shown in~\cite[\S~10]{DangLeautaudRiviere22}. This explains the behaviour at $s=0$ in Theorem~\ref{t:maintheo1}.

Notice also that we have kept the cutoff function in time $\chi_1$ in the series $\mathcal{P}_{K_0}^{\geq N}(s)$ (high frequencies), whereas we have removed it from the finite sum $\mathcal{P}_{K_0}^{<N}(s)$ (its c)

\begin{remark} In~\cite{DangLeautaudRiviere22}, slightly more general Poincar\'e series were considered associated with so-called admissible (Legendrian) submanifolds of $S\T^d$. For the the sake of simplicity, we avoid this discussion here but the proof of our main Theorems remains true in that case by replacing $x_{K_0}$ in the above formulas by the functions $\tilde{x}_2^\pm-\tilde{x}_1^\pm$ appearing in \S4 from this reference. 
\end{remark}

As a direct consequence of Lemma~\ref{l:decomp-zeta-I}, the proofs of Theorems~\ref{t:maintheo1} and~\ref{t:maintheo2} reduce to the analysis of integrals of the form $I^{(\ell)}_{\lambda,\omega}(z,F,\phi)$, defined in~\eqref{eq:def-I(stuff)}.
In the upcoming sections, we will analyze the holomorphic properties of the functions $z\mapsto I^{(\ell)}_{\lambda,\omega}(z,F,\phi) $ with uniform bounds with respect to the parameters $\lambda$, $\omega$, $\|\phi\|$ and $\|F\|$ (for some appropriate choices of analytic norms).

\begin{proof}[Proof of Lemma~\ref{l:decomp-zeta-I}]
We start by observing that the finite number of terms that we removed from the sum~\eqref{e:poincare-fourier} by changing $\chi_0$ into $\chi_1$ contributes by an entire function $\mathsf{H}$ satisfying
\begin{equation}\label{eq:bound-individual-contrib-entire}
|\mathsf{H}(s)|\leq C(1+e^{-C\Re s}),
\end{equation}
and in particular~\eqref{e:expo-bound}. Let us turn to the the series in the RHS of~\eqref{e:poincare-fourier}. Using~\eqref{eq:def-Omega}, it rewrites as
\begin{align}\label{e:poincare-fourier2}
 {\mathcal{P}}_{K_0}(s)-\mathsf{H}(s) & = \frac{1}{(2\pi)^d}\sum_{\ell=0}^{d-1}\sum_{\xi\in\IZ^d} M_\xi(s), \quad \text{ with }\\
 \nonumber
M_\xi(s) &:= \int_{\IR}\chi_1(t)t^\ell e^{-st}\left( \int_{\IS^{d-1}}e^{i\xi\cdot\left(t\mathbf{v}(\theta)+x_{K_0}(\theta)\right)}\Omega_\ell(\theta)d\Vol_{\IS^{d-1}}(\theta)\right) dt  .
\end{align}
Concerning the term with $\xi=0$, we have (see e.g.~\cite[Equation~(8.31)]{DangLeautaudRiviere22} for a justification)
\[
\frac{1}{(2\pi)^d}M_0(s) = \frac{1}{(2\pi)^d} \int_{\IR}\chi_1(t)t^\ell e^{-st}\left( \int_{\IS^{d-1}}\Omega_\ell(\theta)d\Vol_{\IS^{d-1}}(\theta)\right) dt  = \frac{\mathsf{E}^{(\ell)}}{s^{\ell+1}} + \mathsf{H}^{(\ell)}(s) ,
\]
where $\mathsf{E}^{(\ell)} \in \mathbb{R}$ (being interpreted in terms of mixed volumes of the two convex sets $K$ and $K_0$ in~\cite[Section~10]{DangLeautaudRiviere22}) and where $\mathsf{H}^{(\ell)}$ is an entire function, satisfying \eqref{eq:bound-individual-contrib-entire}.
We can therefore concentrate on the sum for  $\xi\neq 0$ in~\eqref{e:poincare-fourier2}.
Writing $(-\d_s)^\ell(e^{-st}) = t^\ell e^{-st}$ and integrating by parts in $t$, we can rewrite for $\xi \neq 0$
\begin{align*}
M_\xi(s) & = (-\partial_s)^\ell \int_\IR  \chi_1(t) e^{-st} \left(\int_{\IS^{d-1}} e^{i\xi\cdot(t\mathbf{v}(\theta) +x_{K_0}(\theta))} \Omega_\ell(\theta)d\Vol_{\IS^{d-1}}(\theta) \right) dt \\
	&=(-\partial_s)^\ell \int_\IR \chi'_1(t) \left(\int_{\IS^{d-1}} \frac{e^{i\xi\cdot(t\mathbf{v}(\theta) + x_{K_0}(\theta)) - ts}}{s - i\xi\cdot\mathbf{v}(\theta)} \Omega_\ell(\theta)d\Vol_{\IS^{d-1}}(\theta)\right) dt,\\
	&=\sum_{j=0}^{\ell} \frac{\ell !}{j !} \int_\IR \chi'_1(t) t^j e^{- st} \left(\int_{\IS^{d-1}} \frac{e^{i\xi\cdot(t\mathbf{v}(\theta) + x_{K_0}(\theta)) }}{(s - i\xi\cdot\mathbf{v}(\theta))^{\ell- j + 1}} \Omega_\ell(\theta)d\Vol_{\IS^{d-1}}(\theta)\right) dt,
\end{align*}
where the last identity follows from the Leibniz formula $(-\partial_s)^\ell(e^{-st}(s-a)^{-1})=\sum_{j=0}^{\ell} \frac{\ell !}{j !} t^je^{-st}(s-a)^{j-\ell-1}$.
Recalling the definition of $I_{\lambda,\omega}^{(k)}$ in~\eqref{eq:def-I(stuff)}, this is directly the expected expression for the part of the sum corresponding to $|\xi|\geq N$.\begin{remark}
Recall from the above discussion that, up to an extra integration by parts in the time variable or in the $\theta$ variable, this sum is convergent in $\xi$ for $\text{Re}(s)>0$.
\end{remark}
For the small frequencies $|\xi|<N$, we observe that for $\eta\in C^\infty_c([0,\infty))$, supported in $[0, 2\kappa_1)$,
\[
\mathsf{H}_\xi(s):=(-\partial_s)^\ell \int_\IR \eta(t) e^{-st} \left(\int_{\IS^{d-1}} e^{i\xi\cdot(t\mathbf{v}(\theta) + x_{K_0}(\theta))} \Omega_\ell(\theta)d\Vol_{\IS^{d-1}}(\theta) \right) dt
\]
is an entire function, satisfying also \eqref{eq:bound-individual-contrib-entire}. If we sum over $\{0<|\xi|<N\}$, i.e $\sim N^d$ terms, we obtain the bound~\eqref{e:def-zeta-sup} announced in the lemma. We apply this argument with $\eta = 1-\chi_1$ to replace $\chi_1$ by integration over $\R^+$. This leads to the formula for $\mathcal{P}^{<N}_{K_0}(s)$ with elementary computations.
\end{proof}

\begin{remark}
 In order to alleviate the notation and except if there may be a confusion, we will often drop the dependence in $F$, $\phi$ and $\ell$ in the following and just write 
\[
 I_{\lambda,\omega}(z) := I^{(\ell)}_{\lambda,\omega}(z,F,\phi) .
\]
\end{remark}

\subsection{Main technical results on $\mathcal{P}_{K_0}^{<N}(s)$ and $\mathcal{P}_{K_0}^{\geq N}(s)$}
\label{s:main-tech-res}
According to Lemma~\ref{l:decomp-zeta-I}, the proofs of Theorems~\ref{t:maintheo1} and~\ref{t:maintheo2} reduce to proving that the functions $\mathcal{P}_{K_0}^{<N}(s)$ and $\mathcal{P}_{K_0}^{\geq N}(s)$ have the expected holomorphic continuation properties.
In the present section, we present these two statements in Propositions~\ref{p:low-freq} and~\ref{p:highfrequency} respectively.

 In order to state the main continuation result for $\mathcal{P}_{K_0}^{<N}(s)$, we introduce, for $\xi \in \R^d$ and $\nu>0$, the scaled box
		\begin{align}
		\label{e:shape-cut}
		\mathcal{V}_{\xi,\nu} &  := (-\nu |\xi|, 0] + i \left(\xi \cdot\mathbf{v}\Big(-\frac{\xi}{|\xi|}\Big)-\nu |\xi|, \xi \cdot\mathbf{v}\Big(\frac{\xi}{|\xi|}\Big)+\nu |\xi|\right) .
		\end{align}
		The following statement, proved in Section~\ref{s:low-freq}, describes the continuation properties of $\mathcal{P}_{K_0}^{< N}(s)$ (low frequency), together with the analysis of the shape of the singularities.

\begin{proposition}[Holomorphic continuation of $\mathcal{P}_{K_0}^{<N}$ with cuts]
\label{p:low-freq}
Let $0<R<1/(2\sqrt{d-1})$  and assume $\mathbf{v}\in \ml{A}_{R}(\IS^{d-1})^d$.
Then, there is $\nu>0$ (given by Corollary~\ref{c:equality-J} below) depending only on $\mathbf{v}$ (and on $R$) such that for all $N>0$, for all $x_{K_0}\in \ml{A}_{R}(\IS^{d-1})^d$,  the function $\mathcal{P}_{K_0}^{<N}$, {\em a priori} defined in $\Re(s)>0$, satisfies the following statements.

\begin{enumerate}
\item 
For any choice of curves $(\mathscr{C}_\xi)_{\xi\in \Z^d,0<|\xi|<N}$ satisfying for all $\xi\in \Z^d,0<|\xi|<N$,
\begin{itemize}
\item $\mathscr{C}_\xi \in C^1_{pw} \big([0,1];  \mathcal{V}_{\xi,\nu} \big)$,
\item $\mathscr{C}_\xi (0) = i \xi\cdot\mathbf{v}(-\frac{\xi}{|\xi|})$ and $\mathscr{C}_\xi(1)=i \xi\cdot\mathbf{v}(\frac{\xi}{|\xi|})$,
\item  $\mathscr{C}_\xi([0,1])$ is {\em strictly homotopic}  to $ i\Big[ \xi\cdot\mathbf{v}(-\frac{\xi}{|\xi|}),   \xi\cdot\mathbf{v}(\frac{\xi}{|\xi|})\Big]$ in $\mathcal{V}_{\xi,\nu}$, 
\end{itemize} 
 the function $\mathcal{P}_{K_0}^{<N}$ extends holomorphically to the connected component of set 
$$
 \tilde\C_N(\mathscr{C}) := \C \setminus \left( \bigcup_{\xi \in \Z^d, 0<|\xi|<N} \mathscr{C}_\xi\right)
 $$
 containing $\{\Re(s)>0\}$.

\item 
For all $\xi_0 \in \Z^d\setminus\{0\}$, there exists a neighborhood $\mathcal{N}_{\xi_0}^\pm$ of $i\xi_0 \cdot  \mathbf{v}(\pm\frac{\xi_0}{|\xi_0|})$ in $\C$ and for all $N>1$ holomorphic functions $\mathcal{H}_{\xi_0,N}^\pm, \mathcal{K}_{\xi_0,N}^\pm$ on $\mathcal{N}_{\xi_0}^\pm$ such that, for all $s \in\mathcal{N}_{\xi_0}^\pm\cap\{ \Re(z)> 0\}$,
  \begin{align}
\mathcal{P}_{K_0}(s) =\mathcal{H}_{\xi_0,N}^\pm (s) 
 \left(s- i\xi_0 \cdot\mathbf{v}\Big(\pm\frac{\xi_0}{|\xi_0|}\Big) \right)^{- \frac{d+1}{2}} 
+ \mathcal{K}_{\xi_0,N}^\pm (s) ,  
\end{align}
 if $d\geq 2$ is even, and  
  \begin{align}
\mathcal{P}_{K_0}^{<N}(s) =\mathcal{H}_{\xi_0,N}^\pm (s) 
 \left(s- i\xi_0 \cdot\mathbf{v}\Big(\pm\frac{\xi_0}{|\xi_0|}\Big)  \right)^{- \frac{d+1}{2}} 
+ \mathcal{K}_{\xi_0,N}^\pm (s) \ln  \left(s- i\xi_0 \cdot\mathbf{v}\Big(\pm\frac{\xi_0}{|\xi_0|}\Big)  \right),  
\end{align}
 if $d\geq 3$ is odd.
 \end{enumerate}
\end{proposition}

		\begin{figure}[!h]
			\centering
			\begin{tikzpicture}			
				\draw[dashed] (-3,0) -- (3,0) node[anchor=west]{$\R$};
				\draw[dashed] (0,-3)--(0,3) node[above]{$i\R$};
				\filldraw[black] (0,2.5) circle (2pt) node[right]{$i\xi\cdot \mathbf{v}\Big(\frac{\xi}{|\xi|}\Big)$};
				\filldraw[black] (0,-2.2) circle (2pt) node[right]{$i\xi\cdot \mathbf{v}\Big(-\frac{\xi}{|\xi|}\Big)$};
				\filldraw[black] (-2,0) circle (2pt) node[anchor=north east]{$-\frac{\nu}2 |\xi|$};
				\draw[very thick] (-2,2.5) -- (-2,-2.2) node[anchor=north east]{$\mathscr{C}_\xi$} ; 
				\draw[very thick]  (-2,2.5) -- (0,2.5)  ; 
				\draw[very thick]  (0,-2.2) -- (-2,-2.2)  ; 
			\end{tikzpicture}
			\caption{A possible curve $\mathscr{C}_\xi$ in  the complex plane (in the general case in which $\mathbf{v}(-\theta)\neq-\mathbf{v}(\theta)$)}
			\label{f:slit}
		\end{figure}
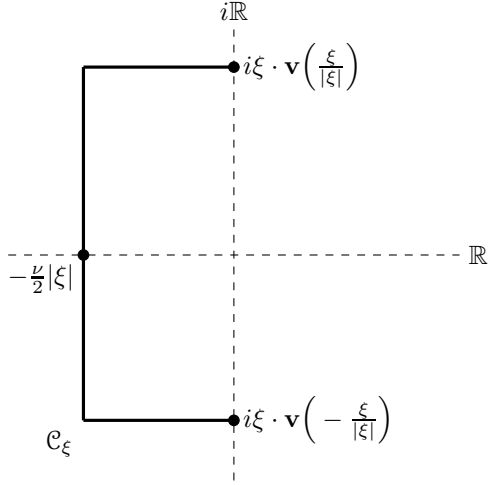

		See Figure~\ref{f:slit} for a picture of one possible curve $\mathscr{C}_\xi$. 
		We shall later on describe further the geometry of the set $\tilde\C_N$. This result will be proved in paragraph~\ref{ss:zeta-low-freq} and the functional space $\mathcal{A}_R(\IS^{d-1})$ is the space of analytic functions with analyticity radius $R>0$. See Section~\ref{s:analytic-norms} for the precise definition. Finally, $C^1_{pw} \big([0,1];  \mathcal{V}_{\xi,\nu} \big)$ denotes the space of continuous and piecewice $C^1$ paths in $\mathcal{V}_{\xi,\nu}$, and, by strictly homotopic, we mean that the endpoints remain the same all along the homotopy path.

\medskip
Once low frequencies are analyzed, the second main technical result concerns the continuation of $\mathcal{P}_{K_0}^{\geq N}(s)$ (high frequencies), where the summation issues over $\xi$ are considered. The following statement is proved in Section~\ref{s:high-freq}.

\begin{proposition}\label{p:highfrequency} Let $0<R<1/(2\sqrt{d-1})$ and assume that $\mathbf{v}$ and $x_{K_0}$  belong to $\ml{A}_{R}(\IS^{d-1})^d$. Then, there exist  $\delta_0,\kappa_1,N_0>0$ such that, for every $N>N_0$, the function $\mathcal{P}_{K_0}^{\geq N}$, initially defined in $\{\operatorname{Re}(w)>0\}$
extends holomorphically to
\begin{align}
\label{e:Omega-delta}
\Omega_{\delta_0}(N):=\left\{\operatorname{Re}(w)>-\delta_0 N,\ |\operatorname{Im}(w)|<\delta_0N\right\}
 \cup \{\operatorname{Re} (w)>0\}
,
\end{align}
 as soon as $\supp(\chi_1)\subset(\kappa_1,\infty)$ and $\supp(\chi_1')\subset(\kappa_1,2\kappa_1)$, where $\chi_1$ is the function appearing in the definition of $\mathcal{P}_{K_0}^{\geq N}(s)$.
\end{proposition}

\subsection{Analytic continuation of $\mathcal{P}_{K_0}$: end of the proof of Theorems~\ref{t:maintheo1} and~\ref{t:maintheo2}}
Based on Propositions~\ref{p:low-freq} and~\ref{p:highfrequency}, we conclude the section by stating and proving holomorphic continuation of $\mathcal{P}_{K_0}$ which is the content of Theorems~\ref{t:maintheo1} and~\ref{t:maintheo2}. 
We recall from~\eqref{e:hK-fct-v}, that $\xi\cdot\mathbf{v}(\frac{\xi}{|\xi|})=h_K(\xi)>0$ and $\xi\cdot\mathbf{v}(-\frac{\xi}{|\xi|})=-h_K(-\xi)<0$, 
so that the set in~\eqref{e:shape-cut} can be rewritten as 
\begin{align}
\label{e:Vxinu-bis}
		\mathcal{V}_{\xi,\nu}=(-\nu |\xi|, 0] + i \big(-h_K(-\xi) -\nu |\xi|, h_K(\xi)+\nu |\xi|\big)  .
\end{align}
\begin{theorem}
\label{t:continuation-zeta}
Let $d\geq 2$ and $K\subset \R^d$ be an analytic strictly convex compact body in the sense of Definition~\ref{e:def-convex-body}, such that $0 \in\mathring{K}$. 
Then, there is $\nu>0$ (given by Corollary~\ref{c:equality-J} and depending only on $K$) such 
for any $K_0$ being either a point or an analytic strictly convex compact body, 
 the function $\mathcal{P}_{K_0}$, {\em a priori} defined in $\{\Re(s)>0\}$, satisfies the following statements.
\begin{enumerate}
\item 
For any choice of curves $(\mathscr{C}_\xi)_{\xi\in \Z^d\setminus\{0\}}$ satisfying for all $\xi\in \Z^d\setminus\{0\}$,
\begin{itemize}
\item $\mathscr{C}_\xi \in C^1_{pw} \big([0,1];  \mathcal{V}_{\xi,\nu} \big)$,
\item $\mathscr{C}_\xi (0) = - i h_K(-\xi)$ and $\mathscr{C}_\xi(1)=i h_K(\xi)$,
\item  $\mathscr{C}_\xi([0,1])$ is {\em strictly homotopic}  to $i[-h_K(-\xi), h_K(\xi)]$ in $\mathcal{V}_{\xi,\nu}$, 
\end{itemize} 
 the function $\mathcal{P}_{K_0}$ extends holomorphically to the connected component of set 
$$
 \tilde\C(\mathscr{C}) := \C \setminus \left( \{0\}  \cup \bigcup_{\xi \in \Z^d \setminus\{0\}} \mathscr{C}_\xi\right)
 $$
 containing $\{\Re(s)>0\}$.
 \item 
For all $\xi_0 \in \Z^d\setminus\{0\}$, there exists a neighborhood $\mathcal{N}_{\xi_0}^\pm$ of $\pm ih_K(\xi_0)$ in $\C$ and holomorphic functions $\mathcal{H}_{\xi_0}^\pm, \mathcal{K}_{\xi_0}^\pm$ on $\mathcal{N}_{\xi_0}^\pm$ such that, for all $s \in\mathcal{N}_{\xi_0}^\pm\cap\{ \Re(z)> 0\}$,
  \begin{align}
\mathcal{P}_{K_0}(s) =\mathcal{H}_{\xi_0}^\pm (s) 
 \left(s \mp ih_K(\xi_0) \right)^{- \frac{d+1}{2}} 
+ \mathcal{K}_{\xi_0}^\pm (s) ,  
\end{align}
 if $d\geq 2$ is even, and  
  \begin{align}
\mathcal{P}_{K_0}(s) =\mathcal{H}_{\xi_0}^\pm (s) 
 \left(s \mp ih_K(\xi_0) \right)^{- \frac{d+1}{2}} 
 + \mathcal{K}_{\xi_0}^\pm (s) \ln  \left(s \mp i h_K(\xi_0)\right),  
\end{align}
 if $d\geq 3$ is odd.
 \end{enumerate}
\end{theorem}
Theorem~\ref{t:continuation-zeta} does not describe the meromorphic continuation of $\mathcal{P}_{K_0}$ at zero but this was already discussed after Lemma~\ref{l:decomp-zeta-I}.
Theorem~\ref{t:continuation-zeta} is complemented by Lemma~\ref{l:geometry-of-cuts} below, describing nice choices of curves $(\mathscr{C}_\xi)_{\xi\in \Z^d\setminus\{0\}}$, and accordingly the geometry of the set $\tilde\C(\mathscr{C})$. The proof of Theorem~\ref{t:continuation-zeta} relies on the decomposition of the function $\mathcal{P}_{K_0}$  in~\eqref{e:decomposition-zeta-N} in terms of $\mathcal{P}_{K_0}^{<N}$ and $\mathcal{P}_{K_0}^{\geq N}$, together with Proposition~\ref{p:highfrequency} concerning  $\mathcal{P}_{K_0}^{\geq N}$ and Proposition~\ref{p:low-freq} concerning $\mathcal{P}_{K_0}^{<N}$.

\begin{proof}[Proof of Theorem~\ref{t:continuation-zeta}]
First, according to Lemma~\ref{l:analytic-analytic} below, there exists $R>0$ such that the parametrizations $\mathbf{v},x_{K_0}$ of $\d K,\d K_0$ respectively satisfy $\mathbf{v},x_{K_0}\in\mathcal{A}_R(\IS^{d-1})^{d}$ (and $\mathbf{v}$ satisfies~\eqref{e:asspt-hyper-importante}).
Then, on the one hand, Proposition~\ref{p:low-freq} furnishes existence of $\nu>0$ (depending only on $\mathbf{v}$ and on $R$) satisfying its statement. On the other hand, Proposition~\ref{p:highfrequency} furnishes $\delta_1>0$ such that $\mathcal{P}_{K_0}^{\geq N}$
extends holomorphically to $\Omega_{\delta_1}(N)$ defined in~\eqref{e:Omega-delta}.

Second, we recall the decomposition of the function $\mathcal{P}_{K_0}$ in~\eqref{e:decomposition-zeta-N}: for any $N>0$, we have  
 \begin{align*}
\mathcal{P}_{K_0}(s)
 =\mathsf{H}_N(s) + \sum_{\ell=0}^{d-1}\frac{\mathsf{E}^{(\ell)}}{s^{\ell+1}} + \mathcal{P}_{K_0}^{<N}(s) + \mathcal{P}_{K_0}^{\geq N}(s)  ,
\end{align*}
where $\mathsf{H}_N$ is an entire function. 
Now, we fix  any given $N>0$. This decomposition, together with applying Proposition~\ref{p:highfrequency} to $\mathcal{P}_{K_0}^{\geq N}$ and Proposition~\ref{p:low-freq} to $\mathcal{P}_{K_0}^{<N}$ implies that the function $\zeta_{\Sigma_1,\Sigma_2}$ extends from $\C_+:=\{\Re(z)>0\}$ to the connected component of $\Omega_{\delta_1}(N) \cap  \tilde\C_N(\mathscr{C})$  (where $\Omega_{\delta_1}(N)$ is defined in~\eqref{e:Omega-delta}) containing $\C_+$. Since $\tilde\C_N(\mathscr{C})\supset \tilde\C(\mathscr{C})$, we deduce that
$\zeta_{\Sigma_1,\Sigma_2}$ extends  to the connected component of $\Omega_{\delta_1}(N)\cap  \tilde\C(\mathscr{C})$ containing $\C_+$.
Finally, since this statement holds for all $N>0$, we conclude that 
$\mathcal{P}_{K_0}$ extends  to the connected component of $$\bigcup_{N>0}\Omega_{\delta_1}(N)\cap  \tilde\C(\mathscr{C}) =  \tilde\C(\mathscr{C})$$ containing $\C_+$,
which proves the first statement of the theorem.

The second statement is a straightforward consequence of the second item of Proposition~\ref{p:low-freq} when recalling from~\eqref{e:hK-fct-v} that $\xi\cdot\mathbf{v}(\frac{\xi}{|\xi|})=h_K(\xi)$ and $\xi\cdot\mathbf{v}(-\frac{\xi}{|\xi|})=-h_K(-\xi)$, together with the fact that $\Z^d$ is stable by $\xi\mapsto -\xi$.
\end{proof}

As a complement for Theorem~\ref{t:continuation-zeta}, we describe in the next lemma particular choices of curves $(\mathscr{C}_\xi)_{\xi\in \Z^d\setminus\{0\}}$, together with the geometry of the associated set $\tilde\C(\mathscr{C})$.
\begin{lemma}[Shape of the cuts]
\label{l:geometry-of-cuts}
 There is $m>0$ such that for all $\nu>0$ the following statements hold:
\begin{enumerate}
\item If one chooses for all $\xi \in \Z^d\setminus\{0\}$  
$$\mathscr{C}_\xi = [ih_K(\xi) , ih_K(\xi)-\frac{\nu}{2}|\xi|] \cup [ ih_K(\xi)-\frac{\nu}{2}|\xi| ,-ih_K(-\xi) -\frac{\nu}{2}|\xi|] \cup  [-ih_K(-\xi) -\frac{\nu}{2}|\xi|, -ih_K(-\xi)] $$ (as in Figure~\ref{f:slit}), then one has 
$$
 \tilde\C(\mathscr{C}) \cap \Big(\mathcal{S}^{\nu m} \cup \left\{\Re(s)\geq -\frac{\nu}{2} \right\}\Big) =  \Big(\mathcal{S}^{\nu m} \cup \left\{\Re(s)\geq -\frac{\nu}{2} \right\} \Big) \setminus (\Lambda_K+ \R_-) ,
$$
where $\Lambda_K$ is defined in~\eqref{e:def-lambdak} and $\mathcal{S}^{\nu m}$ denotes the angular sector
$$\mathcal{S}^{\nu m} := \left\{\Re(s)\geq - \frac{\nu m}{2}|\Im(s)|\right\}.$$
\item If we assume that $K$ is symmetric about zero (that is to say $h_K(-\xi)=h_K(\xi)$), then for all $\xi \in \R^d$, we have 
$$
\ml{V}_{\xi,\nu} \supset  h_K(\xi)  \ml{W}^{\nu m}  , \quad \text{with}\quad \ml{W}^{\nu m} :=  (-\nu m,0] + i(-1-\nu m,1+\nu m)   ,
$$
where $\ml{V}_{\xi,\nu}$ was defined in~\eqref{e:Vxinu-bis}. In this case, for any fixed curve $\mathscr{C}$ satisfying $\mathscr{C} \in C^1_{pw} \big([0,1];  \ml{W}^{\nu m}  \big)$, $\mathscr{C}(0) = - i$ and $\mathscr{C}(1)=i$,  $\mathscr{C}([0,1])$ {\em strictly homotopic}  to $i[-1,1]$ in $\ml{W}^{\nu m}$, one can choose $\mathscr{C}_\xi =  h_K(\xi)\mathscr{C}$ in the statement of Theorem~\ref{t:continuation-zeta}. If in addition 
$\mathscr{C}$ satisfies the property 
\begin{align}
\label{e:condition-scaling}
\lambda \in \R \setminus \{1\} \implies \mathscr{C} \cap \lambda \mathscr{C} \neq \emptyset ,
\end{align}
 then the associated set $\tilde\C(\mathscr{C})$ is connected. 
\end{enumerate}
\end{lemma}
When taking $\mathscr{C}$ to be a triangle (see Figure~\ref{f:agrafes-triangles}, right), the cuts in Theorem~\ref{t:continuation-zeta} and Lemma~\ref{l:geometry-of-cuts} are reminiscent to the lines where the $L^2$-spectrum of stochastic perturbations of the operator $\theta\cdot\partial_x$ seem to accumulate in~\cite[Figure~3]{DyatlovZworski15}. Recall from \S\ref{s:dynamics} that the function $\mathcal{P}_{K_0}$ is naturally related to the resolvent of the operator $\mathbf{v}(\theta)\cdot\partial_x$ (see e.g.~\eqref{e:key-integration-formula2}). We do not however explain how the vanishing viscosity limit selects the triangular shape of the cuts. 
 
\begin{proof}
Firstly, since $0 \in\mathring{K}$, there exist $M\geq m>0$ such that 
\begin{align}
\label{e:equiv-norm}
m h_K(\xi) \leq |\xi| \leq M h_K(\xi) \quad \text{ for all }\quad \xi \in \R^d.
\end{align}
On $\mathcal{S}^{\nu m} := \{\Re(s)\geq - \frac{\nu m}{2}|\Im(s)|\}$ we have $ \Re(s)\geq  - \frac{\nu}{2} \frac{|\xi|}{\min (h_K(\xi), h_K(-\xi)) }|\Im(s)|$ for all $\xi \in \Z^d\setminus\{0\}$, and as a consequence,  for the curves $\mathscr{C}_\xi$ defined in the first item, we deduce
$$
 \mathcal{S}^{\nu m} \cap \mathscr{C}_\xi  =  \mathcal{S}^{\nu m} \cap  
\Big( [ih_K(\xi) , ih_K(\xi)-\frac{\nu}{2}|\xi|] \cup    [-ih_K(-\xi) -\frac{\nu}{2}|\xi|, -ih_K(-\xi)] \Big) .$$
Since this holds for all $\xi \in \Z^d\setminus\{0\}$, this concludes the proof of the first point of the lemma.

Secondly, notice that if $h_K(-\xi)=h_K(\xi)$, then the set $\mathcal{V}_{\xi,\nu}$ in~\eqref{e:Vxinu-bis} can be rewritten as 
$$
\mathcal{V}_{\xi,\nu}=(-\nu |\xi|, 0] + i \big(-h_K(\xi) -\nu |\xi|, h_K(\xi)+\nu |\xi|\big)  .
$$
 Hence, according to~\eqref{e:equiv-norm}, for any $\xi \in \R^d$, the set 
$\mathcal{V}_{\xi,\nu}$ contains the set 
$$
(-\nu m h_K(\xi), 0] + i \big(-h_K(\xi) -\nu m h_K(\xi) , h_K(\xi)+\nu m h_K(\xi)\big) = h_K(\xi) \ml{W}^{\nu m}.
$$
The last two statements of the lemma are straightforward consequences, the condition~\eqref{e:condition-scaling} ensuring that $h_K(\xi) \neq h_K(\xi') \implies \mathscr{C}_\xi \cap \mathscr{C}_{\xi'}= h_K(\xi)\mathscr{C}  \cap h_K(\xi')\mathscr{C} \neq \emptyset$.
\end{proof}

Finally, Theorems~\ref{t:maintheo1} and~\ref{t:maintheo2} are now particular cases of Theorem~\ref{t:continuation-zeta} combined with Lemma~\ref{l:geometry-of-cuts} (applied to particular choices of the curves $\mathscr{C}_\xi$).

\section{Analytic norms on $\IS^{d-1}$}
\label{s:analytic-norms}

In this section, we review the construction of standard families of analytic norms on $\IS^{d-1}$ and discuss some of their properties in view of our applications to the analysis of the oscillatory integrals $I_{\lambda,\omega}(z)$ appearing in Lemma~\ref{l:decomp-zeta-I}. The norms we introduce are more or less standard when dealing with analytic regularity issues. For the sake of completeness, we give in this section a detailed overview of their definition together with their main properties in view of their application to our problem. For more details and references, the reader is invited to consult the recent works~\cite{GalkowskiZworski19a, GuedesBonthonneauJezequel20}. In particular, besides dealing with manifolds, the second reference also includes the extension to Gevrey regularity.

This section is organized as follows. We begin in \S\ref{s:complexification} by describing the analytic structure on $\mathbb{S}^{d-1}$. Then, in~\S\ref{ss:analyticfunction}
, we describe standard spaces of analytic functions. Finally, in~\S\ref{s:calculus-analytic}, we collect a few useful properties of these spaces that are used all along the article. For simplicity of exposition, we only deal with the case of $\mathbb{S}^{d-1}$ but most of the results presented are valid for more general analytic manifolds. We refer to~\cite{GuedesBonthonneauJezequel20} for more details and references on the general set-up (including the case of Gevrey regularity) and also for a more geometric point of view.

\subsection{Complexification of $\IS^{d-1}$ and local charts}
\label{s:complexification}

We first recall the existence of a natural charts on $\IS^{d-1} \subset \R^d$ centered at $e_d=(0,\ldots, 0,1)$, namely
\begin{equation}
\label{e:def-kappa}
\begin{array}{rrcl}
\kappa: &U :=\IS^{d-1}\cap\{y_d>0\}& \to & B^{d-1}(0,1):=\{y'\in\IR^{d-1}:|y'|<1\}  , \\
& \left(y_1,\ldots,y_d\right) & \mapsto & y' =  \left(y_1,\ldots,y_{d-1}\right) , 
\end{array}
\end{equation}
whose inverse is given by the map
\begin{equation}
\label{e:def-kappa-1}
\begin{array}{rrcl}
\kappa^{-1}: &B^{d-1}(0,1) & \to &  U , \\
&y'=(y_1,\ldots,y_{d-1}) & \mapsto & \left(y_1,\ldots,y_{d-1},\sqrt{1-|y'|^2}\right) .
\end{array}
\end{equation}
Now given any point $M\in \IS^{d-1}$, we may choose a rotation $R_M$ such that $R_M(M)=e_d$, and it yields a chart centered at $M$ by
\begin{align}
\label{e:kappaM}
\kappa_M = \kappa \circ R_M : U_M \to B^{d-1}(0,1) , 
\end{align} 
where $U_M$ is the open half-sphere centered at $M$. One can verify that the family $(\kappa_M,U_M)$ endows $\IS^{d-1}$ with a real analytic structure. Indeed, as can be witnessed from their explicit expressions, the transition maps $\kappa_{M}\circ\kappa_{M'}^{-1}$ are given by analytic maps.

The sphere $\mathbb{S}^{d-1}$ can be explicitly complexified through the complex quadric
$$\mathbb{S}^{d-1}_{\IC}:=\left\{(z_1,\ldots, z_d)\in\IC^d: z_1^2+\ldots+z_d^2=1\right\},$$
and, if we set $z=\tilde{y}+i\tilde{\eta}$, then we deduce 
$$\mathbb{S}^{d-1}_{\IC}:=\left\{\tilde{y}+i\tilde{\eta} , (\tilde{y} ,\tilde{\eta})\in \R^d\times \R^d , \quad   |\tilde{y}|^2-|\tilde{\eta}|^2=1,\quad \tilde{y} \cdot \tilde{\eta}=0 \right\}.$$ 
The coordinate chart $\kappa$ in~\eqref{e:def-kappa} can be extended into a holomorphic chart on $\IS^{d-1}_{\IC}$ as follows. Set 
$$U_\IC:=\left\{z\in\IS^{d-1}_{\IC}: \sum_{j=1}^{d-1}|z_j|^2<1\ \text{and}\ \Re(z_d)>0\right\},$$
which is an open set containing $U$ as a (totally real) subset. Writing 
$$
B^{d-1}_{\IC}(0,1):=\{z'\in\IC^{d-1}: |z_1|^2+\ldots+|z_{d-1}|^2<1\} ,
$$
we can verify that the map
$$\tilde{\kappa}:z=(z',z_d)\in U_{\IC}\mapsto z'\in B^{d-1}_{\IC}(0,1),$$
with inverse 
$$\tilde{\kappa}^{-1}: z'\in B^{d-1}_{\IC}(0,1)\mapsto \left(z',\sqrt{1-\sum_{j=1}^{d-1}z_j^2}\right) ,$$
defines a chart in a complex neighborhood of $e_d=(0,\ldots, 0,1)$ inside $\IS^{d-1}_{\IC}$. Together with the rotation $R_M$ in~\eqref{e:kappaM}, this yields an atlas of holomorphic charts $(\tilde{\kappa}_M,U_M^{\IC})$ in a small neighborhood of $M$ inside $\IS^{d-1}_{\IC}$.

\begin{remark} For latter purposes, we can write the matrix of the metric $g_{\text{Can}}$ on $\IS^{d-1}$ induced by the Euclidean structure on $\IR^d$. In the coordinates $y'=(y_1,\ldots, y_{d-1})\in B^{d-1}(0,1)$, it reads 
$$
g(y')=\text{Id}+\frac{1}{1-|y'|^2}\left(y_iy_j\right)_{1\leq i,j\leq d-1},
$$ 
and
$$
g^*(y')=g(y')^{-1}=\text{Id}-\left(y_iy_j\right)_{1\leq i,j\leq d-1}.
$$ 
Recall also that the determinant is given by $|g(y')|=\frac{1}{1-|y'|^2}.$
Hence, we can write down the explicit expressions for the gradient:
\begin{equation}\label{e:gradient}
 \nabla_{\IS^{d-1}}f=\sum_{i=1}^{d-1}\left(\sum_{j=1}^{d-1}(\delta_{ij}-y_iy_j)\partial_{y_j}f\right)\partial_{y_i}.
\end{equation}
This operator has real analytic coefficients that are well defined on $B^{d-1}(0,1)$ and that can be holomorphically  extended to $B^{d-1}_{\IC}(0,1)$.
\end{remark}

\subsection{Analytic functions on $\IS^{d-1}$}\label{ss:analyticfunction}

We first define analytic norms in an open set of $\R^{n}$. Given $n\geq 1$ and $V \subset \R^{n}$ a bounded open set and $R>0$, we set
\begin{align}
\label{e:def-analytic-norm}
\|f\|_{\mathcal{A}(V, R)}:= \sup_{x \in \overline{V}, \alpha\in\mathbb{N}^{n}}\frac{R^{|\alpha|}|\partial^{\alpha} f(x)|}{\alpha!}  , 
\end{align}
and 
\begin{align}
\label{e:defA(V,R)}
\mathcal{A}(V,R) := \{f \in C^\infty_b(V) , \|f\|_{\mathcal{A}(V,R)} < \infty\} .
\end{align}
\begin{lemma}
\label{l:banach-A}
 The space $\mathcal{A}(V, R)$ is a Banach space for the norm $\|\cdot\|_{\mathcal{A}(V, R)}$.
 Moreover, we have $0<R\leq R' \implies \mathcal{A}(V, R')\subset \mathcal{A}(V, R)$ and $\|f\|_{\mathcal{A}(V,R)} \leq  \|f\|_{\mathcal{A}(V,R')}$ for all $f \in \mathcal{A}(V,R')$.
\end{lemma}

\begin{proof}
 Let $(f_j)_{j\in \N} \in \mathcal{A}(V, R)^\N$ be a Cauchy sequence. From the definition of $\|.\|_{\mathcal{A}(V,R)}$, each function $f_j$ can be extended to a holomorphic function $\tilde{f}_j$ in a small open neighborhood of $\overline{V}$ in $\mathbb{C}^n$. This neighborhood does not depend on $j \in \mathbb{N}$. The sequence $(\tilde{f}_j)_{j\in \N}$ is a Cauchy sequence for the uniform convergence in each compact subset of this open neighborhood.
Therefore $ \tilde{f}_j \rightarrow \tilde{f} $ uniformly on each such compact and by~\cite[Corollary 2.2.4 p.~26]{Hormandercomplex}, we find that the limit function $\tilde{f}$ is holomorphic on the open neighborhood of $\overline{V}$. We denote by $f$ its restriction to $\overline{V}$. We repeat the previous proof for each $\partial^\alpha f_j$ which is also Cauchy in $\mathcal{A}(V, R)$ by definition, this also tells us that $ \frac{R^{|\alpha|}|\partial^{\alpha} f_j (x)|}{\alpha!} $ is bounded uniformly in $x\in \overline{V}$, $\alpha$ and $j$. Since $\partial^{\alpha} f_j \rightarrow \partial^{\alpha} f $ uniformly on $\overline{V}$ by Cauchy formula, we deduce that $ \frac{R^{|\alpha|}|\partial^{\alpha} f(x)|}{\alpha!} $ is bounded uniformly in $x\in \overline{V}$, multi--index $\alpha$ yielding $f\in \mathcal{A}(V, R)$.
\end{proof}

Coming back to $\IS^{d-1}$ and using the conventions of the previous section, we now provide with a first natural way to define analytic norms on functions on $\IS^{d-1}$.
\begin{definition}
\label{d:def-norms}
We fix once and for all a finite family of points $(M_j)_{j=1}^{n_0}$ such that $\bigcup_{j=1}^{n_0} B_{\IS^{d-1}}(M_j,\pi/4)=\IS^{d-1}$.
For $f\in\ml{C}^\infty(\IS^{d-1})$, we  then set 
\begin{align}
\label{e:def-analytic-norm-sphere}
\|f\|_{R}:= 
\sum_{j=1}^{n_0}
\| f\circ\kappa_{M_j}^{-1} \|_{\mathcal{A}(B^{d-1}(0,1/2), R)}
=\sum_{j=1}^{n_0}\sup_{|x|\leq 1/2,\alpha\in\mathbb{N}^{d-1}}\frac{R^{|\alpha|}|\partial^{\alpha}(f\circ\kappa_{M_j}^{-1})(x)|}{\alpha!},
\end{align}
and 
\begin{align}
\label{e:def-AR-sphere}
\ml{A}_{R}(\IS^{d-1}):=\left\{f\in\ml{C}^\infty(\IS^{d-1}):\|f\|_R<\infty\right\}, 
\end{align}
\end{definition}
We refer to Lemma~\ref{r:bas-point} below concerning the dependence of the norms with respect to the choice of the points $M_j$.
A direct consequence of Lemma~\ref{l:banach-A} is the following result.
\begin{lemma}
\label{l:AR-Banach}
For all $R>0$, the space $\ml{A}_{R}(\IS^{d-1})$ is a Banach space for the norm $\| \cdot \|_{R}$ and we have $0<R\leq R' \implies\ml{A}_{R'}(\IS^{d-1})\subset \ml{A}_{R}(\IS^{d-1})$ together with $\|f\|_{R} \leq  \|f\|_{R'}$ for all $f \in\ml{A}_{R'}(\IS^{d-1})$.
\end{lemma}
With Definition~\ref{d:def-norms} at hand, one can verify that any function $f\in \ml{A}_R$ satisfies the following property: for all $j \in \{1,\dots,N\}$, the real-analytic function  
 $f\circ\kappa_{M_j}^{-1}:\{|x|\leq 1/2\}\rightarrow\IC$ admits a holomorphic extension $\tilde{f}_{M_j}$ to the complex (and polydisc shaped) neighborhood
 \begin{align}
 \label{e:def-UR}
 \ml{U}_{R}:=\bigg\{z\in\IC^{d-1}:\exists x\in\IR^{d-1}\ \text{with}\ |x|\leq 1/2\ \text{such that}\ \forall 1\leq j\leq d-1,\ |z_j-x_j|<R\bigg\}
 \end{align}
 of $\{x\in\IR^{d-1}:|x|\leq 1/2\}$ in $\IC^{d-1}$.
The holomorphic extension $\tilde{f}_{M_j}$ of  $f\circ\kappa_{M_j}^{-1}$ may be explicitly written as
 \begin{equation}
 \label{e:extension-expansion}
 \tilde{f}_{M_j}(z) := \sum_{\alpha \in \N^{d-1}} \frac{\partial^{\alpha}(f\circ\kappa_{M_j}^{-1})(x)}{\alpha!}(z-x)^\alpha ,
 \end{equation}
for some $x \in \R^{d-1},|x|\leq 1/2$ (and does not depend on the choice of $x$). We can then introduce the following auxiliary norm:
\begin{align}
\label{e:define-HR-norm}
\left\|f\right\|_{\mathcal{H}_R}:=\sum_{j=1}^{n_0}\left\|\tilde{f}_{M_j}\right\|_{L^\infty(\mathcal{U}_R)} .
\end{align}
The norm $\left\| \cdot \right\|_{\mathcal{H}_R}$ has the following nice property: 
\begin{align}
\label{e:norm-algebra}
\|f_1f_2\|_{\mathcal{H}_R} & = \sum_{j=1}^{n_0}\left\|\tilde{f}_{1,M_j}\tilde{f}_{2,M_j}\right\|_{L^\infty(\mathcal{U}_R)}\leq \sum_{j=1}^{n_0}\left\|\tilde{f}_{1,M_j}\right\|_{L^\infty(\mathcal{U}_R)}\left\|\tilde{f}_{2,M_j}\right\|_{L^\infty(\mathcal{U}_R)} \nonumber\\ 
& \leq\|f_1\|_{\mathcal{H}_R}\|f_2\|_{\mathcal{H}_R}.
\end{align}

 \begin{remark}
In the following, we will only consider $\mathcal{A}_R$ spaces for $R<1/(2\sqrt{d-1})$. In that way, $\ml{U}_{R}\cap\mathbb{R}^{d-1}\subset B_{\mathbb{R}}^{d-1}(0,1)$ in~\eqref{e:def-UR}, and thus points in $\ml{U}_{R}$ do not reach the singularity of the charts $\kappa_{M_j}$.
\end{remark}

\begin{lemma}\label{l:comparison-norm} Let $R_0<R_1<\frac{1}{2\sqrt{d-1}}$. Then, for every $f\in\mathcal{A}_{R_1}(\mathbb{S}^{d-1})$, one has
 $$
 \|f\|_{R_0}\leq \left\|f\right\|_{\mathcal{H}_{R_0}}\leq n_0\frac{R_1^{d-1}}{(R_1-R_0)^{d-1}}\|f\|_{R_1}.
 $$
\end{lemma}
\begin{proof}
 The first inequality is a consequence of Cauchy formula and the polydisc shape of $\ml{U}_R$ (see e.g.~\cite[Theorem~2.2.3]{Hormandercomplex}). For the second one, notice that we may write the formula~\eqref{e:extension-expansion} for every $z\in\mathcal{U}_{R_0}$, 
 where $x$ lies in $\mathbb{R}^{d-1}$ with $|x|\leq 1/2$. Hence, one has for all $z\in\mathcal{U}_{R_0}$,
\begin{align*}
|\tilde{f}_{M_j}(z)|&\leq
\sum_{\alpha\in \mathbb{N}^{d-1}} \vert\tilde{f}_{M_j}^\alpha(x)\vert \frac{\vert z-x \vert^{\vert\alpha\vert}}{\alpha !} \leq \sum_{\alpha\in \mathbb{N}^{d-1}} \vert\tilde{f}_{M_j}^\alpha(x)\vert R_1^{\vert \alpha\vert} \frac{\vert z-x \vert^{\vert \alpha\vert }}{\alpha ! R_1^{\vert\alpha\vert} }\leq 
\sum_{\alpha\in \mathbb{N}^{d-1}} R_1^{\vert \alpha\vert}\frac{\vert f_{M_j}^\alpha(x)\vert}{\alpha !}  \frac{R_0^{\vert \alpha\vert }}{ R_1^{\vert\alpha\vert} }
\\
&\leq  \|f\|_{R_1} \sum_{\alpha\in\mathbb{N}^{d-1}}\left(\frac{R_0}{R_1}\right)^{|\alpha|} = \|f\|_{R_1} \frac{R_1^{d-1}}{(R_1-R_0)^{d-1}}.
\end{align*}
We used the generating function for 
multi--indices $\sum_{\alpha\in \mathbb{N}^{d-1}}z^{\vert\alpha\vert} = \left(\sum_{i=0}^\infty z^i \right)^{d-1}  $ which equals $(1-z)^{-(d-1)}$ when $\vert z\vert<1$.
The final result follows when adding the estimates for the $n_0$ charts which explains the $n_0$ in factor.
\end{proof}

\begin{remark}
\label{r:bas-point}
We note that the norms $\|.\|_{R}$ depend on the choice of the base points for our charts (and also on the coordinate map itself) even if we do not emphasize it in our notations. However, if we restore for a moment the dependence on the finite atlas (set of points) $(M_j)_{j=1}^{n_0}$ and introduce another choice $(M_j')_{j=1}^{n_0'}$ of a such atlas, we have,  for $R'>R$ some constant $C(R,R')>0$ such that
$$\|f\|_{R,(M_j)}\leq\sum_{j=1}^{n_0}\|\tilde{f}_{M_j}\|_{L^\infty(\ml{U}_{R})}\leq n_0\sum_{j=1}^{n_0'}\|\tilde{f}_{M_j'}\|_{L^\infty(\ml{U}_{R})}\leq n_0C(R,R')\|f\|_{R',(M_j')}.$$
\end{remark}

Finally, we conclude this paragraph by clarifying the relationship between real-analytic regularity of the convex set $K$ as assumed in Theorems~\ref{t:maintheo1},~\ref{t:maintheo2} etc. and regularity properties of its parametrization maps $\mathbf{v}$ defined in~\eqref{e:inverse-gauss}, and used all along the paper.
We recall that $\partial K$ is said to be a real-analytic submanifold of $\mathbb{R}^d$ of dimension $d-1$ if $\partial K$ can be covered by finitely many charts $\psi:U\subset\mathbb{R}^d\rightarrow B_{\mathbb{R}}^{d}(0,1)$ such that $\psi(\partial K \cap U)\subset \big( \R^{d-1}\times\{0\} \big) \cap B_{\mathbb{R}}^{d}(0,1)$ and whose transition maps are functions lying on some space $\mathcal{A}(V,R)^d$ for $R>0$ small enough and some bounded open set $V\subset \mathbb{R}^d$. 
\begin{lemma}
\label{l:analytic-analytic}
Assume that $K$ is an analytic  strictly convex compact body of $\IR^d$ in the sense of Definition~\ref{e:def-convex-body}, such that $0 \in \mathring{K}$. Then, there exists $R>0$ such that the inverse Gauss map $\mathbf{v}:\mathbb{S}^{d-1}\rightarrow \partial K$ defined in~\eqref{e:inverse-gauss} satisfies $\mathbf{v}\in\mathcal{A}_R(\IS^{d-1})^{d}$ (in the sense that all $d$ components of $\mathbf{v}$ belong to $\mathcal{A}_R(\IS^{d-1})$) together with~\eqref{e:asspt-hyper-importante}.
\end{lemma}

For instance, this holds in the case where $K=B_{\mathbb{R}}^d(0,1)$ by extending, near the north pole, the local charts $\kappa$ given in \S\ref{s:complexification} to a small neighborhood of $\mathbb{S}^{d-1}$. In the case of $K=B_{\mathbb{R}}^d(0,1)$, one has $\mathbf{v}(\theta)=\theta$.
\begin{proof}
Recall that the Gauss map in~\eqref{e:inverse-gauss} is obtained from the local inversion theorem and this theorem holds true in the real-analytic category~\cite[Prop.~6.1]{cartanelementary}. As a consequence, the pullback $\tilde{\mathbf{v}}:B_{\mathbb{R}}^{d-1}(0,1)\rightarrow B_{\mathbb{R}}^{d-1}(0,1)$ of $\mathbf{v}$ by a local chart $\kappa$ given in \S\ref{s:complexification} has all its coordinates in an analytic space $\mathcal{A}(V,R)$ with $V\subset\R^{d-1}$. Hence, one can deduce that $\mathbf{v}$ has all its components in an analytic space $\mathcal{A}_R(\IS^{d-1})$ for $R>0$ small enough. 
\end{proof}

\subsection{Calculus properties of the spaces $\mathcal{A}_R(\IS^{d-1})$}
\label{s:calculus-analytic}
We collect in this section several elementary computational properties of the spaces $\mathcal{A}_R(\IS^{d-1})$, that will be helpful later on in the article.
First of all, it follows from the definition~\eqref{e:def-analytic-norm-sphere} that these analytic spaces are naturally included in $\mathcal{C}^0(\IS^{d-1})$ with an a priori bound.
\begin{lemma}\label{l:continuous} For any $R>0$ and $\psi\in\mathcal{A}_R(\IS^{d-1})$, we have $\psi \in \mathcal{C}^0(\IS^{d-1})$ with 
$$\|\psi\|_{\mathcal{C}^0(\IS^{d-1})}\leq \|\psi\|_{R}.$$
\end{lemma}

We also record the following useful properties:
\begin{lemma}\label{c:product} Let $k\geq 1$, let $R_1,\ldots,R_k>0$ and set $R<\min\{R_j: 1\leq j\leq k\}$. Then, for every $(\psi_1,\ldots,\psi_k)\in\ml{A}_{R_1}\times\ldots\times\ml{A}_{R_k}$, one has $\psi_1\psi_2\ldots\psi_k\in\mathcal{A}_R$ with
$$\|\psi_1\psi_2\ldots\psi_k\|_{R}\leq  n_0^k \prod_{\ell=1}^k\left(\frac{R_\ell}{R_\ell-R}\right)^{d-1}\|\psi_\ell\|_{R_\ell},$$
where $n_0$ is the number of charts in the atlas defining $\|\cdot\|_R$.

Moreover, for any $0<R<\min(R_1,R_2)$, and $\psi_2 \in \ml{A}_{R_2}$, we have $\mathbf{m}_{\psi_2} \in \ml{L}(\ml{A}_{R_1},\ml{A}_{R})$ with 
$$
\| \mathbf{m}_{\psi_2} \|_{\ml{A}_{R_1} \to \ml{A}_{R}} \leq n_0^2 \left(\frac{R_1}{R_1-R}\right)^{d-1} \left(\frac{R_2}{R_2-R}\right)^{d-1},
$$
where $\mathbf{m}_{\psi_2}$ is the multiplication operator as defined in \S\ref{s:multiplicatop}.
\end{lemma}

\begin{proof} From the estimate~\eqref{e:norm-algebra} of the product and Lemma~\ref{l:comparison-norm}, we deduce  
\begin{align*}
\left\|\psi_1\psi_2\ldots\psi_k\right\|_R
& \leq  \left\|\psi_1\psi_2\ldots\psi_k\right\|_{\ml{H}_R} \leq 
\left\|\psi_1\right\|_{\ml{H}_R}\left\|\psi_2\right\|_{\ml{H}_R}\ldots \left\|\psi_k\right\|_{\ml{H}_R} \\
& \leq n_0^k \prod_{\ell=1}^k\left(\frac{R_\ell}{R_\ell-R}\right)^{d-1}\|\psi_\ell\|_{R_\ell},
\end{align*}
 whence the first statement of the lemma. The second statement is a direct consequence of the first.
\end{proof}

\begin{corollary}\label{c:exponential}
For every  $0<R<R_1$ and $\psi \in \ml{A}_{R_1}$, we have $e^\psi \in \ml{A}_R$ together with
$$
\|e^{\psi}\|_{R}\leq e^{\frac{ n_0R_1^{d-1}}{(R_1-R)^{d-1}}\|\psi\|_{R_1}}.
$$
\end{corollary}
\begin{proof}
 Taking $\psi_1=\ldots=\psi_k=\psi\in\mathcal{A}_{R_1}$ in Lemma~\ref{c:product}, one deduces for all $0<R<R_1$,
$$
\|\psi^k\|_{R}\leq   n_0^k \left(\frac{R_1^{d-1}}{(R_1-R)^{d-1}}\right)^k\|\psi\|_{R_1}^k, 
$$
from which the result follows.
\end{proof}

We record an auxiliary lemma to control the size of the functions appearing in our oscillatory integrals in terms of analytic norms.
\begin{lemma}\label{l:product-exp} Let $0<R<R_1<\frac{1}{2\sqrt{d-1}}$. There exists a constant $C_{R,R_1}>0$ such that, for every $F$ and $X$ in $\mathcal{A}_{R_1}(\IS^{d-1})$ and $\lambda \in \R$, one has
$$
\left\|e^{i \lambda X} F\right\|_R\leq C_{R,R_1}e^{C_{R,R_1}|\lambda|\|X\|_{R_1}}\|F\|_{R_1}. 
$$
\end{lemma}
\begin{proof} This lemma follows directly from Lemma~\ref{c:product} and Corollary~\ref{c:exponential}.
\end{proof}

We now turn to the action of analytic differential operators on the analytic spaces $\mathcal{A}_R(\IS^{d-1})$.

\begin{lemma}\label{l:gradient}  For every $0<R_1,R_2<\frac{1}{2\sqrt{d-1}}$, for every $0<R<\min\{R_1,R_2\}$, there exists a constant $C(R,R_2)>0$ (independent of $R_1$) such that, for every $f\in\ml{A}_{R_2}$ and for every $\psi\in\ml{A}_{R_1}$, one has $\mathcal{L}_{\nabla_{\IS^{d-1}}f}\psi \in \ml{A}_{R}$ with
$$\left\|\mathcal{L}_{\nabla_{\IS^{d-1}}f}\psi\right\|_R\leq  \frac{ C(R,R_2)}{(R_1-R)}\|f\|_{R_2}\|\psi\|_{R_1}.$$
\end{lemma}
\begin{proof} The proof is close to the one of Lemma~\ref{c:product}, and we work similarly in one of the local charts $M_j$ (and for simplicity of the notation, we continue writing $\psi_1$ instead of $\psi_{1,M_j}$) with the induced spherical metric. 
We first write,
near every point $y_0$ with $|y_0|\leq 1/2$,
$$f(y)=\sum_{\alpha\in\IN_0^{d-1}}\frac{(y-y_0)2C_R\|\Omega_\ell\|_R\| e^{i\xi\cdot x_{K_0}}\|_{\mathcal{H}_{2R'}}^\alpha}{\alpha!}\partial^\alpha f(y_0).$$
Using~\eqref{e:gradient}, we can write
$$\mathcal{L}_{\nabla_{\IS^{d-1}}f}=\sum_{\alpha}\frac{\partial^\alpha f(y_0)}{\alpha!}\sum_{i,j=1}^{d-1}P_{i,j}(y)\partial_{y_j}(y-y_0)^\alpha\partial_{y_i},$$
where $P_{i,j}(y)=\delta_{ij}-y_iy_j$ is a polynomial in $y$, and in particular is uniformly bounded for $|y-y_0|\leq\min\{R_1,R_2\}$. This can be extended into a holomorphic operator on the open set $\mathcal{U}_{R_2}$:
$$
\mathcal{L}_{Y}=\sum_{j=1}^{d-1}\tilde{Y}_j(z)\partial_{z_j},
$$
where, according to Lemma~\ref{l:comparison-norm}, one can find $C_d(R,R_2)>0$ (depending only on $d$, $R$ and $R_2$) such that, for $\tilde{R}=\frac{R+R_2}{2}$ and for every $1\leq j\leq d-1$, 
\begin{equation}\label{e:gradient-step1}
\|\tilde{Y}_j\|_{L^{\infty}(\mathcal{U}_{\tilde{R}})}\leq C_d(R,R_2)\|f\|_{R_2}.
\end{equation}
We can now consider the holomorphic extension $\tilde{\psi}$ of $\psi$ to $\mathcal{U}_{R_1}$. One has
\begin{equation}\label{eq:gradient-step2}
\|\mathcal{L}_Y\tilde{\psi}\|_{L^\infty(\mathcal{U}_R)}\leq \sum_{j=1}^{d-1}\|\tilde{Y}_j\|_{L^\infty(\mathcal{U}_R)}\|\partial_{z_j}\tilde{\psi}\|_{L^{\infty}(\mathcal{U}_R)}\leq  \sum_{j=1}^{d-1}\|\tilde{Y}_j\|_{L^\infty(\mathcal{U}_{\tilde{R}})}\frac{2}{R_1-R}\|\tilde{\psi}\|_{L^{\infty}(\mathcal{U}_{\frac{R_1+R}{2}})},
\end{equation}
where the last inequality follows from the Cauchy formula applied along the variable $z_j$. Combining Lemma~\ref{l:comparison-norm} to~\eqref{e:gradient-step1} and~\eqref{eq:gradient-step2}, we get the expected result.
\end{proof}

Choosing $\psi=f$ in Lemma~\ref{l:gradient}, we deduce the  following corollary.
\begin{corollary}\label{c:normgradient} For every $0<R<R_1<\frac{1}{2\sqrt{d-1}}$, there exists $C(R,R_1)>0$ such that, for every $f\in\mathcal{A}_{R_1}(\IS^{d-1}),$
$$
\left\|  |\nabla_{\IS^{d-1}} f|^2   \right\|_R\leq C(R,R_1)\|f\|_{R_1}^2.
$$
\end{corollary}

For later applications, we finally introduce a local version of the norm $\|\cdot\|_{R}$ and the associated space $\ml{A}_R$. Given an open subset $O$ of $\IS^{d-1}$ and $0<R<\frac{1}{2\sqrt{d-1}}$, we set
$$
\left\|\psi\right\|_{R,O}:=\sum_{j=1}^N\sup_{|x|\leq 1/2,\ x\in\kappa_{M_j}(\overline{O})}\sup_{\alpha\in\mathbb{N}^{d-1}}\frac{R^{|\alpha|}\left|\partial^\alpha(\psi\circ\kappa_{M_j}^{-1})(x)\right|}{\alpha!}, 
$$
and denote by $\ml{A}_{R,O}$ the set of functions $\psi : \IS^{d-1}\to \C$ such that $\|\psi\|_{R,O}<\infty$. We now turn to iterations of order one differential operators together with their restrictions to open subsets. Given a local chart $\kappa_M : U_M \to B_\R^{d-1}(0,1/2)$ as introduced in Section~\ref{s:complexification}, given an open subset $O$ of $\IS^{d-1}$ and $R>0$, we also define the set
$$
\mathcal{U}_{R,O}:=\{z\in\mathbb{C}^{d-1}:\ \exists x\in \kappa_M(\overline{O})\cap B_{\mathbb{R}}^{d-1}(0,1/2)\ \text{such that}\ \forall j,\ |z_j-x_j|<R\}.
$$ 
With these notations at hand, we can provide an estimate on the powers of a differential operator.
\begin{lemma}\label{l:power-lie-derivative} 
Let $0<R\leq \frac{1}{2\sqrt{d-1}}$. Then, there exists a constant $C(R)>0$ such that, for every $R<R_2\leq \frac{R_1}{2}$, for any open set $O\subset \IS^{d-1}$, for any vector field $Y \in \ml{A}_{R_1,O}^d$, for any functions $W\in \ml{A}_{R_1,O}$, the following two statements hold: 
\begin{enumerate}
\item for all local chart $(U_M,\kappa_M)$ as in Section~\ref{s:complexification}, the operator $\mathcal{L}_Y+W$ (acting on functions) extends holomorphically to the set $\mathcal{U}_{R_2,O}$ as
$$
(\widetilde{Y}+\widetilde{W})(z):=\sum_{j=1}^{d-1}\tilde{Y}_{j}(z)\partial_{z_j} +\widetilde{W}(z), \quad z \in \mathcal{U}_{ R_2,O}  
$$
and we have for all $k\in\N$ and all $\tilde{\psi}$ holomorphic in $\mathcal{U}_{R_1,O}$, 
\begin{align}
\left\|(\widetilde{Y}+\widetilde{W})^k \tilde{\psi}\right\|_{L^\infty(\mathcal{U}_{R,O})}\leq \left(\frac{C(R)}{R_2-R}\left(\max_{j=1,\ldots,d}\|Y_j\|_{R_1,O}+\|W\|_{R_1,O}\right)\right)^k k! \|\tilde{\psi}\|_{L^\infty(\mathcal{U}_{R_2,O})},
\end{align}
\item for all $\psi\in \ml{A}_{R_1,O}$, one has $(\mathcal{L}_Y+W)^k\psi \in \ml{A}_{R,O}$ with  
$$
\left\|(\mathcal{L}_Y+W)^k\psi\right\|_{R,O}\leq \left(\frac{C(R)}{R_2-R}\left(\max_{j=1,\ldots,d}\|Y_j\|_{R_1,O}+\|W\|_{R_1,O}\right)\right)^k k!\|\psi\|_{R_2,O},
$$
where $\mathcal{L}_Y$ is the Lie derivative along the vector field $Y$. 
\end{enumerate}
\end{lemma}

The first item of Lemma~\ref{l:power-lie-derivative}  will be used ``away from the vanishing points of a vector field'' (i.e. with $O$ being the complement of a neighborhood of the vanishing points of $Y = \nabla_{\IS^{d-1}}\big( \omega \cdot \mathbf{v} \big)$) in Section~\ref{s:Anal-away-crit}. The proof of this lemma follows from the properties of the norm $\|\cdot\|_R$ (and similar properties for the norm $\|\cdot\|_{R,O}$) obtained in the previous lemmas from this paragraph. 
\begin{proof} We proceed as in the proof of Lemma~\ref{l:gradient} and we work in local charts without loss of generality. With 
$\psi \in \mathcal{A}_{R_1}(\IS^{d-1})$, we can write the holomorphic extension $\tilde{\psi}$ of $\psi$ to  $\mathcal{U}_{R_1,O}$. 

As in Lemma~\ref{l:comparison-norm}, 
setting $\tilde{R}_1=\frac{R+R_1}{2}$ and $\tilde{R}_2=\frac{R+R_2}{2}$, we have 
\begin{align}
\label{e:truc-chose}
\|\tilde{Y}_{j}\|_{L^\infty(\mathcal{O}_{\tilde{R}_1,O})}\leq \frac{n_0 (2R_1)^{d-1}}{(R_1-R)^{d-1}}\|Y_j\|_{R_1,O},\quad \|\widetilde{W}\|_{L^{\infty}(\mathcal{U}_{\tilde{R}_1,O})}\leq  \frac{n_0 (2R_1)^{d-1}}{(R_1-R)^{d-1}}\|W\|_{R_1,O} .
\end{align}
As in~\eqref{eq:gradient-step2}, this leads (through an application of Cauchy's formula) to the upper bound
\begin{equation}\label{eq:boundLinfty}
 \|(\widetilde{Y}+\widetilde{W})\tilde{\psi}\|_{L^\infty(\mathcal{U}_{R,O})}\leq \frac{2(d-1)}{\tilde{R}_2-R}\left(\max_j\|\tilde{Y}_{j}\|_{L^\infty(\mathcal{U}_{R,O})}+\|\widetilde{W}\|_{L^\infty(\mathcal{U}_{R,O})}\right)\|\tilde{\psi}\|_{L^\infty(\mathcal{U}_{\tilde{R}_2,O})},
\end{equation}
which yields the result of the lemma for $k=1$.
We now want to iterate this argument and to this end, we set 
$$
\mathsf{R}_\ell := R+\frac{\ell}{k}(\tilde{R}_2-R) ,  \quad \text{ for } \ell \in \{0,\dots , k\}, 
$$
to interpolate between $R$ and $\tilde{R}_1$, i.e.
 $$\mathsf{R}_0 = R , \quad \mathsf{R}_k = \tilde{R}_2, \quad \text{ and }\quad \mathsf{R}_\ell -\mathsf{R}_{\ell -1} = \frac1{k}(\tilde{R}_2-R).
$$
With the notation
\begin{align}
\label{e:defCYW}
C_{Y,W} := 2(d-1)\left(\max_j\|\tilde{Y}_{j}\|_{L^\infty(\mathcal{U}_{\tilde{R}_1,O})}+\|\widetilde{W}\|_{L^\infty(\mathcal{U}_{\tilde{R}_1,O})}\right), 
\end{align}
the  exact same argument as  in~\eqref{eq:gradient-step2} and~\eqref{eq:boundLinfty} shows
\begin{align*}
\|(\widetilde{Y}+\widetilde{W})^k\tilde{\psi}\|_{L^\infty(\mathcal{U}_{\mathsf{R}_0,O})}
\leq \frac{C_{Y,W}}{\mathsf{R}_1-\mathsf{R}_0} \|(\widetilde{Y}+\widetilde{W})^{k-1}\tilde{\psi}\|_{L^\infty(\mathcal{U}_{\mathsf{R}_1,O})} 
=  \frac{C_{Y,W}k }{\tilde{R}_2-R} \|(\widetilde{Y}+\widetilde{W})^{k-1}\tilde{\psi}\|_{L^\infty(\mathcal{U}_{\mathsf{R}_1,O})} .
\end{align*}
Iterating this argument $\ell$ times, we obtain, for all $\ell\in \{0,\dots,k\}$, 
\begin{align*}
\|(\widetilde{Y}+\widetilde{W})^k\tilde{\psi}\|_{L^\infty(\mathcal{U}_{\mathsf{R}_0,O})}
\leq\left(  \frac{C_{Y,W} k }{\tilde{R}_2-R} \right)^\ell \|(\widetilde{Y}+\widetilde{W})^{k-\ell}\tilde{\psi}\|_{L^\infty(\mathcal{U}_{\mathsf{R}_\ell,O})} , 
\end{align*}
and finally, for $\ell=k$,
\begin{align*}
\|(\widetilde{Y}+\widetilde{W})^k\tilde{\psi}\|_{L^\infty(\mathcal{U}_{R,O})}
\leq\left(  \frac{C_{Y,W} k }{\tilde{R}_2-R} \right)^k  \|\tilde{\psi}\|_{L^\infty(\mathcal{U}_{\mathsf{R}_k,O})} =\left(  \frac{C_{Y,W} k }{\tilde{R}_2-R} \right)^k  \|\tilde{\psi}\|_{L^\infty(\mathcal{U}_{\tilde{R}_2,O})} .
\end{align*}
The $k^k$ in factor comes from the fact that at each iterate, the radius decreases like $\frac{(\tilde{R}_1-R)}{k}$.
Combining this inequality with~\eqref{e:truc-chose} and recalling the definition of $C_{Y,W}$ in~\eqref{e:defCYW}, we can infer that, for all $k\geq 1$
\begin{align*}
\|(\widetilde{Y}+\widetilde{W})^k \tilde{\psi}\|_{L^\infty(\mathcal{U}_{R,O})}
 \leq  
 \frac{2^k(d-1)^k k^k}{(\tilde{R}_2-R)^k}
\tilde{C}_{Y,W}^k\|\tilde{\psi}\|_{L^\infty(\mathcal{U}_{\tilde{R}_2,O})},\quad  \text{with} \\
\tilde{C}_{Y,W}  =  \frac{n_0  (2R_1)^{d-1}}{(R_1-R)^{d-1}}\left(\max_{j=1,\ldots d}\|Y_j\|_{R_1,O}+\|W\|_{R_1,O}\right) , 
\end{align*}
and the conclusion of the first item of the lemma follows from Stirling's formula. The second item of the lemma follows from another application of (a local version of) Lemma~\ref{l:comparison-norm}.
\end{proof}

\section{Analyzing the low frequencies and proof of Proposition~\ref{p:low-freq}}
\label{s:low-freq}

The goal of this section is to analyze the holomorphic properties of $\mathcal{P}_{K_0}^{<N}$ which is a finite sum of integrals of the form $I^{(\ell)}_{\lambda,\omega}(z,F,\phi)$, defined in~\eqref{eq:def-I(stuff)}. The ultimate goal is to prove Proposition~\ref{p:low-freq} in \S\ref{ss:zeta-low-freq}.
The analysis essentially reduces to studying the continuation of $I^{(\ell)}_{\lambda,\omega}(z,F,\phi)$.  
As opposed to the upcoming Section~\ref{s:high-freq}, no precise control with respect to the parameter $\lambda$ is required at this stage (there is a finite number of such terms). Yet some care is required to describe the singularity at $z= i \omega\cdot\mathbf{v}(\pm \omega)$.
We first write
\begin{equation}
\label{e:decomp-I-I}
 I^{(\ell)}_{\lambda,\omega}(z, F,X)  =   \mathcal{I}_{(\ell)}\Big( z ,\omega\cdot\mathbf{v}(\cdot) , e^{i\lambda X(\cdot)}F \Big),
\end{equation}
where the second term depends holomorphically on $z$ and where we have set, for $\Re(z)>0$,
\begin{align}
\label{e:def-Ibis}
\mathcal{I}_{(\ell)}(z,\f,G) & :=  
 \int_{\IS^{d-1}}\frac{ G (\theta)}{(z-i\f(\theta))^{\ell+1}}d\Vol_{\IS^{d-1}}(\theta) . 
\end{align}
 The purpose of this section is to analyze the holomorphic properties of these integrals, hence deducing holomorphic properties of $\mathcal{P}_{K_0}^{<N}$ when specified to $\f=\omega\cdot\mathbf{v}$ and $G=e^{i\lambda X(\cdot)}F$. We shall see in Lemma~\ref{l:Iell-I} that the proof reduces to the study of the resolvent of the operator $\mathbf{m}_\f$ (with some care paid to the dependence in $\omega$), and thus prove Theorem~\ref{t:maintheo-multiplication} along the way.

\subsection{Preliminary reduction}

 The following lemma reduces the analysis of the family of integrals $\mathcal{I}_{(\ell)}(z,\f,G)$ in~\eqref{e:def-Ibis}  for $\ell \in \{0,\dots,d-1\}$ to a single integral.
\begin{lemma}
\label{l:Iell-I}
Setting 
\begin{align}
\label{e:def-II}
\mathcal{I}(z) :=  \int_{\IS^{d-1}}\frac{G(\theta)}{z-\f(\theta)}d\Vol_{\IS^{d-1}}(\theta) ,
\end{align}
we have for all $\ell \in \{0, \dots , d-1\}$, 
\begin{align}
\label{link-I-ell}
\mathcal{I}_{(\ell)}(z,\f,G)  = - i^{\ell+1} (\d_z^{\ell} \mathcal{I})\big(\frac{z}{i}\big)  , \quad \text{ for } \Re(z)>0 . 
\end{align}
\end{lemma}

\begin{remark}
Note that $\mathcal{I}_{(\ell)}$ is a priori defined on $\{\Re(z)>0\}$ that we want to understand its holomorphic extension through the imaginary axis (especially through the segment $i[\min(\f),\max(\f)]$). Hence, we need to understand the holomorphic extension of $\mathcal{I}$ from $\{\Im(z)<0\}$.              
             \end{remark}

As already explained,  the study of continuation properties of $I^{(\ell)}_{\lambda,\omega}(z, F,X)$ starting from $\Re(z)>0$ is reduced to that of $\mathcal{I}_{(\ell)}(z,\f,G)$. The latter reduces to the analysis of $\ml{I}(z)$ in~\eqref{e:def-II} according to~\eqref{link-I-ell}, with 
\begin{align}
\label{e:values-f-G}
\f(\theta) =\f_\omega(\theta) = \omega\cdot\mathbf{v}(\theta) , \quad G(\theta) = e^{i\lambda X(\theta)}F(\theta) , \quad \text{ for }X, F\in \mathcal{A}_{R_1} .
\end{align}
 Analyzing the holomorphic properties of $\mathcal{I}$ is the main purpose of this section.

 Thus, thanks to Lemma~\ref{l:product-exp}, $G$ belongs to $\mathcal{A}_R$ for every $0<R<R_1$ (with an uniform control in terms of the analytic norms of $F$ and $X$). Moreover, the (family of) function(s) $\f$ of interest is $\f(\theta)=\omega\cdot\mathbf{v}(\theta)$ with $\omega\in\IS^{d-1}$ and $\mathbf{v}\in\mathcal{A}_{R_1}^d$. In particular (see Lemma~\ref{l:geometric-lemma} and its proof), $\f$ has only two critical points that are nondegenerate and given by $\theta=\pm\omega$. In this situation, we thus have 
$$
\min_{\IS^{d-1}}\f= \omega\cdot\mathbf{v}(-\omega) <0<\omega\cdot\mathbf{v}(\omega) = \max_{\IS^{d-1}} \f ,
$$
since~\eqref{e:asspt-hyper-importante} yields that $\theta \cdot \mathbf{v}(\theta) >0$, for all $\theta \in \IS^{d-1}$ (consequence of the convexity of $K$ together with the assumption $0 \in \mathring{K}$).

\subsection{Main result for low frequencies}

The main goal of the present section is to prove the following result concerning the integral $\mathcal{I}_{(\ell)}(z,\f,G)$ defined in~\eqref{e:def-II}. 
We refer to Appendix~\ref{s:remainder-multivalued} for a brief reminder on multivalued holomorphic functions.

\begin{theorem} 
\label{t:extension-integral}  Let $0<R<\frac{1}{2\sqrt{d-1}}$ and let $\f$ and $g$ be two elements in $\mathcal{A}_R(\mathbb{S}^{d-1})$. Assume that $\f$ is a Morse function on $\mathbb{S}^{d-1}$ which admits only two critical points.
Then, setting
\begin{align*}
\mathcal{U}_\delta :=\{z\in \C,  \dist(z,i[\min(\f),\max(\f)] ) < \delta  \},
\end{align*}
there exists $\delta>0$ 
such that the function $z\mapsto \mathcal{I}_{(\ell)}(z,\f,G)$ has a {\em multivalued holomorphic extension} to the set $\mathcal{U}_\delta\setminus \{i\max(\f),i\min(\f)\} $, in the sense of Definition~\ref{def:multivalued} (from any $z_0$ such that $\Re(z_0)>0)$), satisfying,
for $z\in\mathcal{U}_\delta\cap\{\Re(z)>0\}$, 
\begin{align*}
\mathcal{I}_{(\ell)}(z,\f,G)=\mathcal{H}(z) (z-i\min(\f))^{\frac{d-3}{2}-\ell}(z-i\max(\f))^{\frac{d-3}{2}-\ell}+\mathcal{K}(z)
\end{align*}
where $\mathcal{H},\mathcal{K}$ are holomorphic on $\mathcal{U}_\delta$ if $d$ is even, and 
\begin{align*}
\mathcal{I}_{(\ell)}(z,\f,G)&= \left((z-i\min(\f))(i\max(\f)-z)\right) ^{\frac{d-3}{2}-\ell}\left(\ln \left(z-i\min(\f)\right)-\ln\left(z-i\max(\f)\right)\right) \mathcal{H}(z)\\
&+\mathcal{K}(z)\quad \text{ if }\ell\leqslant \frac{d-3}{2} , \\
\mathcal{I}_{(\ell)}(z,\f,G)&= \left((z-i\min(\f))(i\max(\f)-z)\right) ^{\frac{d-3}{2}-\ell}\mathcal{H}(z) , \\
&+\left(\ln \left(z-i\min(\f)\right)-\ln\left(z-i\max(\f)\right)\right) \mathcal{K}(z)\quad  \text{ if }\ell > \frac{d-3}{2} ,
\end{align*}
where $\mathcal{H},\mathcal{K}$ are holomorphic on $\mathcal{U}_\delta$ if $d\geqslant 3$ is odd.
\end{theorem}

Theorem~\ref{t:extension-integral} is presented here for fixed real-analytic functions $\f$ and $G$. 
In view of applications to Poincar\'e series, {\em families} depending on $\lambda$ and $\omega$ of such functions  $\f$ and $G$ need  to be considered (see~\eqref{e:values-f-G}), with a uniform control on the size $\delta$ of the domain of the extension. 
Henceforth all results and proofs presented below, in addition to proving Theorem~\ref{t:extension-integral}, also keep track of the uniformity of domains for the analytic extension with respect to the parameters $\lambda$ and $\omega$.
 As a corollary (of this quantitative approach of Theorem~\ref{t:extension-integral}), we deduce from decomposition \eqref{e:decomp-I-I} that the integral $I^{(\ell)}_{\lambda,\omega}(z, F,X)$ 
inherits exactly the same properties as $\mathcal{I}_{(\ell)}(z,\f=\omega.\mathbf{v} ,G=e^{i\lambda X}F)$ from Theorem \ref{t:extension-integral}. See \S\ref{ss:zeta-low-freq} for precise statements in view of Poincar\'e series. As a byproduct of our analysis, we deduce a result of independent interest describing the analytic continuation of the resolvent of multiplication operators, see Theorem~\ref{t:multiplicationoperators}.

The remainder of the section is devoted to the proof of Theorem~\ref{t:extension-integral} (and its relatives), and it proceeds in 
several steps:
\begin{enumerate}
\item Decomposing the integration form $Gd\text{Vol}_{\mathbb{S}^{d-1}}$ along level sets of the Morse function $\f$: this yields a new set of functions $[\min\f,\max\f]\ni \tau \mapsto\mathcal{J}(\tau)$ (the ``average'' of the form $Gd\text{Vol}_{\mathbb{S}^{d-1}}$ along level set $\{\f=\tau\}$) which allows to express $\mathcal{I}$ (and related derivatives $\mathcal{I}_{(\ell)}$) in terms of an analytic function $\mathcal{J}$ through the formula
\begin{equation}\label{e:relation-I-J}
 \mathcal{I}(z)=\int_{\min(\f)}^{\max(\f)} \frac{\mathcal{J}(\tau)}{z-\tau}d\tau.
\end{equation}
This is done in~\S\ref{ss:level-lines}.
\item Extending $\mathcal{J}$  holomorphically away from the critical values $\min\f,\max\f$, using the transport by a well-chosen analytic flow which is constructed using the Morse function $\f$. This is done in~\S\ref{ss:away-crit}.
\item Extending $\mathcal{J}$ near the critical values $\min\f,\max\f$ as multivalued function using the analytic Morse lemma together with explicit computations in the Morse chart. This is done in~\S\ref{s:multivalued-continuation-J}.
\item Using the formula~\eqref{e:relation-I-J} to relate the analytic continuation of $\mathcal{J}$ with that of $\mathcal{I}$. This is done in~\S\ref{ss:I-to-J-2}.
\item Taking  derivatives and deducing the analytic continuation of the functions $\mathcal{I}_{(\ell)}$ and $I^{(\ell)}_{\lambda,\omega}$ from the information obtained on $\mathcal{I}$. This is done in~\S\ref{ss:Iell-and-Iell} (hence proving Theorem~\ref{t:extension-integral}) and it allows us to conclude the proof of Proposition~\ref{p:low-freq} in~\S\ref{ss:zeta-low-freq}.
\end{enumerate}
In summary, we have the following chain of dependence between the analytical quantities introduced so far:
$$ \mathcal{J}\rightarrow \mathcal{I}\rightarrow \mathcal{I}_{(\ell)}\rightarrow I^{(\ell)}   \rightarrow \mathcal{P}_{K_0}^{<N},$$
where the last three arrows were already discussed at the beginning of this section and rely on relatively simple transformations. Hence, as already emphasized, the main focus of the upcoming sections is on the description of $\mathcal{I}$ (and thus of $\mathcal{J}$).

\subsection{Decomposing $\mathcal{I}$ according to the level sets of $\f$}\label{ss:level-lines}
 We denote by $\Crit(\f)$ the set containing the two critical points of $\f$;   in case $\f=\omega\cdot \mathbf{v}$, we have  $\Crit(\f) = \{-\omega,\omega\}$.
We  introduce the vector field 
\begin{align}
\label{e:def-champ-Yf}
Y_{\f}:=\frac{\nabla_{\IS^{d-1}}\f}{\left|\nabla_{\IS^{d-1}}\f\right|^2},
\end{align}
which is well defined away from $\Crit(\f)$ (that is to say, away from $\theta=\pm\omega$ in case $\f=\omega\cdot \mathbf{v}$). It has the same flow lines as $\nabla_{\IS^{d-1}}\f $ (only their parametrizations change). Here, we have one more time taken the canonical metric $g$ on  $\IS^{d-1}$ to define $|v|^2 = g(v,v)$ and $g(\nabla_{\IS^{d-1}} \f , v) = d\f(v)$ for all $v \in T\mathbb{S}^{d-1}$. We denote by $\Phi^t$ the flow associated to the vector field $Y_\f$. Given $x \in\mathbb{S}^{d-1}\setminus \Crit(\f)$ and by the Cauchy-Lipschitz Theorem, the map $\tau \mapsto \Phi^\tau(x)$ is well-defined in a neighborhood of $0$ (depending on the point $x$). 
Since $\mathcal{L}_{Y_\f}\f= d\f (Y_\f) =\frac{d\f(\nabla \f)}{|\nabla \f|^2 } = 1$, we have $\f(\Phi^\tau(x)) = \f(x)+\tau$ on the interval of definition of $\tau \mapsto \Phi^\tau(x)$. As a consequence, we deduce that it is well-defined for all $\tau \in \big(\min(\f)-\f(x) , \max(\f)-\f(x) \big)$ and that it satisfies
 $$\f(\Phi^\tau(x)) = \f(x)+\tau , \quad \text{ for all } x \in\mathbb{S}^{d-1}\setminus \Crit(\f) ,  \tau \in \big(\min(\f)-\f(x) , \max(\f)-\f(x) \big).  $$
We now fix $u_0\in (\min(\f) , \max(\f))$ and we deduce from the previous line that 
$$
\f(\Phi^{u-u_0}(x)) = \f(x) + u-u_0 = u , \quad \text{ for all } x \in \f^{-1}(u_0)  ,
$$
whence $\Phi^{(u-u_0)}_*\big(\f^{-1}(u_0)\big)= \f^{-1}(u)$.

The relevance of this vector field towards the analysis of $\mathcal{I}(z)$ follows from the following decomposition of the integral.

\begin{lemma}
\label{l:coarea}
Let $\f$ be a Morse function on $\IS^{d-1}$ with only two critical points, $Y_\f$ be defined in~\eqref{e:def-champ-Yf}, and $G \in C^0(\IS^{d-1})$. We have
$$
\mathcal{I}(z)=\int_{\min \f}^{\max\f}\frac{d\tau}{z-\tau}\left(\int_{\f^{-1}(0)}\iota_{Y_{\f}}\Phi^{\tau *}\left(G\Vol_{\IS^{d-1}}\right)\right).
$$
\end{lemma}
This is a version of the coarea formula (see e.g.~\cite[Chapter~7]{AGZV}) which follows from~\cite[Lemma~2.1]{DangRiviere20d}
after the following modifications.  
\begin{proof}
For every $\varepsilon>0$,~\cite[Lemma~2.1]{DangRiviere20d} shows
$$\int_{\min \f+\varepsilon}^{\max\f-\varepsilon}\frac{d\tau}{z-\tau}\left(\int_{\f^{-1}(0)}\iota_{Y_{\f}}\Phi^{\tau *}\left(G\text{Vol}_{\IS^{d-1}}\right)\right)= \int_{\mathbb{S}^{d-1}} \mathds{1}_{ [\min \f+\varepsilon,\max\f-\varepsilon] }(\f) \frac{ G }{z-\f} d\text{Vol}_{\IS^{d-1}} $$
since we removed the critical value of the Morse function $\f$. 
Then for $z\notin  [\min \f,\max\f]$ an application of the dominated convergence theorem allows to let $\varepsilon\rightarrow 0^+$ on both sides. 
  \end{proof}
  
  In particular, it is natural to introduce the following function, for all $\tau \in (\min(\f) , \max(\f))$, 
\begin{align}
\label{e:def-J}
\mathcal{J}(\tau)=\mathcal{J}(\tau,\f,G):=\int_{\f^{-1}(0)}\iota_{Y_{\f}}\Phi^{\tau *}\left(G\text{Vol}_{\IS^{d-1}}\right)=\int_{\f^{-1}(\tau)}\iota_{Y_{\f}}\left(G\text{Vol}_{\IS^{d-1}}\right)
\end{align}
in which case the integral of interest becomes 
\begin{align}
 \label{e:I-in-terms-of-J}
\mathcal{I}(z)=\int_{\min(\f)}^{\max(\f)} \frac{\mathcal{J}(\tau)}{z-\tau}d\tau , 
\end{align}
Our next goal is to prove that $\mathcal{J}$ is actually a real analytic function of $u \in (\min(\f) , \max(\f))$, and even that it has a {\em multivalued holomorphic extension} to $\mathcal{U}\setminus \{\sup (f),\inf (f)\} $, in the sense of Definition~\ref{def:multivalued}, where $\mathcal U$ is a neighborhood of $[\min(\f) , \max(\f)]$ in $\C$. This will then allow to deform the integration contour in~\eqref{e:I-in-terms-of-J} in order to derive the expected analytic properties of $\mathcal{I}$.

In order to prove these analytic properties of $\mathcal{J}$, we split the analysis in two cases. We first prove analytic continuation of $\mathcal{J}$ away from the singular points $\{\min(\f),\max(\f)\}$ in~\S\ref{ss:away-crit} and then prove multivalued holomorphic continuation of $\mathcal{J}$ near $\{\min(\f),\max(\f)\}$ in~\S\ref{s:multivalued-continuation-J}.

\subsection{Analytic properties of $\mathcal{J}$ away from the critical values}\label{ss:away-crit}
\label{s:Anal-away-crit}

The next lemma shows that $\mathcal{J}$ defines an analytic function of $\tau$ with a precise control on the radius of analyticity in terms of the distance to the critical values:
\begin{lemma}\label{l:desintegration}
 Let $0<R<\frac{1}{2\sqrt{d-1}}$ and let $\mathbf{v}\in\mathcal{A}_R(\IS^{d-1})^d$. Then, there exists a constant $C(R,\mathbf{v})>1$ such that, for every $0<R'<R/2$, for every $\omega\in\IS^{d-1}$, for every $k\geq 0$, for every $G\in\mathcal{A}_{R'}(\IS^{d-1})$, and for every $\tau_0\in(\min\f,\max\f)$, 
 $$
 \left|\mathcal{J}^{(k)}(\tau_0,\f,G)\right|\leq \|G\|_{R'}k! (R')^{-k} \left(\frac{C(R,\mathbf{v})}{\min\{|\tau_0-\min\f|,|\tau_0-\max\f|\}}\right)^{k+1} 
 $$
where we recall that $\f(\theta)=\omega\cdot\mathbf{v}(\theta)$.
\end{lemma}
Here we keep track of the dependence in $\omega$ of $\f$ in view of our applications but the proof works as well with more general $\f$ having only two critical points that are \emph{nondegenerate}.
The proof of Lemma~\ref{l:desintegration} relies on the formula~\eqref{e:def-J} together with an application of Lemma~\ref{l:power-lie-derivative} to the vector field $Y=Y_{\f}$ away from $\Crit(\f)$.

\begin{proof} For simplicity, we drop the indices $\omega$ in the function $\f=\f_\omega$.

Given $\tau_0\in(\min\f,\max\f)$, we aim at computing the derivative of order $k$ at $\tau_0$. To do that, we write
$$
\mathcal{J}(\tau_0+h)=\int_{\f^{-1}(\tau_0)}\iota_{Y_{\f}}\Phi^{h *}\left(G\text{Vol}_{\IS^{d-1}}\right),
$$
from which we infer that, for every $k\geq 0$,
$$
\mathcal{J}^{(k)}(\tau_0)=\int_{\f^{-1}(\tau_0)}\iota_{Y_{\f}}\mathcal{L}_{Y_{\f}}^k\left(G\text{Vol}_{\IS^{d-1}}\right).
$$
 We now fix 
\begin{equation}\label{e:security-distance}
0<\varepsilon_1^2<c_0\min\left\{|\tau_0-\min\f|,|\tau_0-\max\f|,R\right\},
\end{equation}
where $0<c_0<1/4$ is smaller than the geometric constant, also called $c_0$, appearing in Lemma~\ref{l:geometric-lemma}. From Lemma~\ref{l:geometric-lemma}, we also deduce that every point in $\f^{-1}(\tau_0)$ is at distance at least $\varepsilon_1$ of $\Crit(\f) = \{-\omega,\omega\}$. 
To prove it, just apply the second inequality of Lemma~\ref{l:geometric-lemma} to $\theta$ in the level set $\f^{-1}(\tau_0)$, $ c_0\vert \max(\f)-\tau_0 \vert\leqslant d_{\mathbb{S}^{d-1}}(\theta,\omega)^2 $ which implies
that
$\varepsilon_1 \leqslant d_{\mathbb{S}^{d-1}}(\theta,\omega) $.

In particular, still according to Lemma~\ref{l:geometric-lemma}, one has that, for every point $\theta$ in $\mathbb{S}^{d-1}$ 
\begin{align}
\label{e:nabla-bdd-below}
d_{\mathbb{S}^{d-1}}(\pm\omega, \theta)\geq\varepsilon_1 \implies |\nabla_{\IS^{d-1}}\f(\theta)|\geq c_0\varepsilon_1 , 
\end{align} for some positive constant $c_0$ that depends only on the convex set $K$ (but not on $\omega$). 

Technically, we need to prove in a quantitative way 
that the $(d-2)$-form
$\iota_{Y_\mathbf{f}}G\text{Vol}_{\mathbb{S}^{d-1}}$ has analytic continuation in the complex domain outside some neighborhood of the critical points $\pm \omega$, and the discussion should be uniform in $\omega\in \mathbb{S}^{d-1}$.
The function 
\begin{align}
\label{e:def-tilde-g}
 \tilde{g} : =|\nabla_{\mathbb{S}^{d-1}}\f|^2
 \end{align}  belongs to the space $\mathcal{A}_{R_1}(\IS^{d-1})$ for all $0<R_1<R$ according to Corollary~\ref{c:normgradient}.
We study the function $\tilde{g}$ in a local chart $\kappa_M:U_M\subset\IS^{d-1}\to B_\R^{d-1}(0,1/2)$, as introduced in Section~\ref{s:complexification}.
Given $y_0\in B_\R^{d-1}(0,1/2) \setminus \kappa_M \big( B_{\IS^{d-1}}(\pm\omega,\varepsilon_1)\big)$, the Taylor expansion of the function $\tilde{g}$ at $y_0$ reads
\begin{align}
\label{e:fMj-decomp}
\tilde{g}(z)=\sum_{\alpha\in\IN^{d-1}}\frac{(z-y_0)^{\alpha}}{\alpha!}\partial^\alpha \tilde{g}(y_0), 
\end{align}
where, for every $0<R_1<R$ and for every $\alpha\in\mathbb{N}^{d-1}$, $|\partial^\alpha \tilde{g}(y_0)|\leq C(R,R_1)\alpha! R_1^{-|\alpha|}\|\mathbf{v}\|_{R}^2$ for some constant $C(R,R_1)>0$ depending only on $0<R_1<R$. In fact, one can verify that, for $0<R_1\leq R/2$, this constant can be chosen as depending only on $R>0$. Thus, we will suppose this from now on and set $C(R,R_1)\leq C(R)<\infty$.
From this, we deduce that, for $|z-y_0|<R_1$, 
\begin{align*}
\sum_{|\alpha|\geq 1}\frac{|z-y_0|^{|\alpha|}}{\alpha!}|\partial^\alpha \tilde{g}(y_0)|
& \leq  C(R)\sum_{|\alpha|\geq 1} (R_1^{-1}|z-y_0|)^{|\alpha|} \|\mathbf{v}\|_R^2   \\
& \leq  C(R) \left( \frac{R_1^{-1}|z-y_0|}{(1-R_1^{-1}|z-y_0|)}\right)^{d-1} \|\mathbf{v}\|_R^2
\end{align*}
where we used the generating series identity
$\sum_{\vert\alpha\vert\geqslant 1, \alpha\in \mathbb{N}^{d-1}}z^{\vert\alpha\vert}=(\frac{z}{1-z})^{d-1}$.
To get the upper bound, $\sum_{|\alpha|\geq 1}\frac{|z-y_0|^{|\alpha|}}{\alpha!}|\partial^\alpha \tilde{g}(y_0)|\leqslant \frac{c_0^2\varepsilon_1^2}{2}$, it is enough that
$$
\frac{R_1^{-1}|z-y_0|}{(1-R_1^{-1}|z-y_0|)}\leqslant \left(\frac{c_0^2\varepsilon_1^2}{ 2C(R)\| \mathbf{v}\|^2_R } \right)^{\frac{1}{d-1}}
$$ so we need that
$$
|z-y_0|\leqslant R_1\frac{\left(c_0^2\varepsilon_1^2 \right)^{\frac{1}{d-1}}}{\left( 2C(R)\| \mathbf{v}\|^2_R  \right)^{\frac{1}{d-1}}+\left(c_0^2\varepsilon_1^2\right)^{\frac{1}{d-1}}} .
$$ 
Hence, there exists
 $\eps_2= \eps_2\big(C(R)\| \mathbf{v}\|_R^2 , c_0\eps_1 \big)\leqslant \varepsilon_1$
such that for all $y_0,z$ such that $0<|z-y_0|\leq\eps_2 R_1$, we have 
\begin{align*}
\sum_{|\alpha|\geq 1}\frac{|z-y_0|^{|\alpha|}}{\alpha!}|\partial^\alpha \tilde{g}(y_0)|
& \leq   \frac{c_0^2\varepsilon_1^2}{2}.
\end{align*}

In particular, for such $z$, recalling~\eqref{e:nabla-bdd-below}, we deduce that $|\tilde{g}(z)-\tilde{g}(y_0)|\leq c_0^2\varepsilon_1^2/2$ and hence $|\tilde{g}(z)|\geq c_0^2\varepsilon_1^2/2$.
As a consequence, the function $z\mapsto1/\tilde{g}(z)$ is holomorphic on the open set 
$$
\mathcal{U}_{R_1,\varepsilon_1,\varepsilon_2}  := \Big\{z\in\mathbb{C}^{d-1}: \exists y_0\in  B_\R^{d-1}(0,1/2) \setminus \kappa_M \big( B_{\IS^{d-1}}(\pm\omega,\varepsilon_1)\big):\ |z-y_0|<  \varepsilon_2 R_1 \Big\}.
$$ 
By construction, this concludes the proof of the following claim with $\tilde{g}$ defined in~\eqref{e:def-tilde-g} (and extended to the complex domain):
\begin{align}
\label{e:intermediate-claim}
\Bigg\|\frac{1}{\tilde{g}}\Bigg\|_{L^{\infty}(\mathcal{U}_{R_1,\varepsilon_1,\varepsilon_2})}\leq \frac{2}{c_0^2\varepsilon_1^2}.
\end{align}
We can now apply Lemma~\ref{l:power-lie-derivative} inside the open set $\mathcal{U}_{R_1,\varepsilon_1,\varepsilon_2}$. The $(d-1)$-form $G\Vol_{\mathbb{S}^{d-1}}$ is real-analytic on $\IS^{d-1}$ and hence admits a holomorphic extension of the form 
$$
\Omega(z,d z)=\omega_1(z)dz_1\wedge\ldots \wedge dz_{d-1},
$$
where $\omega_1$ is a holomorphic complex-valued function and where, according to Lemma~\ref{l:comparison-norm}, it verifies $\|\omega_1\|_{L^\infty(\mathcal{U}_{R'/2})}\leq C_d\|G\|_{R'}$ for some constant $C_d>0$ that depends only on the dimension $d$. Then, one has 
$$
\mathcal{L}_{Y_{\f}}(\Omega)=\left(\mathcal{L}_{\tilde{Y}_{\f}}(\omega_1)+\omega_1(z)\sum_{j=1}^{d-1}\partial_{z_j}\widetilde{Y}_{\f,j}(z)\right)dz_1\wedge\ldots \wedge dz_{d-1}.
$$
Hence, in a local chart, the operator is of the form $\mathcal{L}_{Y_{\f}}+W$ as in Lemma~\ref{l:power-lie-derivative}. Using the Cauchy formula and the above upper bounds, one can verify that $W$ lies in some analytic space as the one appearing in the upper bound of Lemma~\ref{l:power-lie-derivative} with $U=\IS^{d-1}\setminus B(\pm\omega,\varepsilon_1)$ and $\tilde{R}<C(R,\varepsilon_1,\mathbf{v})\varepsilon_1$. The same also holds for the components of the vector field $Y_{\f}$ so that we are in position to apply this lemma. The conclusion follows directly from it.
\end{proof}
In the next sections, given $\mathcal{U} \subset \C$ a bounded open set, we will denote by $\ml{H}^\infty(\mathcal{U})$ the space of bounded holomorphic functions on $\mathcal{U}$, normed by $\|f\|_{\ml{H}^\infty(\mathcal{U})} := \|f\|_{L^\infty(\mathcal{U})}$. From the quantitative estimates of Lemma~\ref{l:desintegration}, we deduce the following.
\begin{lemma}
\label{l:J-far-crit}
Let $R>0$ and assume that $\f\in \ml{A}_{R}(\IS^{d-1})$ is real-valued and has only two nondegenerate critical points.
For any $\delta$ such that $0< \delta< \max(\f)-\min(\f)  $, there exists $\nu(\delta)>0$ such that for all  $G \in \ml{A}_{R}(\IS^{d-1})$ the function $\mathcal{J}$, defined in~\eqref{e:def-J} on $(\min(\f),\max(\f))$, extends as a bounded holomorphic function on $$
\ml{W}_\delta : = (\min(\f)+\delta,\max(\f)-\delta) + i (-\nu(\delta),\nu(\delta)).
$$
Moreover, the map
\begin{equation}
\label{e:continuity-map-G}
\begin{array}{rcl}
\ml{A}_{R}(\IS^{d-1}) & \to & \ml{H}^\infty(\ml{W}_\delta), \\
G & \mapsto&\mathcal{J}
\end{array}
\end{equation}
is linear continuous.
Finally, assume that $\mathcal{J}_\omega$ is defined by~\eqref{e:def-J} for the function $\f_\omega=\omega\cdot\mathbf{v}$, with $\omega \in\mathbb{S}^{d-1}$. 
Then, for any 
\begin{align}
\label{e:delta-range}
0<\delta< \frac{\min_{\omega}\max_{\theta}(\f_\omega(\theta))-\max_\omega\min_\theta(\f_\omega(\theta))}{2} ,
\end{align} 
there exists $\nu(\delta)>0$ such that, for all $\omega \in \IS^{d-1}$, the function $\mathcal{J}_\omega$ defined on $(\min(\f_\omega),\max(\f_\omega))$ extends as a bounded holomorphic function on $\ml{W}_\delta$.  
\end{lemma}
\begin{proof} We only discuss the last part as the first part works analogously except that we consider a general Morse function $\f$ with two critical points on $\mathbb{S}^{d-1}$ (so the bounds of Lemma~\ref{l:desintegration} apply with a constant $C(R,\f)$ instead of $C(R,\mathbf{v})$). Now we fix $\delta$ as in~\eqref{e:delta-range}. 
For every $\tau_0\in   (\min(\f)+\delta,\max(\f)-\delta) $, we have $\delta\leq \min(  \vert \tau_0-\min f\vert,\vert \tau_0-\max(f) \vert)$ and we choose any $\nu(\delta)$ such that  
$\nu(\delta)< \inf(\frac{\delta}{C(R,\mathbf{v})},\delta)$. Then for $z\in\ml{W}_\delta$, we may write $z=\tau_0+ it$ with $\tau_0\in   (\min(\f)+\delta,\max(\f)-\delta)$ and the series 
$$
\mathcal{J}_\omega(z) = \sum_{k \in \N} \frac{(z-\tau_0)^k}{k!}\mathcal{J}_\omega^{(k)}(\tau_0)
$$
 converges according to Lemma~\ref{l:desintegration} with
\begin{eqnarray*}
\Vert \mathcal{J}_\omega \Vert_{\mathcal{H}^\infty(\ml{W}_\delta)}\leqslant \frac{C(R,\mathbf{v})}{\delta} \Vert G\Vert_R\sum_{k=0}^\infty \left(\frac{C(R,\mathbf{v})}{\delta}\nu(\delta)\right)^k = \frac{C(R,\mathbf{v})}{\delta-C(R,\mathbf{v})\nu(\delta)} \Vert G\Vert_R,
\end{eqnarray*}
which concludes the proof of the lemma.
\end{proof}

\subsection{Multivalued holomorphic continuation of $\mathcal{J}$ near $\{\min(\f),\max(\f)\}$}
\label{s:multivalued-continuation-J}
We discuss the continuation of $\mathcal{J}$ near $\min(\f)$, and the proof near $\max(\f)$ follows the same procedure.

\begin{lemma}
\label{l:J-near-crit}
Let $R>0$ and assume that $\f\in \ml{A}_{R}(\IS^{d-1})$ is a Morse function with two critical points.
Then, there exists $\delta_0>0$ such that for all $G \in \ml{A}_{R}(\IS^{d-1})$ there exists a bounded holomorphic function $\mathsf{A}$ in $\{|z|<\delta_0^2\}$ such that the function $\mathcal{J}$ defined in~\eqref{e:def-J} satisfies
\begin{align}
\label{e:J-near-critical}
\mathcal{J}(u)=\big(u -\min(\f)\big)^{\frac{d-3}{2}} \mathsf{A}(u -\min(\f)) ,\quad \text{ for all } u \in (\min(\f),\min(\f)+\delta_0^2)  .
\end{align}
 Moreover, the map 
$$
\begin{array}{rcl}
\ml{A}_{R}(\IS^{d-1}) & \to & \ml{H}^\infty(\{|z|<\delta_0^2\}) , \\
G& \mapsto& \mathsf{A}
\end{array}
$$
is linear continuous. 

Finally, assume that $\mathcal{J}_\omega$ is defined by~\eqref{e:def-J} for the function $\f_\omega=\omega\cdot\mathbf{v}$, with $\omega \in\mathbb{S}^{d-1}$. 
Then there is $\delta_0>0$ such that for all $\omega \in \IS^{d-1}$, for all $G \in \ml{A}_{R}(\IS^{d-1})$, one can find a holomorphic function $\mathsf{A}=\mathsf{A}_{\omega}(G)$ in $\{|z|<\delta_0^2\}$ such that the function $\mathcal{J}_\omega$ satisfies~\eqref{e:J-near-critical}.
\end{lemma}

Note also that, if we denote by $\theta_{\min}$ the point such that $\f(\theta_{\min})=\min\f$ ($\theta_{\min}= -\omega$ for $\f=\omega\cdot\mathbf{v}$), up to some computations, one can express explicitly the value $\mathsf{A}(0)$ (and in principle, also the values of $\d^\alpha\mathsf{A}(0)$ but the computations are cumbersome) in terms of $G(\theta_{\min})$ and the determinant of the Morse chart $\phi$, which can be expressed as a function of $\Hess\f(\theta_{\min})$ (respectively, derivatives of $G$ and $\f$ at $\theta_{\min}$).

\begin{proof}[Proof of Lemma~\ref{l:J-near-crit}]
One more time, we deal with the case of $\f_{\omega}=\omega\cdot\mathbf{v}$ in order to keep track of the dependence in $\omega$. The proof for a fixed $\f$ works analoguously. 
We choose the analytic chart $(\kappa_{-\omega},U_{-\omega})$ centered at $-\omega$ defined in~\eqref{e:def-kappa}-\eqref{e:kappaM}, that is to say $\kappa_{-\omega} : U_{-\omega} \to B^{d-1}(0,1/2)$ such that $U_{-\omega}$ is a sphere cap centered at $-\omega$, $\kappa_{-\omega}(-\omega)=0$, and $\kappa_{-\omega}$ is an analytic diffeomorphism.
Recall from Remark~\ref{r:bas-point} that the norm $\|.\|_R$ depends on the base points $(M_j)_{j=1,\ldots,N}$ used for our analytic atlas. Yet, for two distinct choices $(M_j)_{j=1,\ldots,N}$ and  $(\tilde{M}_j)_{j=1,\ldots,N}$ and for all $0<R_0<R_1<R_2$, one has, according to this remark and to Lemma~\ref{l:comparison-norm},
$$
\|\psi\|_{R_0,(M_j)}\leq N\frac{d!R_1^{d+1}}{(R_1-R_0)^{d+1}}\|\psi\|_{R_1,(\tilde{M}_j)}\leq N^2\frac{(d!)^2R_2^{2d+2}}{(R_1-R_0)^{d+1}(R_2-R_1)^{d+1}}\|\psi\|_{R_2,(M_j)}.
$$
Hence, up to an arbitrarly small loss in the analyticity radius, we can choose a  local chart centered at $-\omega$.

According to the real-analytic Morse Lemma~\ref{theo:morselemmauniform} (using that $\f_\omega$ is a Morse function), there is a local real-analytic diffeomorphism $\phi_\omega$ near $0$ with $\phi_\omega(0)=0$, depending continuously on $\omega$, such that 
$$
\f_\omega\circ\kappa_{-\omega}^{-1} \circ\phi_\omega(x') =\min(\f_\omega)+ |x'|^2 , \qquad |x'|^2=\sum_{i=1}^{d-1}x_i^2 .
$$ 
In particular, since $\kappa_{-\omega}$ can be chosen continuous in $\omega$, as well as $\f_\omega,\phi_\omega$, there exists $\delta_1>0$ such that, for all $\omega\in\IS^{d-1}$, $\f_\omega^{-1}(u) \subset U_{-\omega} $ for all $u \in (\min\f_\omega,\min\f_{\omega}+\delta_1]$.

Changing variables in the integral~\eqref{e:def-J} defining $\mathcal{J}_{\omega}$, we deduce that 
\begin{align*}
\mathcal{J}_\omega(u)  = \int_{\phi_\omega^{-1}\circ\kappa_{-\omega}(\f_\omega^{-1}(u))} (\kappa_{-\omega}^{-1}\circ \phi_\omega)^*\left( \iota_{Y_{\f_\omega}} \left( G \ d\Vol_{\IS^{d-1}} \right)\right)
 , \quad \text{ for } u \in (\min(\f_\omega) , \min(\f_\omega)+\delta_1) .
\end{align*}
Since we have $\phi_\omega^{-1}\circ\kappa_{-\omega}(\f_\omega^{-1}(u)) = \{x'\in \R^{d-1}, \min(\f_\omega) + |x'|^2=u\}$, together with 
$$(\kappa_{-\omega}^{-1}\circ \phi_\omega)^*\left(G \ d\Vol_{\IS^{d-1}}\right) = \mathsf{G}(x')\ d\Vol_{\IR^{d-1}}(x'),\quad d\Vol_{\IR^{d-1}} = dx_1'\wedge \dots \wedge dx_{d-1}',$$ 
for some real-analytic function $\mathsf{G}$ in a neighborhood of $0$ in $\R^{d-1}$, depending implicitly and continuously on the parameter $\omega$, 
this now 
 implies that 
\begin{align*}
\mathcal{J}(u)  = \frac{1}{2} \int_{\{|x'|^2=u-\min(\f_\omega)\} } \mathsf{G}(x')  \iota_{Y_{\f_\omega}} d\Vol_{\R^{d-1}}
 , \quad \text{ for } u \in (\min(\f_\omega) , \min(\f_\omega)+\delta_1) .
\end{align*}
We change again variables to spherical coordinates in $\R^{d-1}$, namely write $x' = \rho \varphi$ with $\rho>0$ (small) and $\varphi \in \IS^{d-2}$. Remarking that $d\Vol_{\R^{d-1}} = \rho^{d-2} d\rho \wedge d\Vol_{\IS^{d-2}}(\varphi)$ and that the current of integration $[|x'|^2=u-\min\f_\omega]$ reads $ \delta_0(\rho-\sqrt{u-\min\f_\omega})d\rho$, we only need to compute the $\partial_\rho$ component of the vector field $Y_{\f_{\omega}}$. By definition of $Y_{\f_{\omega}}$ in~\eqref{e:def-champ-Yf} (see \S\ref{ss:level-lines}), one has $d\rho(Y_{\f_{\omega}})=\rho^{-1}$, this finally yields

\begin{align*}
\mathcal{J}(u)  = \frac{1}{2}  \int_{\IS^{d-2}} \big( \mathsf{G}(\rho \varphi)\rho^{d-3}  \big)\big|_{\{\rho^2=u-\min(\f_\omega)\}}d\Vol_{\IS^{d-2}}(\varphi)
 , \quad \text{ for } u \in (\min(\f_\omega) , \min(\f_\omega)+\delta_1) .
\end{align*}
Still for (real) $u \in (\min(\f_\omega) , \min(\f_\omega)+\delta_1)$, we may compute this integral by Taylor expanding the real-analytic function $\mathsf{G}$ to obtain  
\begin{align*}
\mathcal{J}_\omega(u)&= \frac{1}{2}\sum_{\alpha\in \mathbb{N}^{d-1}} \frac{\d^{\alpha}\mathsf{G}(0)}{\alpha !} \left(\int_{\mathbb{S}^{d-2}} \varphi^\alpha  d\Vol_{\IS^{d-2}}(\varphi) \right)   (u -\min(\f_\omega))^{\frac{\vert \alpha\vert+ d-3}{2}}\\
&=\frac{1}{2}\sum_{\alpha\in \mathbb{N}^{d-1}} \frac{\d^{\alpha}\mathsf{G}(0)}{\alpha !} c_\alpha  (u -\min(\f_\omega))^{\frac{\vert \alpha\vert+ d-3}{2}},
\end{align*}
where we have set $$c_\alpha := \int_{\mathbb{S}^{d-2}} \varphi^\alpha  d\Vol_{\IS^{d-2}}(\varphi) , \quad \alpha\in \mathbb{N}^{d-1}.$$
We also notice that $c_\alpha \in \R$, $\vert c_\alpha\vert\leq 1$ for all multiindex $\alpha\in \mathbb{N}^{d-1}$, and $c_{\alpha}$ vanishes when one of the $\alpha_i$ in the multiindex $(\alpha_1,\dots,\alpha_{d-1})$ is odd. We hence obtain, for all $u \in (\min(\f_\omega) ,\min(\f_\omega)+\delta_1)$, 
\begin{align*}
\mathcal{J}_\omega(u)=\frac{1}{2}\sum_{\beta\in \N^{d-1}} c_{2\beta} \frac{\d^{2\beta}\mathsf{G}(0)}{(2\beta) !}  \big(u -\min(\f_\omega)\big)^{\frac{d-3}{2}+\vert\beta\vert} .
\end{align*}
We conclude that if $R>0$ is the convergence radius of $\mathsf{G}$ at $0$, then 
\begin{align*}
\mathcal{J}_\omega(u)=\big(u -\min(\f_\omega)\big)^{\frac{d-3}{2}} \mathsf{A}(u -\min(\f_\omega)), \quad \text{ with }\quad \mathsf{A}(z) = \frac{1}{2}\sum_{\beta\in \N^{d-1}} c_{2\beta} \frac{\d^{2\beta}\mathsf{G}(0)}{(2\beta) !} z^{\vert\beta\vert} ,
\end{align*}
is a holomorphic function near zero with convergence radius $R^2$. By construction, the $\mathcal{H}^\infty$ norm of $\mathsf{A}$ is bounded in terms of the $\mathcal{A}_R$ norm of $G$.
\end{proof}

For the moment, combining the two analytic extensions given by Lemmata~\ref{l:J-far-crit} and~\ref{l:J-near-crit}, we control the analytic continuation of $\mathcal{J}$ in a region of the form 
$$
(\min(\f_\omega)+\delta,\max(\f_\omega)-\delta)+ i(-\nu(\delta),\nu(\delta))\cup \{|z-\min(\f_\omega)|<\delta_0^2\}\cup\{|z-\max(\f_\omega)|<\delta_0^2\}
$$ 
for any $0<\delta<\min_\omega\max_\theta(\f_\omega(\theta))-\max_\omega\min_\theta(\f_\omega(\theta))$ and for some $\delta_0$ depending on both $\mathbf{v}$ and $G$. Hence, if we define the open rectangular neighborhood $\mathcal{R}_{\nu}(\f)$ of size $\nu$ around the interval $[\min(\f) ,\max(\f)]$ by  
\begin{align}
\label{e:def-Unu}
\mathcal{R}_{\nu}(\f) := \big( \min(\f)- \nu ,\max(\f) + \nu\big) + i (-\nu,\nu) ,  
\end{align}
then the set $\mathcal{R}_{\nu}(\f)$ is contained in the above region (where $\mathcal{J}$ is analytically continued up to singular branches) for $\nu>0$ small enough.
As a corollary of Lemmata~\ref{l:J-far-crit} and~\ref{l:J-near-crit}, we obtain the following description of the structure of $\mathcal{J}$:

\begin{corollary}
\label{c:equality-J}
Let $R>0$ and assume that $\f\in \ml{A}_{R}(\IS^{d-1})$ is a Morse function with two critical points.
Then, there exists $\nu >0$ such that for all $G \in \ml{A}_{R}(\IS^{d-1})$ there exists a bounded holomorphic function $\mathcal{H}$ on $\mathcal{R}_{\nu}(\f)$ (defined in~\eqref{e:def-Unu}) such that
\begin{align}
\label{e:J-holom}
\mathcal{J}(u) =  \big(u -\min(\f)\big)^{\frac{d-3}{2}} \big(\max(\f) -u\big)^{\frac{d-3}{2}} \mathcal{H}(u) ,\quad \text{ for all } u \in \big(\min(\f),\max(\f)\big) ,
\end{align}
 Moreover, the map 
$$
\begin{array}{rcl}
\ml{A}_{R}(\IS^{d-1}) & \to & \ml{H}^\infty(\mathcal{R}_{\nu}(\f)) , \\
G& \mapsto&  \mathcal{H}
\end{array}
$$
is linear continuous. 

Finally, assume that $\mathcal{J}_\omega$ is defined by~\eqref{e:def-J} for the function $\f_\omega=\omega\cdot\mathbf{v}$ with $\omega \in \IS^{d-1}$. Then there is $\nu>0$ such that for all $\omega \in \IS^{d-1}$, for all $G \in \ml{A}_{R}(\IS^{d-1})$, the above statements hold for $\mathcal{J}_\omega$ in $\mathcal{R}_{\nu}(\f_\omega)$.
\end{corollary}

\begin{proof}
We consider the function 
$$
\mathcal{H}(u) : =\mathcal{J}(u) \big(u -\min(\f)\big)^{-\frac{d-3}{2}} \big(\max(\f) -u\big)^{-\frac{d-3}{2}} , \quad  u \in \big(\min(\f),\max(\f)\big) .
$$
According to Lemma~\ref{l:J-near-crit} and~\eqref{e:J-near-critical}, we have 
\begin{align*}
\mathcal{H}(u) =\mathcal{J}(u) \big(\max(\f) -u\big)^{-\frac{d-3}{2}} \tilde{\mathsf{A}}(u),\quad \text{ for all } u \in (\min(\f),\min(\f)+\delta_0^2) ,
\end{align*} where 
$\tilde{\mathsf{A}}$ is a holomorphic function in $\{|z-\min(\f)|<\delta_0^2\}$. This implies that $\mathcal{H}$ extends holomorphically to 
$\{|z-\min(\f)|<\delta_0^2\}$ 
and the analogue of Lemma~\ref{l:J-near-crit} near $\max(\f)$ implies that $\mathcal{H}$ extends also holomorphically to 
$\{|z-\max(\f)|<\delta_0^2\}$. 
Picking $\delta<\delta_0^2$ small enough, Lemma~\ref{l:J-far-crit} directly implies that $\mathcal{H}$ extends holomorphically to $(\min(\f)+\delta, \max(\f)-\delta) + i (-\nu(\delta),\nu(\delta))$. These three sets cover  $\mathcal{R}_{\nu}(\f)$ for $\nu>0$ depending only on $\delta_0$ and $\nu(\delta)$ and it concludes the proof of the corollary.
\end{proof}

\subsection{From continuation of $\mathcal{J}$ to continuation of $\mathcal{I}$}\label{ss:I-to-J-2}

We now turn to the description of the function $\mathcal{I}_{(\ell)} =\mathcal{I}_{(\ell)}(\cdot ,\f,G) $, defined in~\eqref{e:def-Ibis} and rewritten in~\eqref{link-I-ell} in terms of $\mathcal{I}$ which was itself expressed in~\eqref{e:I-in-terms-of-J} as an integral of $\mathcal{J}=\mathcal{J}(\cdot ,\f,G)$.
We start with the first following elementary fact.
\begin{lemma}
\label{l:I-holo-domain}
The function $z\mapsto \mathcal{I}_{(\ell)}(z,\f,G)$, initially defined for $\Re(z)>0$ in~\eqref{e:def-Ibis}, extends (uniquely) as a holomorphic function to $\C \setminus [i \min(\f), i \max(\f)]$ satisfying
$$
\left| \mathcal{I}_{(\ell)}(z,\f,G) \right| \leq \frac{\|G\|_{L^1(\IS^{d-1})}}{\dist\left( z,[i\min(\f),i \max(\f)] \right)^{\ell+1}} .
$$

\end{lemma}
\begin{proof}
This follows from the expression of $\mathcal{I}_{(\ell)}$ in~\eqref{e:def-Ibis}-\eqref{e:def-II}-\eqref{link-I-ell} and holomorphy under the integral.
\end{proof}

We can now state the main continuation result concerning the holomorphic properties of $\mathcal{I}$. We recall that, in application to Poincaré series, the function $\mathcal{I}$ is evaluated at the point $z/i$ with $\Re(z)>0$. We hence study the continuation properties of $\mathcal{I}$ starting from the region $\Im(z)<0$.

\begin{proposition}
\label{p:main-I}
  Let $R>0$ and assume that $\f\in \ml{A}_{R}(\IS^{d-1})$ is a Morse function with two critical points. Let $\ln$ be the standard determination of the logarithm on $\C\setminus\R_-$.

  With $\nu>0$ (and $\mathcal{R}_{\nu}(\f)$ defined in~\eqref{e:def-Unu}) and $\mathcal{H}=\mathcal{H}(G)$ given by Corollary~\ref{c:equality-J}, there exist a holomorphic function $\mathcal{K}$ on $\mathcal{R}_{\nu}(\f)$ (depending also continuously on $G$) such that
\begin{align}
\label{e:I-z-holo-decompo}
\mathcal{I}(z) = \mathcal{H}(z) \mathbf{F}_{\frac{d-3}{2}}(z) + \mathcal{K}(z) , \quad \text{ for all } z \in \mathcal{R}_{\nu}(\f) , \quad \Im(z)< 0 ,
\end{align}
where
  \begin{align}
  \label{e:F-chelou-utile-1}
  \mathbf{F}_{\frac{d-3}{2}}(z) & = \big(z -\min(\f)\big)^{\frac{d-3}{2}} \big(\max(\f) -z\big)^{\frac{d-3}{2}} \big(  \ln(iz-i\min(\f)) - \ln(iz-i\max(\f)) \big)   , \\ & \quad \text{ for }d\geq 3 \text{ odd}, \nonumber \\
    \label{e:F-chelou-utile-2}
  \mathbf{F}_{\frac{d-3}{2}}(z)&  =(-1)^{\frac{d-2}{2}}  (-i)^{d-3} \pi \big(iz -i\min(\f)\big)^{\frac{d-3}{2}} \big(iz-i\max(\f) \big)^{\frac{d-3}{2}}    , \quad \text{ for }d\geq 2 \text{ even} .
  \end{align}

If $\mathcal{J}_\omega$ is defined by~\eqref{e:def-J} for the function $\f_\omega=\omega\cdot\mathbf{v}$, with $\omega \in \IS^{d-1}$. 
Then, there is $\nu>0$ (given by Corollary~\ref{c:equality-J}) such that, for all $\omega \in \IS^{d-1}$ and for all $G \in \ml{A}_{R}(\IS^{d-1})$, the above statements hold.

 Finally, the map 
$$
\begin{array}{rcl}
\ml{A}_{R}(\IS^{d-1}) & \to & \ml{H}^\infty(\mathcal{R}_{\nu}(\f)) \times \ml{H}^\infty(\mathcal{R}_{\nu}(\f)) , \\
G& \mapsto& ( \mathcal{H} ,  \mathcal{K} ) 
\end{array}
$$
is linear continuous. 

  \end{proposition}

 \begin{remark}
\label{e:extension-box}
As a direct corollary of Proposition~\ref{p:main-I}, together with the properties of the singular functions $\mathbf{F}_{\frac{d-3}{2}}$ involved, we may now extend the integral $\mathcal{I}(z)$ either as a holomorphic function to a subset of the complex plane with  a cut at any curve $\Gamma\subset \overline{\ml{R}_{\nu}(\f)}$ joining $\min(\f)$ and $\max(\f)$ and homotopic to the segment $[\min(\f) , \max(\f)]$, or as a multivalued holomorphic function in the sense of Appendix~\ref{a:branched}. This extension part follows now directly from Proposition~\ref{p:main-I}  combined with the last three items of Lemma~\ref{l:laplace-bessel} (see also Remark~\ref{r:lemme-pret-a-utiliser}) together with the link between $\mathbf{F}_{\frac{d-3}{2}}$ and $\mathsf{F}_{\frac{d-3}{2}}$ given by 
\begin{align*}
\mathbf{F}_{\frac{d-3}{2}}(z) = \left( \frac{\max(\f)-\min(\f)}{2}\right)^{d-3} \mathsf{F}_{\frac{d-3}{2}}\left(  \frac{2}{\max(\f)-\min(\f)} z -  \frac{\max(\f)+\min(\f)}{\max(\f)-\min(\f)} \right)  + \mathsf{H}_d(z), 
 \end{align*}
 where $ \mathsf{H}_d$ is an entire function.
\end{remark}

\begin{proof}
Recalling the expression~\eqref{e:J-holom} of $\mathcal{J}$ in Corollary~\ref{c:equality-J}, together with the expression of $\mathcal{I}$ in terms of $\mathcal{J}$ in~\eqref{e:I-in-terms-of-J}, we obtain
\begin{align*}
\mathcal{I}(z) =\int_{\min(\f)}^{\max(\f)} \frac{ \big(u -\min(\f)\big)^{\frac{d-3}{2}} \big(\max(\f) -u\big)^{\frac{d-3}{2}} \mathcal{H}(u)}{z-u}du  ,
\end{align*}
where $\mathcal{H}$ is a holomorphic in $\mathcal{R}_\nu(\f)$. We may thus write 
\begin{align}
\label{scrounch-scrou}
\mathcal{H}(u)= \mathcal{H}(z) + (u-z)\mathbf{R}(u,z) , \quad \mathbf{R}(u,z) = \int_0^1  \mathcal{H}'(su+(1-s)z) ds .
\end{align}
Since $\mathcal{R}_\nu(\f)$ is a convex set, the function $\mathbf{R}$ is holomorphic in $\mathcal{R}_\nu(\f)\times\mathcal{R}_\nu(\f)$ according to holomorphy under the integral.
We have now obtained
\begin{align}
\label{e:almost-decomp}
\mathcal{I}(z) =\mathcal{H}(z) \mathcal{F}_\frac{d-3}{2}(z) +  \mathcal{K}_{d,0}(z)  ,
\end{align}
with 
\begin{align*}
 \mathcal{K}_{d,0}(z) & := 
- \int_{\min(\f)}^{\max(\f)} \big(u -\min(\f)\big)^{\frac{d-3}{2}} \big(\max(\f) -u\big)^{\frac{d-3}{2}} \mathbf{R}(u,z)du , \\
\mathcal{F}_\frac{d-3}{2}(z)&  := \int_{\min(\f)}^{\max(\f)} \frac{\big(u -\min(\f)\big)^{\frac{d-3}{2}} \big(\max(\f) -u\big)^{\frac{d-3}{2}}}{z-u}du  .
\end{align*}
The function $\mathcal{K}_{d,0}$ is a holomorphic function in $z \in \mathcal{R}_\nu(\f)$ according to holomorphy under the integral since $\frac{d-3}{2}\geq -\frac12 > -1$.
Concerning $\mathcal{F}_\gamma(z)$, making the affine change of variable $t\mapsto u := \frac{\max(\f)-\min(\f)}{2} t +\frac{\max(\f)+\min(\f)}{2}$, we obtain 
\begin{align*}
& \mathcal{F}_\gamma(z) = \left( \frac{\max(\f)-\min(\f)}{2}\right)^{2\gamma} \mathsf{F}_\gamma\left(  \frac{2}{\max(\f)-\min(\f)} z -  \frac{\max(\f)+\min(\f)}{\max(\f)-\min(\f)} \right) , 
 \end{align*}
 with $\mathsf{F}_\gamma$ defined in~\eqref{e:def-Fgamma}.
According to Lemma~\ref{l:laplace-bessel}, under the form of Remark~\ref{r:lemme-pret-a-utiliser}, we deduce that for $ \Im(z)<0$, 
  \begin{align*}
  \mathcal{F}_{\frac{d-3}{2}}(z) & = \big(z -\min(\f)\big)^{\frac{d-3}{2}} \big(\max(\f) -z\big)^{\frac{d-3}{2}} \big(  \ln(iz-i\min(\f)) - \ln(iz-i\max(\f)) \big)  + P_d(z)  ,\\ 
  & \quad \text{ for }d\geq 3 \text{ odd}, \\
  \mathcal{F}_{\frac{d-3}{2}}(z)&  = (-1)^{\frac{d-2}{2}}\pi  (-i)^{d-3} \big(iz -i\min(\f)\big)^{\frac{d-3}{2}} \big(iz-i\max(\f) \big)^{\frac{d-3}{2}} + P_d(z)  , \quad \text{ for }d\geq 2 \text{ even},
  \end{align*}
  where $P_d$ is a polynomial.   Coming back to~\eqref{e:almost-decomp}, this concludes the proof of the first part of the proposition.
  Finally, the continuity statement follows from the continuity statement in Corollary~\ref{c:equality-J} together with the explicit expression of $\mathbf{R}$ in terms of $\ml{H}$ in~\eqref{scrounch-scrou}.  
\end{proof}

\subsection{Analytic continuation of multiplication operators}
\label{ss:multiplication}
As a first byproduct of Proposition~\ref{p:main-I}, we deduce a proof of Theorem~\ref{t:maintheo-multiplication} stating continuation properties for the resolvent of multiplication operators by analytic Morse functions having two critical points. 
This functional point of view is not necessary for application  to the Poincar\'e series described here but we think the result is of independent interest. 
\begin{theorem}
\label{t:multiplicationoperators}
  Let $R>0$, assume that $\f\in \ml{A}_{R}(\IS^{d-1})$ is a Morse function with two critical points and denote by $\mathbf{m}_\f : L^2(\IS^{d-1}) \to L^2(\IS^{d-1})$ the operator $\psi \mapsto \f \psi$.
  Then, there is $\nu>0$ and two families of operators $\left(\mathcal{H}(z) , \mathcal{K}(z)\right)_{z \in \mathcal{R}_{\nu}(\f)} \in \ml{L}(\ml{A}_R , \ml{A}_R')^{2\mathcal{R}_{\nu}(\f)}$
    such that the bilinear map $$(\varphi_1,\varphi_2) \mapsto \left\langle \varphi_2 ,(z-\mathbf{m}_\f)^{-1}\varphi_1\right\rangle , $$ defined on $L^2(\IS^{d-1}) \times L^2(\IS^{d-1})$ for all $z \in \C \setminus [\min(\f),\max(\f)]$, extends when $\varphi_1,\varphi_2\in \mathcal{A}_R^2$ as 
\begin{align*}
 \left\langle (z-\mathbf{m}_\f)^{-1}\varphi_2 ,\varphi_1\right\rangle = \mathbf{F}_{\frac{d-3}{2}}(z)\left\langle \varphi_2 , \mathcal{H}(z)\varphi_1\right\rangle  + \left\langle \varphi_2 , \mathcal{K}(z)\varphi_1\right\rangle , \quad \text{ for all } z \in \mathcal{R}_{\nu}(\f)\cap\{  \Im(z)< 0\} ,
\end{align*}
where $\mathbf{F}_{\frac{d-3}{2}}$ is defined in Proposition~\ref{p:main-I}
 and where $z \mapsto \left\langle \varphi_2 ,\mathcal{H}(z)\varphi_1\right\rangle$ and $z \mapsto\left\langle \varphi_2 ,\mathcal{K}(z)\varphi_1\right\rangle  $ are holomorphic on $\mathcal{R}_\nu(\f)$, with
$$
\sup_{z\in \mathcal{R}_\nu(\f)} \left(\vert  \left\langle \varphi_2 ,\mathcal{H}(z)\varphi_1\right\rangle \vert +\vert  \left\langle \varphi_2 ,\mathcal{K}(z)\varphi_1\right\rangle \vert\right)  \leq C_{R,\f} \Vert \varphi_1 \Vert_{\mathcal{A}_{R}}  \Vert \varphi_2 \Vert_{\mathcal{A}_{R}}.
$$
\end{theorem}
From the statement of Theorem~\ref{t:multiplicationoperators}, together with the properties of the functions $\mathbf{F}_{\frac{d-3}{2}}$, described precisely in Lemma~\ref{l:laplace-bessel} in  Appendix~\ref{a:branched} (the link between $\mathbf{F}_{\frac{d-3}{2}}$ and $\mathsf{F}_{\frac{d-3}{2}}$ is given in Remark~\ref{e:extension-box}) we may deduce several extension results (such as holomorphic function with cuts or multivalued holomorphic function on the set $\mathcal{R}_{\nu}(\f)\setminus\{\min(\f),\max(\f)\}$). In particular, Lemma~\ref{l:laplace-bessel}  implies that the resolvent 
$(z-\textbf{m}_{\f})^{-1}:L^2(\IS^{d-1})\rightarrow L^2(\IS^{d-1})$ extends from $\{z\in\mathcal{R}_\nu(\f):\ \Im(z)<0\}$ as a multivalued holomorphic function in $\mathcal{R}_{\nu}(\f)\setminus\{\min(\f),\max(\f)\}$ with values in continuous linear mappings from $\mathcal{A}_R$ to its dual $\mathcal{A}_R'$. Finally, Theorem~\ref{t:maintheo-multiplication} is a reformulation of Theorem~\ref{t:multiplicationoperators} in which we continue the resolvent starting from $\Im(z)>0$ instead of $\Im(z)<0$  (but the change $z\to \bar{z}$ does not play any role with respect to $\mathbf{m}_\f$). 
\begin{proof} 
We notice that 
$$\left\langle (z-\mathbf{m}_\f)^{-1}\varphi_2 ,\varphi_1\right\rangle=\mathcal{I}(z,\f,\varphi_2\overline{\varphi_1}).$$
As a consequence, the proof follows from Proposition~\ref{p:main-I} applied to $\f$ and $G=\varphi_2\overline{\varphi_1}$ (together with Lemma~\ref{c:product}). 
\end{proof}

\subsection{Analytic continuation of $\mathcal{I}_{(\ell)}$ and $I^{(\ell)}_{\lambda,\omega}$}\label{ss:Iell-and-Iell}

After having analyzed the properties of $\mathcal{I}$, we can now come back to the description of the analytic functions $\mathcal{I}_{(\ell)}$ defined in~\eqref{e:def-Ibis} and to $I^{(\ell)}_{\lambda,\omega}$ (the difference being a holomorphic function according to~\eqref{e:decomp-I-I}, see also the discussion right after this equation). Ultimately, this will give us the proof of Proposition~\ref{p:low-freq} (hence the expected analytic properties for $\mathcal{P}_{K_0}^{<N}$). We begin with the following lemma.

\begin{lemma}
\label{l:mathcal-I}
Let $\mathcal{I}_{(\ell)}(z)=\mathcal{I}_{(\ell)}(z,\f_\omega,G)$ be defined on $\{\Re(z)>0\}$ by~\eqref{e:def-Ibis} for the function $\f_\omega=\omega\cdot\mathbf{v}$, with $\omega \in \IS^{d-1}$. 
Assume that there is $R>0$ such that $(\mathbf{v},G)\in\mathcal{A}_R(\IS^{d-1})^{d+1}$.

Then, there is $\nu>0$ (given by Corollary~\ref{c:equality-J}) such that for all $\omega \in \IS^{d-1}$, for all $G \in \ml{A}_{R}(\IS^{d-1})$ and $\ell \in \{0,\dots,d-1\}$,
 there exist holomorphic functions $\mathcal{H}_{\ell,\omega,G},\mathcal{K}_{\ell,\omega,G}$ on $i \mathcal{R}_{\nu}(\f_\omega)$ such that, for all $z \in i \mathcal{R}_{\nu}(\f_\omega)\cap\{\Re(z)> 0\}$,
  \begin{align}
\label{e:I-z-holo-decompo-Iell0}
\mathcal{I}_{(\ell)} (z) =\mathcal{H}_{\ell,\omega,G} (z) 
\Big(\big(z - i \min(\f_\omega)\big)\Big)^{\frac{d-3}{2}-\ell}\Big(\big( z-i\max(\f_\omega) \big) \Big)^{\frac{d-3}{2}-\ell} 
+ \mathcal{K}_{\ell,\omega,G} (z) , 
\end{align}
 if $d\geq 2$ is even, and, for all $z \in i \mathcal{R}_{\nu}(\f_\omega) \cap\{ \Re(z)> 0\}$,
  \begin{align*}
\mathcal{I}_{(\ell)} (z) &=\mathcal{H}_{\ell,\omega,G} (z) \Big(\big(z - i \min(\f_\omega)\big)\big( i\max(\f_\omega) - z\big) \Big)^{\frac{d-3}{2}-\ell} \left(\ln\Big(z - i \min(\f_\omega)\Big)-\ln\Big( z-i\max(\f_\omega) \Big)\right) \\
 & \quad  + \mathcal{K}_{\ell,\omega,G} (z) , \quad \text{ if } \ell \leq \frac{d-3}{2} ,\\
\mathcal{I}_{(\ell)} (z) &=\mathcal{H}_{\ell,\omega,G} (z) \Big(\big(z - i \min(\f_\omega)\big)\big( i\max(\f_\omega) - z\big) \Big)^{\frac{d-3}{2}-\ell}\\
& \quad  + \mathcal{K}_{\ell,\omega,G} (z) \left(\ln\Big(z - i \min(\f_\omega)\Big)-\ln\Big( z-i\max(\f_\omega) \Big)\right)  , \quad \text{ if } \ell > \frac{d-3}{2},
\end{align*} 
 if $d\geq 3$ is odd.
\end{lemma}
Note that in the case $d$ even, all powers $\frac{d-3}{2}-\ell$ are half integers (square root singularities) whereas in case $d$ odd, all powers $\frac{d-3}{2}-\ell$ are nonnegative integers if $\ell \leq \frac{d-3}{2}$ and negative integers if $\ell > \frac{d-3}{2}$. Holomorphic continuation results (as a holomorphic function with a cut at a curve $\Gamma\subset i \overline{\mathcal{R}_{\nu}(\f_\omega)}$ linking the two singularities, or as a multivalued holomorphic Nilsson function on the set $i \mathcal{R}_{\nu}(\f_\omega)$) follow from Remark~\ref{e:extension-box}, itself relying on Lemma~\ref{l:laplace-bessel} in Appendix~\ref{a:branched}. This Lemma when specified for a fixed function $\f$ implies Theorem~\ref{t:extension-integral}.

\begin{proof}
The expression of $\mathcal{I}_{(\ell)}$ in terms of $\mathcal{I}$ is given in Lemma~\ref{l:Iell-I} by Equation~\eqref{link-I-ell} and the result will thus follow from the analytic properties of $\mathcal{I}$ obtained in Proposition~\ref{p:main-I} together with Remark~\ref{e:extension-box} and Lemma~\ref{l:laplace-bessel} in Appendix~\ref{a:branched}.

Assume $d\geq 2$ is even. Then, we notice from the expression of $\mathbf{F}_{\frac{d-3}{2}}$ in~\eqref{e:F-chelou-utile-2} that for any $\ell \in \N$, we have
$$
\d^\ell \mathbf{F}_{\frac{d-3}{2}} = h_{\ell} \mathbf{F}_{\frac{d-3}{2}-\ell} , 
$$
for some polynomial function $h_{\ell}$. We thus deduce from~\eqref{e:I-z-holo-decompo} in Proposition~\ref{p:main-I} that 
$$
\d^\ell \mathcal{I} = H_{\ell} \mathbf{F}_{\frac{d-3}{2}-\ell}  + K_{\ell}  ,
$$
where $H_{\ell},K_{\ell}$ are holomorphic in $\mathcal{R}_{\nu}(\f_\omega)$. 
Recalling the expression of $\mathcal{I}_{(\ell)}$ in terms of $\mathcal{I}$ in Equation~\eqref{link-I-ell} implies~\eqref{e:I-z-holo-decompo-Iell0} in this case.

Assume now that $d\geq 3$ is odd. We first notice that, if we set
$$
f_k(x) = x^k\ln(x) , 
$$
then one can check by induction that for any $k , \ell \in \N$, we have
\begin{align*}
f_k^{(\ell)}(x) &= \frac{k!}{(k-\ell)!}x^{k-\ell}\ln(x) + \alpha_{k,\ell} x^{k-\ell}, \quad \text{ for } \ell \leq k, \text{ with } \alpha_{k,\ell} \in \N  ,\\
f_k^{(\ell)}(x) &= (-1)^{\ell-k-1}k!(\ell-k-1)! x^{k-\ell}, \quad \text{ for } \ell > k .
\end{align*}
We then notice that the function $\mathbf{F}_{\frac{d-3}{2}}$ in~\eqref{e:F-chelou-utile-1} can be rewritten for $\Im(z)>0$ as
$$
  \mathbf{F}_{\frac{d-3}{2}}(z) = f_k \circ h(z) , \quad \text{ with }  h(z) = \big(z -\min(\f_\omega)\big) \big(\max(\f_\omega) -z\big), \quad k=\frac{d-3}{2}\in \N
$$ 
Writing $(f_k \circ h )^{(\ell)}= \sum_{j=0}^\ell h_j f_k^{(j)}\circ h$ for some entire functions $h_j$, we see from the above explicit expression of $f_k^{(j)}$ that 
\begin{align*}
  \mathbf{F}_{\frac{d-3}{2}}^{(\ell)} & = (f_k \circ h )^{(\ell)}= h_{\ell,0} h^{k-\ell}\ln(h) +  h_{\ell,1} \quad \text{ for } \ell \leq k=\frac{d-3}{2}  , \\
    \mathbf{F}_{\frac{d-3}{2}}^{(\ell)} & = (f_k \circ h )^{(\ell)}= h_{\ell,0} h^{k-\ell} +  h_{\ell,1} \ln(h)  \quad \text{ for } \ell > k=\frac{d-3}{2} ,
\end{align*}
where $h_{\ell,0},h_{\ell,1}$ are entire functions. Finally, using again Proposition~\ref{p:main-I} together with the expression of $\mathcal{I}_{(\ell)}$ in terms of $\mathcal{I}$ in Equation~\eqref{link-I-ell} concludes the proof in the case $d$ odd.
\end{proof}

We may then come back to the description of $I^{(\ell)}_{\lambda,\omega}$.
 The  expression of $I^{(\ell)}_{\lambda,\omega}$ in terms of $\mathcal{I}_{(\ell)}$ is given in~\eqref{e:decomp-I-I}, where $\f_\omega=\omega\cdot \mathbf{v}$ and $G=e^{i\lambda X}F$ (see~\eqref{e:values-f-G}).
 Recalling that the second term in the right-hand side of~\eqref{e:decomp-I-I} is an entire function, we directly obtain the following corollary of  Lemma~\ref{l:mathcal-I}.

\begin{corollary}
\label{c:I-}
 Let $R>0$ and assume $\mathbf{v}\in \ml{A}_{R}(\IS^{d-1})^d$.
Then, there is $\nu>0$ (given by Corollary~\ref{c:equality-J}) such that for all $\omega \in \IS^{d-1}$, $F \in \ml{A}_{R}(\IS^{d-1})$, $X\in \ml{A}_{R}(\IS^{d-1})$, $\ell \in \{0,\dots,d-1\}$, and $\lambda >0$,
 there exist holomorphic functions $\mathcal{H} = \mathcal{H}_{\ell,\omega,F,X,\lambda},\mathcal{K} = \mathcal{K}_{\ell,\omega,F,X,\lambda}$ on 
 \begin{align}
\label{e:def-Vnu}
 i \mathcal{R}_{\nu}(\f_\omega) = 
i (\omega\cdot  \mathbf{v}(-\omega) - \nu , \omega\cdot  \mathbf{v}(\omega)+\nu) + (-\nu,\nu)
\end{align}
 such that  for all $z \in i\mathcal{R}_{\nu}(\f_\omega) \cap\{ \Re(z)> 0\}$,
  \begin{align}
\label{e:I-z-holo-decompo-Iell}
 I^{(\ell)}_{\lambda,\omega}(z, F,X)  =\mathcal{H} (z) 
\Big(\big(z -i \omega\cdot\mathbf{v}(-\omega)\big)\Big)^{\frac{d-3}{2}-\ell}\Big(\big(z- i\omega\cdot\mathbf{v}(\omega)\big) \Big)^{\frac{d-3}{2}-\ell} 
+ \mathcal{K} (z) , 
\end{align}
 if $d\geq 2$ is even, and for all $z \in  i\mathcal{R}_{\nu}(\f_\omega) \cap\{ \Re(z)> 0\}$,
  \begin{align*}
 I^{(\ell)}_{\lambda,\omega}(z, F,X) &=\mathcal{H} (z) \Big(\big(z - i\omega\cdot\mathbf{v}(-\omega)\big)\big( i\omega\cdot\mathbf{v}(\omega)- z\big) \Big)^{\frac{d-3}{2}-\ell} \left(\ln\Big(z -i \omega\cdot\mathbf{v}(-\omega)\Big)-\ln\Big(z -i\omega\cdot\mathbf{v}(\omega) \Big)\right) \\
 & \quad  + \mathcal{K}  (z) , \quad \text{ if } \ell \leq \frac{d-3}{2} ,\\
 I^{(\ell)}_{\lambda,\omega}(z, F,X) &=\mathcal{H}  (z) \Big(\big(z -i \omega\cdot\mathbf{v}(-\omega)\big)\big( i\omega\cdot\mathbf{v}(\omega) - z\big) \Big)^{\frac{d-3}{2}-\ell}\\
& \quad  + \mathcal{K} (z) \left(\ln\Big(z -i \omega\cdot\mathbf{v}(-\omega)\Big)-\ln\Big(z -i\omega\cdot\mathbf{v}(\omega) \Big)\right)  , \quad \text{ if } \ell > \frac{d-3}{2},
\end{align*} 
 if $d\geq 3$ is odd.

 In particular, given any curve $\Gamma_\omega \in C^1_{pw} \big([0,1];  i\mathcal{R}_{\nu}(\f_\omega) \cap\{ \Re(z) \leq 0\}\big)$ be such that $\Gamma_\omega (0) = i \omega\cdot\mathbf{v}(-\omega)$ and $\Gamma_\omega(1)=i \omega\cdot\mathbf{v}(\omega)$ and $\Gamma_\omega([0,1])$ is {\em strictly homotopic}  to $ i\big[\omega\cdot\mathbf{v}(-\omega),  \omega\cdot\mathbf{v}(\omega)\big]$ in $i\mathcal{R}_{\nu}(\f_\omega) \cap\{ \Re(z) \leq 0\}$, the function $I^{(\ell)}_{\lambda,\omega}(z, F,X)$ extends holomorphically to $\C \setminus \Gamma_\omega$.
 
\end{corollary}

				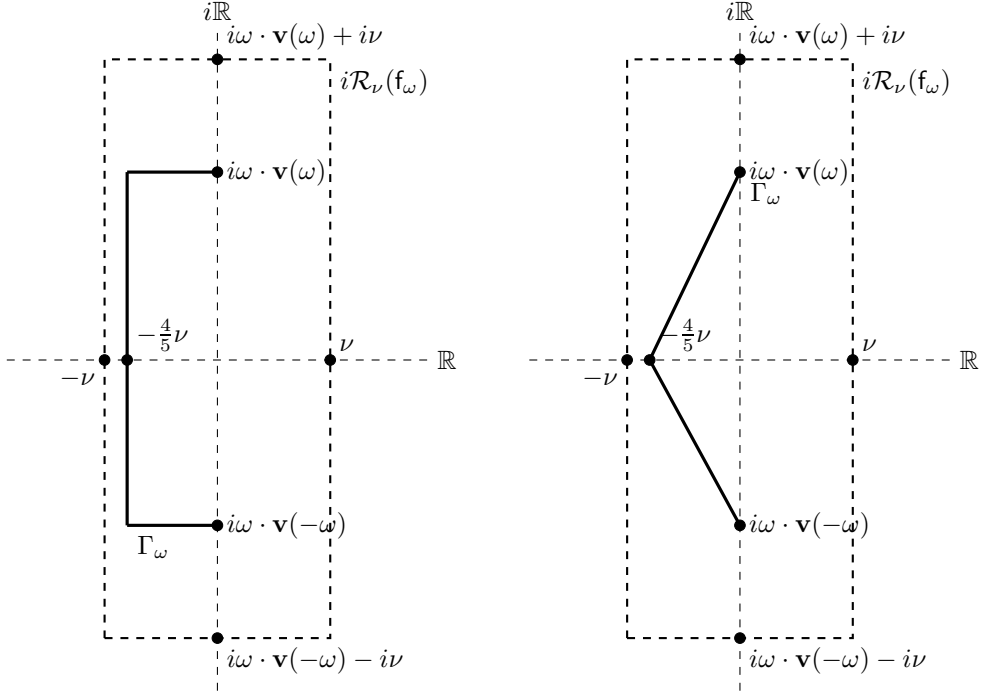
\begin{figure}[!h]
\begin{minipage}[b]{0.4\textwidth} 
		\begin{tikzpicture}
				\draw[dashed] (-2.8,0) -- (2.8,0) node[anchor=west]{$\R$};
				\draw[dashed] (0,-4.4)--(0,4.4) node[above]{$i\R$};
				\filldraw[black] (0,2.5) circle (2pt) node[right]{$i\omega\cdot \mathbf{v}(\omega)$};
				\filldraw[black] (0,-2.2) circle (2pt) node[right]{$i\omega\cdot \mathbf{v}(-\omega)$};
				\filldraw[black] (-1.2,0) circle (2pt) node[anchor=south west]{$-\frac{4}{5}\nu$};
				\draw[very thick] (-1.2,2.5) -- (-1.2,-2.2) node[anchor=north west]{$\Gamma_{\omega}$} ; 
				\draw[very thick]  (-1.2,2.5) -- (0,2.5)  ; 
				\draw[very thick]  (0,-2.2) -- (-1.2,-2.2)  ; 
				\draw[thick,dashed] (-1.5,-3.7) rectangle (1.5,4) node[anchor=north west]{$i\ml{R}_\nu(\f_\omega)$};
				\filldraw[black] (0,4) circle (2pt) node[anchor=south west]{$i\omega\cdot \mathbf{v}(\omega)+i\nu$};
				\filldraw[black] (0,-3.7) circle (2pt) node[anchor=north west]{$i\omega\cdot \mathbf{v}(-\omega)-i\nu$};
				\filldraw[black] (-1.5,0) circle (2pt) node[anchor=north east]{$-\nu$};
				\filldraw[black] (1.5,0) circle (2pt) node[anchor=south west]{$\nu$};
			\end{tikzpicture}
\end{minipage} 
\hspace{0.05\textwidth} 
\begin{minipage}[b]{0.4\textwidth} 
			\begin{tikzpicture}
				\draw[dashed] (-2.8,0) -- (2.8,0) node[anchor=west]{$\R$};
				\draw[dashed] (0,-4.4)--(0,4.4) node[above]{$i\R$};
				\filldraw[black] (0,2.5) circle (2pt) node[right]{$i\omega\cdot \mathbf{v}(\omega)$};
				\filldraw[black] (0,-2.2) circle (2pt) node[right]{$i\omega\cdot \mathbf{v}(-\omega)$};
				\filldraw[black] (-1.2,0) circle (2pt) node[anchor=south west]{$-\frac{4}{5}\nu$};
				\draw[very thick]  (-1.2,0) -- (0,2.5)  node[anchor=north west]{$\Gamma_{\omega}$} ; 
				\draw[very thick]  (0,-2.2) -- (-1.2,0)  ; 
				\draw[thick,dashed] (-1.5,-3.7) rectangle (1.5,4) node[anchor=north west]{$i\ml{R}_\nu(\f_\omega)$};
				\filldraw[black] (0,4) circle (2pt) node[anchor=south west]{$i\omega\cdot \mathbf{v}(\omega)+i\nu$};
				\filldraw[black] (0,-3.7) circle (2pt) node[anchor=north west]{$i\omega\cdot \mathbf{v}(-\omega)-i\nu$};
				\filldraw[black] (-1.5,0) circle (2pt) node[anchor=north east]{$-\nu$};
				\filldraw[black] (1.5,0) circle (2pt) node[anchor=south west]{$\nu$};
			\end{tikzpicture}
\end{minipage} 
			\caption{Two possible curves $\Gamma_{\omega}$ (in thick lines) in the box $i\ml{R}_\nu(\f_\omega)\subset\C$ (the thick dashed square) for the function $\f_\omega=\omega\cdot\mathbf{v}(\theta)$, in the general case in which $\mathbf{v}(-\theta)\neq- \mathbf{v}(\theta)$.}
			\label{f:slit-1}
\end{figure}

Note that in the last statement, illustrated on Figure~\ref{f:slit-1}, we identify the map $\Gamma_\omega$ and its image $\Gamma_\omega([0,1])$. 
The first part of the corollary is a straightforward consequence of Lemma~\ref{l:mathcal-I}.
The holomorphic continuation part 
follows from Remark~\ref{e:extension-box}, itself relying on Lemma~\ref{l:laplace-bessel} Item~\eqref{i:move-slit} in Appendix~\ref{a:branched}. 
Other possible holomorphic continuation statements (e.g. as a multivalued holomorphic Nilsson function) follow from Lemma~\ref{l:laplace-bessel} (see Remark~\ref{e:extension-box}). We do not state them since they are inconvenient to sum in view of applying Corollary~\ref{c:I-} to continue $\mathcal{P}_{K_0}^{<N}$. 

Note that recalling~\eqref{e:hK-fct-v}, we have $\omega\cdot\mathbf{v}(\omega)=h_K(\omega)>0$ and $\omega\cdot\mathbf{v}(-\omega)=-h_K(-\omega)<0$.

\subsection{Analytic continuation of $\mathcal{P}_{K_0}^{<N}$: proof of Proposition~\ref{p:low-freq}}
\label{ss:zeta-low-freq}

From Corollary~\ref{c:I-} together with the holomorphic continuation results it implies, we now deduce the expected results for $\mathcal{P}_{K_0}^{<N}$. We recall the expression of $\mathcal{P}_{K_0}^{<N}$ as a finite sum of terms of the form $I^{(\ell)}_{\lambda,\omega}(z, F,X)$ in~\eqref{e:def-zeta-inf}.

We first deduce from Corollary~\ref{c:I-} that (in addition to the uniform statement) $I^{(\ell)}_{\lambda,\omega}(z, F,X)$ extends as a holomorphic function to 
 the set $\C \setminus \Gamma_\omega$, where $\Gamma_\omega\subset i\mathcal{R}_{\nu}(\f_\omega) \cap\{ \Re(z) \leq 0\}$ is any curve satisfying the statement of Corollary~\ref{c:I-}.
 As a consequence, for any given $\xi\in\Z^d\setminus\{0\}$,  we deduce that $s \mapsto I_{|\xi|,\frac{\xi}{|\xi|}}\left(\frac{s}{|\xi|},\Omega_\ell,-\frac{\xi}{|\xi|}\cdot x_{K_0} \right)$ extends to 
  $\C \setminus |\xi| \Gamma_{\frac{\xi}{|\xi|}}$, where, according to~\eqref{e:shape-cut}
   $$
  |\xi|\Gamma_{\frac{\xi}{|\xi|}} \subset i|\xi|\mathcal{R}_{\nu}(\f_{\frac{\xi}{|\xi|}}) \cap\{ \Re(z) \leq 0\} = \ml{V}_{\xi,\nu} .
  $$
 Now notice that any curve $\mathscr{C}_{\xi}$ satisfying the requirements of the statement of the proposition may be written as $\mathscr{C}_{\xi}=  |\xi|\Gamma_{\frac{\xi}{|\xi|}}$ where $\Gamma_\omega\subset i\mathcal{R}_{\nu}(\f_\omega) \cap\{ \Re(z) \leq 0\}$ is curve satisfying the statement of Corollary~\ref{c:I-}.
 Summing over all $\xi \in \Z^d, 0<|\xi|<N$ and recalling the expression of $\mathcal{P}_{K_0}^{<N}$ in~\eqref{e:def-zeta-inf} concludes the proof of the first part of the proposition.
Concerning the second part of the proposition, for any given $\xi\in\Z^d\setminus\{0\}$, the appropriate explicit expression of $s \mapsto I_{|\xi|,\frac{\xi}{|\xi|}}\left(\frac{s}{|\xi|},\Omega_\ell,-\frac{\xi}{|\xi|}\cdot x_{K_0} \right)$ in a neighborhood of $i\xi \cdot  \mathbf{v}(\pm\frac{\xi}{|\xi|}) +(-\nu|\xi|,0]$ also follows from Corollary~\ref{c:I-}.
Summing over all $\xi \in \Z^d, 0<|\xi|<N$ and recalling the expression of $\mathcal{P}_{K_0}^{<N}$ in~\eqref{e:def-zeta-inf} concludes the proof of the proposition, noticing that
\begin{itemize}
\item  there is at most a finite number of $\xi \in \Z^d$ such that  $\xi \cdot\mathbf{v}\Big(\frac{\xi}{|\xi|}\Big) = \xi_0 \cdot\mathbf{v}\Big(\pm \frac{\xi_0}{|\xi_0|}\Big)$;
\item in Corollary~\ref{c:I-}, the exponent of the singularity is $\frac{d-3}{2}-\ell$ for $\ell \in \{ 0, \dots , d-1\}$ and hence the most singular power in  the expression of $\mathcal{P}_{K_0}^{<N}$ is $\frac{d-3}{2}-(d-1) = -\frac{d+1}{2}$.
\end{itemize}
This concludes the proof of Proposition~\ref{p:low-freq}.
  \hfill \qedsymbol \endproof

\section{Analyzing the high frequency terms and proof of Proposition~\ref{p:highfrequency}}
\label{s:high-freq}

The goal of this section is to complete the proofs of Proposition~\ref{p:highfrequency} by analyzing the high frequency function $\mathcal{P}_{K_0}^{\geq N}(s)$ defined in~\eqref{e:def-zeta-sup}. Recall from Lemma~\ref{l:decomp-zeta-I} that this part of the Poincar\'e series can be written as a series
\begin{align*}
 \mathcal{P}_{K_0}^{\geq N}(s)& =\frac{1}{(2\pi)^d}\sum_{|\xi|\geq N}\sum_{\ell=0}^{d-1}\sum_{j=0}^\ell \frac{\ell!}{j!} \mathbf{I}_{\ell,j}(s,\xi), \quad \text{with }\\
\mathbf{I}_{\ell,j}(s,\xi)&:=\int_{\R}\chi_1'(t)t^j e^{-st}\left(\int_{\IS^{d-1}}\frac{\Omega_\ell(\theta) e^{i\xi\cdot(t\mathbf{v}(\theta)+x_{K_0}(\theta))}}{\left(s-i\xi\cdot\mathbf{v}(\theta)\right)^{\ell+1-j}}\right)dt,\quad \xi\in\mathbb{Z}^d, |\xi|\geq N ,
\end{align*}
where $\chi_1'$ is compactly supported in $(\kappa_1,2\kappa_1)$ and $0\leq j\leq\ell\leq d-1$. As a consequence of Lemma~\ref{l:coarea} with $\f(\theta)=\xi\cdot\mathbf{v}(\theta)$, this integral can be rewritten as
$$
\mathbf{I}_{\ell,j}(s,\xi)=\int_{\R}\chi_1'(t)t^j e^{-st}\left(\int_{\xi\cdot\mathbf{v}(-\xi)}^{\xi\cdot\mathbf{v}(\xi)}\frac{e^{it\tau}}{(s-i\tau)^{\ell+1-j}}\mathcal{J}\left(\frac{\tau}{|\xi|},\frac{\xi}{|\xi|}\cdot\mathbf{v},\Omega_\ell e^{i\xi\cdot x_{K_0}}\right) d\tau \right)dt,
$$
where $\mathcal{J}$ is the function defined in~\eqref{e:def-J}. In particular, it is analytic and it satisfies the estimates from Lemma~\ref{l:desintegration}. After $d$ integration by parts in the time variable, one has
$$
\mathbf{I}_{\ell,j}(s,\xi)=\int_{\R}(\chi_1'(t)t^j)^{(d)} e^{-st}\left(\int_{\xi\cdot\mathbf{v}(-\xi)}^{\xi\cdot\mathbf{v}(\xi)}\frac{e^{it\tau}}{(s-i\tau)^{\ell+d+1-j}}\mathcal{J}\left(\frac{\tau}{|\xi|},\frac{\xi}{|\xi|}\cdot\mathbf{v},\Omega_\ell e^{i\xi\cdot x_{K_0}}\right) d\tau \right)dt.
$$
Note that the number $d$ of integration by parts is chosen so that to obtain the convergent series $\sum_{|\xi| \geq N}|\xi|^{-d-1}$ in the end of the argument.
We now fix some (small) parameter $\epsilon_0>0$ and we split this integral into three parts:
\begin{align}
\mathbf{I}_{\ell,j}(s,\xi) & =  \mathbf{I}_{\ell,j}^{(+ 1)}(s,\xi) +  \mathbf{I}_{\ell,j}^{(- 1)}(s,\xi) +  \mathbf{I}_{\ell,j}^{(0)}(s,\xi) ,  \quad \text{with} \nonumber \\
\label{e:close-critical}
 \mathbf{I}_{\ell,j}^{(\pm 1)}(s,\xi)&:=\int_{\R}(\chi_1'(t)t^j)^{(d)} e^{-st}\left(\int_{\xi\cdot\mathbf{v}(\pm\xi)\mp\epsilon_0|\xi|}^{\xi\cdot\mathbf{v}(\pm\xi)}\frac{e^{it\tau}}{(s-i\tau)^{\ell+d+1-j}}\mathcal{J}\left(\frac{\tau}{|\xi|},\frac{\xi}{|\xi|}\cdot\mathbf{v},\Omega_\ell e^{i\xi\cdot x_{K_0}}\right) d\tau \right)dt, \\
\label{e:far-critical}
 \mathbf{I}_{\ell,j}^{(0)}(s,\xi)&: =\int_{\R}(\chi_1'(t)t^j)^{(d)} e^{-st}\left(\int_{\xi\cdot\mathbf{v}(-\xi)+\epsilon_0|\xi|}^{\xi\cdot\mathbf{v}(\xi)-\epsilon_0|\xi|}\frac{e^{it\tau}}{(s-i\tau)^{\ell+d+1-j}}\mathcal{J}\left(\frac{\tau}{|\xi|},\frac{\xi}{|\xi|}\cdot\mathbf{v},\Omega_\ell e^{i\xi\cdot x_{K_0}}\right) d\tau \right)dt.
\end{align}

We begin with $\mathbf{I}_{\ell,j}^{(\pm 1)}(s,\xi)$ which is slightly easier to handle. Let $s\in\mathbb{C}$ such that $|\text{Im}(s)|\leq\delta_0 N$ for some small enough $\delta_0>0$ to be determined. On the interval of integration in the $\tau$-variable, one has $|s-i\tau|\geq c_1|\xi|-\delta_0N$ for some constant $c_1>0$ depending only on the convex set $K$ and on $\epsilon_0>0$. Unfolding the argument in \S\ref{ss:level-lines}, one finds then
$$
\left|\int_{\xi\cdot\mathbf{v}(\pm\xi)\mp\epsilon_0|\xi|}^{\xi\cdot\mathbf{v}(\pm\xi)}\frac{e^{it\tau}}{(s-i\tau)^{\ell+d+1-j}}\mathcal{J}\left(\frac{\tau}{|\xi|},\frac{\xi}{|\xi|}\cdot\mathbf{v},\Omega_\ell e^{i\xi\cdot x_{K_0}}\right) d\tau \right|\leq\frac{\int_{\IS^{d-1}}|\Omega_\ell(\theta)|d\text{Vol}_{\IS^{d-1}}(\theta)}{(c_1|\xi|-\delta_0N)^{\ell+d+1-j}}.
$$
As $\ell+d+1-j\geq d+1$, $c_1|\xi|-\delta_0 N \geq |\xi|(c_1-\delta_0)$ for $|\xi|\geq N$, and recalling that $\chi_1$ is compactly supported, this yields a uniformly convergent sum in $|\xi|$ and we find that the  sum 
$$
\sum_{|\xi|\geq N}\sum_{\ell=0}^{d-1}\sum_{j=0}^{\ell} \sum_\pm \frac{\ell!}{j!} \mathbf{I}_{\ell,j}^{(\pm 1)}(s,\xi)
$$
is holomorphic in the strip $\{|\text{Im}(s)|\leq \delta_0 N\}$ provided that  $\delta_0$ is chosen so that $\delta_0\in(0,c_1)$. Observe that this step of the argument does not require any analyticity of the convex sets (only finite regularity is required).

Hence, the proof of Proposition~\ref{p:highfrequency} reduces to proving the analytic continuation of
\begin{align}
\label{e:PK0-tilde}
\widetilde{\mathcal{P}}_{K_0}^{\geq N}(s)=\sum_{|\xi|\geq N}\sum_{\ell=0}^{d-1}\sum_{j=0}^{\ell} \frac{\ell!}{j!}\mathbf{I}_{\ell,j}^{(0)}(s,\xi).
\end{align}
Again, we analyze each term individually and, to that aim, we introduce the following (counterclockwisely oriented, see Figure~\ref{f:high-frequency-1}) curve
\begin{multline*}
\widetilde{\mathscr{C}}_\xi:=\left\{ \text{Re}(z)=\xi\cdot\mathbf{v}(\pm\xi)\mp \epsilon_0|\xi|,\ \text{Im}(z)\in[0,\epsilon_0|\xi|]\right\}\\
\cup\left\{ \text{Re}(z)\in [\xi\cdot\mathbf{v}(-\xi)+\epsilon_0|\xi|,\xi\cdot\mathbf{v}(\xi)-\epsilon_0|\xi|],\ \text{Im}(z)=\epsilon_0 |\xi|\right\}.
\end{multline*}

				\begin{figure}[!h]
		\begin{tikzpicture}
				\draw[dashed] (-5,0) -- (5,0) node[anchor=west]{$\R$};
				\draw[dashed] (0,-1)--(0,3) node[above]{$i\R$};
				\filldraw[black] (-2.2,0) circle (2pt) node[anchor=north]{\ \ \ $\xi\cdot \mathbf{v}(-\xi)+\epsilon_0|\xi|$};
				\filldraw[black] (2.5,0) circle (2pt) node[anchor=north]{$\xi\cdot \mathbf{v}(\xi)-\epsilon_0|\xi|$\ \ \ };
				\draw[->, very thick] (-2.2,1.7) -- (2.5,1.7) node[anchor=north west]{$\tilde{\mathscr{C}}_{\xi}$} ; 
				\draw[->, very thick]   (-2.2,0) -- (-2.2,1.7)  ; 
				\draw[->, very thick]   (2.5,1.7) -- (2.5,0)  ; 
				\filldraw[black] (4,0) circle (2pt) node[anchor=south west]{$\xi\cdot \mathbf{v}(\xi)$};
				\filldraw[black] (-3.7,0) circle (2pt) node[anchor=south east]{$\xi\cdot \mathbf{v}(-\xi)$};
				\filldraw[black] (0,1.7) circle (2pt) node[anchor=north east]{$i\epsilon_0|\xi|$};
			\end{tikzpicture}
				\caption{Deformation of the interval $\big[\xi\cdot \mathbf{v}(-\xi)+\epsilon_0|\xi|, \xi\cdot \mathbf{v}(\xi)-\epsilon_0|\xi|\big]$ into the contour $\tilde{\mathscr{C}}_{\xi}$, where the variable $z$ lives.}
				\label{f:high-frequency-1}
\end{figure}
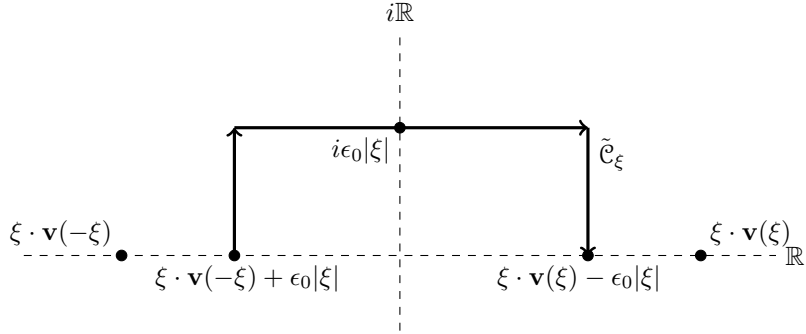
According to Lemma~\ref{l:desintegration}, the function $\tau \mapsto \mathcal{J}\left(\frac{\tau}{|\xi|},\frac{\xi}{|\xi|}\cdot\mathbf{v},\Omega_\ell e^{i\xi\cdot x_{K_0}}\right)$ is holomorphic on a neighborhood of the rectangle surrounded by $\widetilde{\mathscr{C}}_\xi$ and $[\xi\cdot\mathbf{v}(-\xi),\xi\cdot\mathbf{v}(\xi)]$. Hence, for $\Re(s)>0$, the integrand in the $\tau$-variable in~\eqref{e:far-critical} is holomorphic on the same region and we can thus change the contour of integration in $\tau$ from $[\xi\cdot\mathbf{v}(-\xi),\xi\cdot\mathbf{v}(\xi)]$ to  $\widetilde{\mathscr{C}}_\xi$, namely
\begin{align}
\label{e-last-I0}
\mathbf{I}_{\ell,j}^{(0)}(s,\xi)=\int_{\R}(\chi_1'(t)t^j)^{(d)} e^{-st}\left(\int_{\widetilde{\mathscr{C}}_\xi}\frac{e^{itz}}{(s-iz)^{\ell+d+1-j}}\mathcal{J}\left(\frac{z}{|\xi|},\frac{\xi}{|\xi|}\cdot\mathbf{v},\Omega_\ell e^{i\xi\cdot x_{K_0}}\right) dz \right)dt.
\end{align}
The integral in $t$ is over a compact interval, so we only need to provide a bound on the integral on $\widetilde{\mathscr{C}}_\xi$, which is summable in $\xi \in \Z^d,|\xi|\geq N$. To this aim, we will need the following property.
\begin{lemma}\label{l:tricky} With the above conventions, there exists a constant $C_1>0$ depending only on $R$, $K_0$, $\mathbf{v}$ and $\epsilon_0$ such that, for all $N>1$, for all $|\xi|\geq N$ and for all $z\in\tilde{\mathscr{C}}_\xi$,
$$
\left|\mathcal{J}\left(\frac{z}{|\xi|},\frac{\xi}{|\xi|}\cdot\mathbf{v},\Omega_\ell e^{i\xi\cdot x_{K_0}}\right) \right|
 \leq 2C_1\|\Omega_\ell\|_R e^{C_1|\operatorname{Im}(z)|}.
 $$ 
\end{lemma}
The proof of this Lemma is postponed to the end of this section and we first explain how it allows to conclude. Indeed, we infer from this Lemma that
$$
\left|\frac{e^{itz}}{(s-iz)^{\ell+d+1-j}}\mathcal{J}\left(\frac{z}{|\xi|},\frac{\xi}{|\xi|}\cdot\mathbf{v},\Omega_\ell e^{i\xi\cdot x_{K_0}}\right)\right|\leq \frac{2C_1\|\Omega_\ell\|_R e^{C_1|\text{Im}(z)|-t\text{Im}(z)}}{|s-iz|^{\ell+d+1-j}}.
$$
In particular, if we pick $\text{Im}(z)\geq 0$ and $t\geq \kappa_1$, one finds
$$
\left|\frac{e^{itz}}{(s-iz)^{\ell+d+1-j}}\mathcal{J}\left(\frac{z}{|\xi|},\frac{\xi}{|\xi|}\cdot\mathbf{v},\Omega_\ell e^{i\xi\cdot x_{K_0}}\right)\right|\leq \frac{2C_1\|\Omega_\ell\|_R e^{-(\kappa_1-C_1)|\text{Im}(z)|}}{|s-iz|^{\ell+d+1-j}}.
$$
Hence, fixing $\kappa_1>C_1$, the upper bound is of the form $2C_1\|\Omega_\ell\|_R |s-iz|^{-\ell-d-1+j}.$

Observe now that, for $z \in \widetilde{\mathscr{C}}_\xi$ (hence  $i z \in i\widetilde{\mathscr{C}}_\xi$) and for $s\in\{\text{Re}(s)\geq -\delta_0N, |\text{Im}(s)|\leq\delta_0N\}$, one has $|s-iz| \geq |\epsilon_0|\xi|- \delta_0N| \geq |\xi|(\epsilon_0-\delta_0)$ (see also Figure~\ref{f:high-frequency-2}).
\begin{figure}
		\begin{tikzpicture}
				\draw[dashed] (-2.8,0) -- (2.8,0) node[anchor=west]{$\R$};
				\draw[dashed] (0,-4.4)--(0,4.4) node[above]{$i\R$};
				\filldraw[black] (0,2.5) circle (2pt) node[right]{$i(\xi\cdot \mathbf{v}(\xi)-\epsilon_0|\xi|)$};
				\filldraw[black] (0,-2.2) circle (2pt) node[right]{$i(\xi\cdot \mathbf{v}(-\xi)+\epsilon_0|\xi|)$};
				\draw[very thick] (-1.7,2.5) -- (-1.7,-2.2) node[anchor=north west]{$i\tilde{\mathscr{C}}_{\xi}$} ; 
				\draw[very thick]  (-1.7,2.5) -- (0,2.5)  ; 
				\draw[very thick]  (0,-2.2) -- (-1.7,-2.2)  ; 
				\filldraw[black] (0,4) circle (2pt) node[right]{$i\xi\cdot \mathbf{v}(\xi)$};
				\filldraw[black] (0,-3.7) circle (2pt) node[right]{$i\xi\cdot \mathbf{v}(-\xi)$};
				\filldraw[black] (-1.7,0) circle (2pt) node[anchor=north east]{$-\epsilon_0|\xi|$};
				\filldraw[black] (0,-1.2) circle (2pt) node[anchor=north west]{$-i\delta_0 N$};
				\filldraw[black] (-1.2,0) circle (2pt) node[anchor=north west]{$-\delta_0 N$};
				\filldraw[black] (0,1.2) circle (2pt) node[anchor=south west]{$i\delta_0 N$};
				\draw[pattern=north west lines, pattern color=blue] (-1.2,-1.2) rectangle (1.2,1.2);
			\end{tikzpicture}
		\caption{Region where $s$ lives (striped) and contour $i\tilde{\mathscr{C}}_{\xi}$ where $iz$ lives.}
			\label{f:high-frequency-2}
\end{figure}
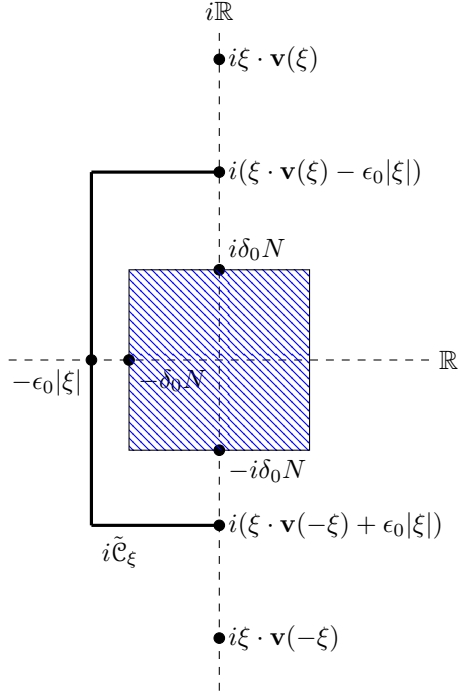
Hence, recalling that $\ell+d+1-j\geq d+1$, we obtain in~\eqref{e-last-I0} that $|\mathbf{I}_{\ell,j}^{(0)}(s,\xi)| \leq C|\xi|^{-d-1}$, hence an absolutely convergent contribution in~\eqref{e:PK0-tilde} which is thus a holomorphic function in the $s$ variable. This completes the proof of Proposition~\ref{p:highfrequency}.
  \hfill \qedsymbol \endproof

\begin{proof}[Proof of Lemma~\ref{l:tricky}]
Thanks to Lemma~\ref{l:desintegration}, one knows that, for all $z \in \widetilde{\mathscr{C}}_\xi$ and $R'\in (0,R/4)$ ($R'$ to be fixed later on wisely), 
\begin{multline}\label{e:Linfty-norm-J}
 \left|\mathcal{J}\left(\frac{z}{|\xi|},\frac{\xi}{|\xi|}\cdot\mathbf{v},\Omega_\ell e^{i\xi\cdot x_{K_0}}\right) \right| = \left| \sum_{k=0}^\infty \left(i\frac{\Im(z)}{|\xi|}\right)^k \frac{\mathcal{J}^{(k)}}{k!}\left(\frac{\Re(z)}{|\xi|},\frac{\xi}{|\xi|}\cdot\mathbf{v},\Omega_\ell e^{i\xi\cdot x_{K_0}}\right) \right|  \\
 \leq \|\Omega_\ell e^{i\xi\cdot x_{K_0}}\|_{R'}\sum_{k=0}^\infty \left(\frac{|\text{Im}(z)|}{R'|\xi|}\right)^k\left(\frac{C(R,\mathbf{v})}{\min\left\{\frac{\xi}{|\xi|}\cdot\mathbf{v}\left(\frac{\xi}{|\xi|}\right)-\epsilon_0,\frac{\xi}{|\xi|}\cdot\mathbf{v}\left(-\frac{\xi}{|\xi|}\right)+\epsilon_0\right\}}\right)^{k+1}.
\end{multline}
From Lemma~\ref{c:product}, one finds that $\|\Omega_\ell e^{i\xi\cdot x_{K_0}}\|_{R'}\leq C(R)\|\Omega_\ell\|_R\|e^{i\xi\cdot x_{K_0}}\|_{2R'}$ for some constant $C(R)>0$ depending only on $R$. Thanks to Lemma~\ref{l:comparison-norm}, one also has $\|e^{i\xi\cdot x_{K_0}}\|_{2R'}\leq\|e^{i\xi\cdot x_{K_0}}\|_{\mathcal{H}_{2R'}}$. Now, as  $R'\in (0,R/4)$ can be chosen arbitrarily, we fix $z \in \widetilde{\mathscr{C}}_\xi$ such that $\text{Im}(z)\neq 0$ and the radius $R'$ to be equal to 
\begin{align}
\label{e:choice-R'}
R'= R'(z) := \frac{2 C(R,\mathbf{v})|\text{Im}(z)|}{|\xi|\min\left\{\frac{\xi}{|\xi|}\cdot\mathbf{v}\left(\frac{\xi}{|\xi|}\right)-\epsilon_0,\frac{\xi}{|\xi|}\cdot\mathbf{v}\left(-\frac{\xi}{|\xi|}\right)+\epsilon_0\right\}} ,
\end{align}
 which implies 
\begin{align}
\label{e:Linfty-norm-J-bis}
\left|\mathcal{J}\left(\frac{z}{|\xi|},\frac{\xi}{|\xi|}\cdot\mathbf{v},\Omega_\ell e^{i\xi\cdot x_{K_0}}\right) \right|
 \leq 2C_R\|\Omega_\ell\|_R\| e^{i\xi\cdot x_{K_0}}\|_{\mathcal{H}_{2R'}}.
 \end{align}
Recalling the definition of the $\mathcal{H}_{2R'}$ norm in~\eqref{e:define-HR-norm}, we now estimate the term $$
\| e^{i\xi\cdot x_{K_0}}\|_{\mathcal{H}_{2R'}} =\sum_{j=1}^{n_0}\left\| e^{i\xi\cdot \tilde{x}_{K_0,M_j}}  \right\|_{L^\infty(\mathcal{U}_{2R'})} ,
$$
where $\tilde{x}_{K_0,M_j}$ is a holomorphic extension of the function $x_{K_0}$ in chart $M_j$ to the polydic-shaped domain in~\eqref{e:def-UR}. We estimate each term in the sum and remove the index $M_j$ for readability. On the one hand, we have 
\begin{align}
\label{e:intermediate-expo-i}
\left\| e^{i\xi\cdot \tilde{x}_{K_0} } \right\|_{L^\infty(\mathcal{U}_{2R'})} 
= \sup_{\zeta \in \mathcal{U}_{2R'}}\left| e^{i\xi\cdot \tilde{x}_{K_0}(\zeta)}\right| 
\leq \sup_{\zeta \in \mathcal{U}_{2R'}} e^{|\xi| | \Im( \tilde{x}_{K_0} )(\zeta)|} = e^{|\xi| \| \Im( \tilde{x}_{K_0})\|_{L^\infty(\mathcal{U}_{2R'})}} .
\end{align}
On the other hand, we have for $\zeta \in \mathcal{U}_{2R'}$
$$
\tilde{x}_{K_0} (\zeta) = \tilde{x}_{K_0} (\Re(\zeta)) + i \Im(\zeta)\cdot \int_0^1 d \tilde{x}_{K_0} \big(\Re(\zeta)+ i\sigma \Im(\zeta)\big) d\sigma ,
$$
and hence, again for $\zeta \in \mathcal{U}_{2R'}$ (which implies for $|\Im(\zeta_j)| \leq 2R'$ for all $j\in \{1,\dots n\}$), 
\begin{align*}
\| \Im( \tilde{x}_{K_0})\|_{L^\infty(\mathcal{U}_{2R'})}
&  = \left\|  \Im(\zeta)\cdot \int_0^1 d \tilde{x}_{K_0} \big(\Re(\zeta)+ i\sigma \Im(\zeta)\big) d\sigma  \right\|_{L^\infty(\mathcal{U}_{2R'})}  \\
& \leq CR' \|\tilde{x}_{K_0}  \|_{L^\infty(\mathcal{U}_{3R'})} \leq  CR' \|x_{K_0}  \|_{R/2} , 
\end{align*}
for constants $C>1$ depending only on $K_0$ and $R$. Recalling the choice of $R'=R'(z)$ in~\eqref{e:choice-R'} and the estimate~\eqref{e:intermediate-expo-i}, this implies existence of a constant $C(K_0,K,R)>1$ depending only on $K_0$, $K$ and $R$ such that for all $|\xi|\geq N$, 
$$
\| e^{i\xi\cdot x_{K_0}}\|_{\mathcal{H}_{2R'}} \leq e^{C(K_0,K,R)|\text{Im}(z)|} .
$$
Hence,~\eqref{e:Linfty-norm-J-bis} becomes, for $N>1$ large enough and all $|\xi|\geq N$,
$$
\left|\mathcal{J}\left(\frac{z}{|\xi|},\frac{\xi}{|\xi|}\cdot\mathbf{v},\Omega_\ell e^{i\xi\cdot x_{K_0}}\right) \right|
 \leq 2C(R,\mathbf{v},\epsilon_0)\|\Omega_\ell\|_R e^{C(K_0,K,R)|\text{Im}(z)|},
 $$
for some constant $C(R,\mathbf{v},\epsilon_0)$ depending only on $R$, $\mathbf{v}$ and $\epsilon_0$.
\end{proof}

\appendix

\section{Brief reminder on multivalued holomorphic functions}\label{a:branched}
\label{s:remainder-multivalued}

When taking $\IC\setminus\IR_-$ for the domain of the holomorphic extension for $\ln(z)$ or $\sqrt{z}$, one makes a somehow arbitrary choice of simply connected domain. One way to fix this arbitrary choice is to introduce the more adapted notion of multivalued analytic functions. Another (essentially equivalent) way consists in defining the proper Riemann surface associated with functions involving several logarithmic (or squareroot) singularities. We briefly review these notions in this appendix.

\subsection{Multivalued holomorphic functions: definitions}
\label{s:multivalued-holom-fcts-def}

We follow the approach of~\cite{CandelpergherNosmasPham93,Ebeling07}. For every open subset $U\subset \mathbb{C}$, we denote by $\mathcal{O}(U)$ the algebra of holomorphic functions on $U$ and for every $a\in U$, $\mathcal{O}_a$ denotes the algebra of holomorphic germs at $a$. Recall that two holomorphic functions $f$ and $g$ defined in a neighborhood of $a$ give the same holomorhic germ at $a$ if there exists $r>0$ such that $f=g$ on $D(a,r)$ (or equivalently if $f^{(k)}(a) = g^{(k)}(a)$ for all $k \in \N$). Using these conventions, we can define the notion of analytic continuation along paths~\cite[\S1.4, p.27]{Ebeling07}:

\begin{definition}[Analytic continuation along paths]
\label{d:analytic-path}
Let $\gamma:[0,1]\rightarrow\IC$ be a continuous curve connecting $(a,b)\in\mathbb{C}^2$, that is to say such that $\gamma(0)=a$ and $\gamma(1)=b$ and let $f\in \mathcal{O}_a$ be a germ of holomorphic function at $a$.
Then, given a holomorphic germ at $b$, $g \in \mathcal{O}_b$, we say that $g$ is obtained from analytic continuation of $f$ along $\gamma$ if there exists a subdivision
$$0=t_0<t_1<t_2\ldots<t_{p}=1$$
of $[0,1]$, open sets $(U_i)_{i=1}^p$ of $\mathbb{C}$ such that $\gamma([t_{i-1},t_{i}])\subset U_i$, holomorphic functions $g_i\in \mathcal{O}(U_i)$ such that
\begin{enumerate}
 \item $g_i=g_{i-1}$ on the connected component of $U_{i-1}\cap U_i$ containing $\gamma(t_i)$;
 \item one has
 $$f=g_1\ \text{in}\ \mathcal{O}_a,\quad\text{and}\quad g=g_p\ \text{in}\ \ml{O}_b.$$
\end{enumerate}
We say that $f$ can be analytically continued along $\gamma$ if there is such a $g \in \mathcal{O}_b$ obtained from analytic continuation of $f$ along $\gamma$.
\end{definition}
We emphasize that we may have $a=b$ but not necessarly $g=f$ in $\ml{O}_a$. The most important fact on these analytic continuations along path is the following (see~\cite[Prop.~1.21]{Ebeling07} or~\cite[Th\'eor\`eme~1.7 p84]{CandelpergherNosmasPham93}).
\begin{proposition}[Monodromy Theorem] Let $(a,b)\in\mathbb{C}^2$. Let $\gamma$ be a continuous map $\gamma:(s,t)\in[0,1]^2\mapsto \gamma_s(t)\in\mathbb{C}$ such that, for every $s\in[0,1]$, $\gamma_s(0)=a$ and $\gamma_s(0)=b$. Let $f\in \mathcal{O}_a$. Suppose that, for every $s\in[0,1]$, $f_s\in\mathcal{O}_b$ is an analytic continuation of $f$ along $\gamma_s$. 

Then
$f_s$ is independent of $s$. 
\end{proposition}

Let $f\in \mathcal{O}_{z_0}$ be some holomorphic germ centered at some base point $z_0$. We denote by $\mathcal{P}_{z_0,z}(f)$, the homotopy classes of the path $\gamma$ starting from $z_0$, ending at $z$ and such that $f$ can be analytically continued along $\gamma$ in the sense of the previous proposition. By the monodromy Theorem there is a natural action of the space $\mathcal{P}_{z_0,z}(f)$ on the germ $f$ such that, for any $[\gamma]\in \mathcal{P}_{z_0,z}(f)$ with endpoint $z$, the element $[\gamma].f$ denotes a germ in $\mathcal{O}_z$ obtained by analytically continuing the germ $f$ along $\gamma$. 

\begin{remark}
For the sake of clarity, we will distinguish the analytic continuation along the path $\gamma$ denoted by $\gamma.f$ and the action of elements $[\gamma]$ of $\mathcal{P}_{z_0,z}(f)$ on the germ $f\in\ml{O}_{z_0}$ which is denoted $[\gamma].f$. 
\end{remark}

By construction, one has
\begin{lemma}[Distributive action of continuation and group action]
Let $f$ and $g$ be two holomorphic germs that can be analytically continued along $\gamma$. Then, $fg$ and $f+g$ can be analytically continued along $\gamma$ and we have the identities
\begin{eqnarray*}
\gamma.(fg)=(\gamma.f)(\gamma.g)\\
\gamma.(f+g)=\gamma.f+\gamma.g
\end{eqnarray*}
\end{lemma}
Given two continuous paths $\gamma_1$ and $\gamma_2$ such that $\gamma_1(1)=\gamma_2(0)$,  the concatenation path $\gamma:=\gamma_1\gamma_2$ is defined by 
$$\gamma(\tau)=\gamma_1(2\tau)\ \text{if}\ \tau\in[0,1/2],\ \text{and}\ \gamma(\tau)=\gamma_2(2\tau-1)\ \text{if}\ \tau\in[1/2,1].$$
Then the following holds.
\begin{lemma}\label{l:compositionpath} Let $f\in\ml{O}_a$, let $\gamma_1$ be a continuous path connecting $a$ to $b\in\IC$ and let $\gamma_2$ be a continuous path connecting $b$ to $c\in\IC$. If $f\in\ml{O}_a$ can be analytically continued along $\gamma_1$ and if $\gamma_1.f\in\ml{O}_b$ can be analytically continued along $\gamma_2$, then $f$ can be analytically continued along $\gamma_1\gamma_2$ and one has
$$(\gamma_1\gamma_2).f=\gamma_2.(\gamma_1.f).$$ 
\end{lemma}

We can now define the notion of multivalued analytic continuation:
\begin{definition}[Multivalued analytic continuation]\label{def:multivalued} Let $U$ be a connected open subset of $\IC$. Let $z_0\in U$ and let $f\in\ml{O}_{z_0}$. 
\begin{enumerate}
 \item We say that $f$ has a {\em multivalued holomorphic extension} to $U$ from $z_0$ if for any continuous path $\gamma:[0,1]\rightarrow U$ such that $\gamma(0)=\gamma(1)=z_0$, $f$ can be analytically continued along $\gamma$ (in the sense of Definition~\ref{d:analytic-path}).
 \item We say that the germ $f\in\ml{O}_{z_0}$ has {\em finite determination} if the vector space 
\begin{align}
\label{e:def-sheaf}
\mathcal{V}_{z_0} := \vect\left\{([\gamma].f )(z_0),[\gamma]\in\mathcal{P}_{z_0}(f) \right\} \subset \mathcal{O}_{z_0}
\end{align}
is a {\em finite dimensional} subspace of $\mathcal{O}_{z_0}$ over $\C$ where $\mathcal{P}_{z_0}(f)=\mathcal{P}_{z_0,z_0}(f)$.
\end{enumerate}
 \end{definition}

\begin{remark}
 From the connectedness of the open set $U$, such a germ $f$ can actually be analytically continued along {\em any} path starting at $z_0$ and ending at $z_1\in U$ (if $\gamma(0)=z_0, \gamma(1)=z_1$, then $\gamma \gamma^{-1}=\Id$, hence according to Definition~\ref{def:multivalued}, $f$ can actually be analytically continued along $\gamma \gamma^{-1}$ and thus along $\gamma$ as a consequence of Definition~\ref{d:analytic-path}.
\end{remark}

The multivalued holomorphic function implicitly defined from the germ $f \in \mathcal{O}_{z_0}$ in Definition~\ref{def:multivalued} is explicitly given by
\begin{equation}
\label{e:multivalued-extension}
F : U \ni z \mapsto \left\{ (\gamma.f)(z) , \gamma \in C^0([0,1];U) , \gamma(0)=z_0, \gamma(1)=z \right\} \in \mathcal{P}(\C)  ,
\end{equation}
 where $(\gamma.f)(z)$ denotes the evaluation of the germ $\gamma.f\in\mathcal{O}_z$ at $z$ and  $\mathcal{P}(\C)$ the set of subsets of $\C$.
  In the following, we will often make the abuse of notation to denote by $f$ the multivalued extension $F$.

  \begin{remark}
We refer to~\cite[Section~1.5]{Ebeling07} for a different definition of branched meromorphic continuation (which includes the definition of the associated Riemann surface).
\end{remark}

\begin{definition}[Multivalued holomorphic function]\label{def:multi-holom}
Let $U$ be a connected open subset of $\IC$.
\begin{enumerate}
 \item We say that $F$ is a {\em multivalued holomorphic function} on $U$ if there is $z_0\in U$ and $f\in\ml{O}_{z_0}$ having a multivalued holomorphic extension to $U$ from $z_0$ such that $F$ is defined by~\eqref{e:multivalued-extension} with respect to $f$.
 \item If $F$ is a multivalued holomorphic function on $U$, we say that it has finite determination if the germ $f\in\ml{O}_{z_0}$ has {\em finite determination} (in the sense of Definition~\ref{def:multivalued}). This property does not depend on $z_0$ and on the germ $f\in\ml{O}_{z_0}$.
\end{enumerate}

\end{definition}

If the germ $f\in\ml{O}_{z_0}$ has finite determination on $U$, then we recall that the vector space $\mathcal{V}_{z_0} \subset \mathcal{O}_{z_0}$ defined in~\eqref{e:def-sheaf} is finite dimensional. For any $[\gamma]\in\pi_1(U,z_0)$, the linear map $\mathcal{T}_{[\gamma]} : \mathcal{V}_{z_0} \to \mathcal{V}_{z_0}$ 
defined by  $\mathcal{T}_{[\gamma]}g : = [\gamma].g$ satisfies $\mathcal{T}_{[\gamma]}\in\GL(\mathcal{V}_{z_0})$. The map
$$
\begin{array}{rcl} 
\mathcal{T}:\pi_1(U,z_0)&\to&\GL(\mathcal{V}_{z_0}),\\
\left[\gamma\right]&\mapsto&\mathcal{T}_{[\gamma]} ,
\end{array}
$$
is then a group homomorphism, and its image defines the monodromy of the couple $(U,f)$.
\begin{definition}[Monodromy]
\label{d:monodromy}
Let $U$ be a connected open subset of $\IC$, let $f$ be a multivalued holomorphic function on $U$, with finite determination, and let $z_0\in U$.
For any $[\gamma] \in \pi_1(U,z_0)$, the map $\mathcal{T}_{[\gamma]} \in \GL(\mathcal{V}_{z_0})$ is called a {\em monodromy transformation} of $(U,f)$, and its matrix in a given basis of $\mathcal{V}_{z_0}$ is called a {\em monodromy matrix} of $(U,f)$.
Finally, the image of this homomorphism $\Ran(\mathcal{T}) \subset \GL(\mathcal{V}_{z_0})$ (i.e. the group of all monodromy transformations) is called the {\em monodromy group} of $(U,f)$.
\end{definition}
Note that, as defined, the monodromy group of $(U,f)$ also depends on the point $z_0$. However, given two different points $z_0,z_1\in U$, the two monodromy groups defined with respect to $z_0$ and $z_1$ are isomorphic and we thus omit the point $z_0$ in the notation.

Let us finally mention that all functions considered in this article belong to the so-called {\em Nilsson class} of multivalued holomorphic functions (see the original article~\cite[Def. p.~463]{Nilsson65} or the textbooks~\cite{Deligne} and~\cite[Ch.~VIII, Def.~1.1]{Pham}). The latter class (the simplest one appearing in singularity theory) corresponds to those multivalued holomorphic function having finite determination in the sense of Definition~\ref{def:multi-holom} and having in addition ``moderate growth'' (i.e. blowing up at most like a power) near any of its (isolated) singular points.

\subsection{A few examples}
\label{ex:monodromyexplicit}
We finally give a few concrete standard examples related to the problem studied in the main part of this article.

\begin{ex}\label{ex:trivial}
If $f\in\ml{O}_{z_0}$ extends as a holomorphic function on the open set $U$, then $f$ has a multivalued holomorphic extension to $U$ (which is single-valued in that case). 
\end{ex}

\begin{ex}\label{ex:log} The germ $\ln(z)$ is a well defined holomorhic function on $\{\operatorname{Re}(z)>0\}$. In particular, $\ln\in\ml{O}_1$. Let now $\gamma:[0,1]\rightarrow\IC^*$ be a continuous path verifying $\gamma(0)=\gamma(1)=1$. In order to make the analytic continuation along this loop, we consider the universal cover of $\mathbb{C}^*$:
$$\exp:\mathbb{C}\rightarrow\mathbb{C}^*.$$
One has $\exp(0)=1$ and there exists a unique lift $\tilde{\gamma}$ of $\gamma$ to $\mathbb{C}$ such that $\tilde{\gamma}(0)=1$. By construction of the universal cover, one has $\tilde{\gamma}(1)=2i\pi k_0$ for some $k_0\in\mathbb{Z}$. The lift of the function $\ln(z)$ to $\mathbb{C}$ (through $\exp$) is given by the identity map which is obviously holomorphic on $\IC$. In particular, it can be analytically continued along $\tilde{\gamma}$ and the projection of this analytic continuation along $\tilde{\gamma}$ gives the expected analytic continuation along $\gamma$ on $\mathbb{C}^*$. Note also that $[\gamma].\ln =\ln+ 2i\pi k_0.$ In particular, if we let $\gamma$ be a representative of a generator of $\pi_1(\mathbb{C}^*,1)\simeq\mathbb{Z}$, then it acts on the 
multivalued extension of $\ln$ as
\begin{eqnarray*}
\boxed{[\gamma]. \ln =\ln +2i\pi\ \text{in}\ \ml{O}_1.}
\end{eqnarray*} 
In particular, it has finite determination as $\mathcal{V}_1$ is spanned by $\ln$ and $1$, and the monodromy matrix in the basis $(\ln,1)$ is given by $\begin{pmatrix}1&0\\ 2i\pi& 1\end{pmatrix}$ since $\mathcal{T}_{[\gamma]}1=1$ and $\mathcal{T}_{[\gamma]}\ln = \ln+2i\pi$. Finally, the extension of $\ln$ on $\IC\setminus\IR_-$ could be set through the explicit formula (consequence of $\tan(\theta/2)=\frac{\sin\theta}{\cos\theta+1}$)
$$
z\mapsto \ln(|z|)+2i \operatorname{arctan}\left(\frac{\Im(z)}{\Re(z)+|z|}\right).
$$
This directly yields an expression for the logarithm on $\C\setminus e^{-i\theta}\R_-$ by letting $z\mapsto \ln(e^{i\theta}z)-i\theta.$
\end{ex}

\begin{ex}\label{ex:squareroot}
The germ $\sqrt{z-a}$ near $a+1$ has a multivalued holomorphic continuation from $\{\operatorname{Re}(z)>\operatorname{Re}(a)\}$ to $\mathbb{C}\setminus\{a\}$. Indeed, one has $\sqrt{z-a}=\exp\frac{1}{2}\ln(z-a)$. In particular, the holomorhic continuation along closed paths in $\IC\setminus\{a\}$ follows from the the holomorhic continuation of $\ln$ along closed paths in $\IC^*$ that was discussed in example~\ref{ex:log}. Still according to this example, if $\gamma$ is a generator of $\pi_1(\mathbb{C}\setminus\{a\},a+1)$, then it acts on the multivalued extension of $\sqrt{z-a}$ as
\begin{eqnarray*}
\boxed{[\gamma]. \sqrt{z-a}=-\sqrt{z-a}.}
\end{eqnarray*} 
Again, it has finite determination (equal to $\dim \mathcal{V}_{1+a}=1$, with monodromy matrix given by $-\Id$).
\end{ex}

\subsection{Riemann surface associated to a multivalued holomorphic function}

Given a multivalued holomorphic function, we would like to describe briefly how one can obtain a natural Riemann surface following~\cite[p.~85-87]{CandelpergherNosmasPham93}, \cite[p.~33-34]{de2016uniformization}.   
If we are given some holomorphic germ $f\in \mathcal{O}_a$, then 
we consider the set of all paths $\gamma$ starting from $a$ such that $f$ has well defined analytic continuation along $\gamma$. Denote by $U\subset \mathbb{C}$ the open subset containing all possible ends of such paths and define the 
\emph{\'etale space}: 
\begin{eqnarray*}
\mathcal{S}=\{ [\gamma].f:\ \gamma \in \mathcal{P}_{a, z}(f),\ z\in U\} .
\end{eqnarray*}
The reader should think that intuitively, 
$$\mathcal{S}=\{ \gamma.f \text{ where $\gamma$ is a path in $\mathbb{C}$ along which $f$ admits an analytic continuation} \},$$
with the natural projection $\pi:\mathcal{S}\mapsto U $ mapping every holomorphic germ $\gamma. f$ to its center which is the endpoint $z\in \mathbb{C}$ of the path $\gamma$.  
Then, it is proved in~\cite{CandelpergherNosmasPham93, de2016uniformization}  that $\mathcal{S}$ has a natural structure of a Riemann surface which is the Riemann surface associated to the germ $f$. We refer to these references for further informations.

\subsection{Integral expression of $\ln z$ and $\sqrt{z}$} Two natural multivalued holomorphic functions in the Nilsson class appear in our Theorem, namely $\ln(z)$ and $\sqrt{z}$. The following elementary lemma describing the singularities of certain integrals is the central result in the description of the singularities of the multiplication operator and the Poincar\'e series in Section~\ref{s:low-freq}. The functions $\mathsf{F}_\gamma$ first appear (after a rescaling) in Proposition~\ref{p:main-I}.
We refer to Remark~\ref{r:log-definition} for the convention notation regarding $\ln$ and $\sqrt{\cdot}$.

\begin{lemma}\label{l:laplace-bessel} 
We set
 \begin{align}
 \label{e:def-Fgamma}
  \mathsf{F}_\gamma(\zeta) := \int_{-1}^1\frac{(1-t^2)^\gamma}{\zeta-t} dt,\quad \zeta\notin[-1,1] .
\end{align}
\begin{enumerate}
\item \label{i:explicit-012}  We have
\begin{align*}
 \mathsf{F}_0(\zeta)&= \int_{-1}^1\frac{1}{\zeta-t} dt  = \ln(\zeta+1)-\ln(\zeta-1), \quad \text{ for all }\zeta >1,  \\
 \mathsf{F}_{-\frac12}(\zeta)&= \int_{-1}^1\frac{(1-t^2)^{-\frac12}}{\zeta-t} dt  =  \frac{\pi}{(\zeta+1)^{\frac12}(\zeta-1)^{\frac12}}, \quad \text{ for all } \zeta>1.
\end{align*}
\item \label{i:recurrence} For all $\gamma >-1$, there is $C_\gamma>0$ such that 
 \begin{align*}
  \mathsf{F}_{\gamma+1}(\zeta) =  (1-\zeta^2)  \mathsf{F}_{\gamma}(\zeta) +C_{\gamma} \zeta  , \quad \text{ for all }\zeta \in \C \setminus[-1,1] .
  \end{align*}
  \item \label{i:straight-induct} For any $d \in \N, d \geq 2$, there is a nontrivial polynomial $P_d$ with real coefficients such that 
 \begin{align}
 \label{e:fonctions-speciales-1}
  \mathsf{F}_{\frac{d-3}{2}}(\zeta) & =  (1-\zeta^2)^{\frac{d-3}{2}} \big( \ln(\zeta+1)-\ln(\zeta-1) \big)  + P_d(\zeta)  , \quad \text{ for }d\geq 3 \text{ odd}, \\
   \label{e:fonctions-speciales-2}
  \mathsf{F}_{\frac{d-3}{2}}(\zeta)&  =  (-1)^{\frac{d-2}{2}}\pi (\zeta+1)^{\frac{d-3}{2}}(\zeta-1)^{\frac{d-3}{2}}  + P_d(\zeta)  , \quad \text{ for }d\geq 2 \text{ even} .
  \end{align}
 \item  \label{i:move-slit} For any $\gamma>-1$, fix $\zeta_0 \in \C\setminus [-1,1]$, and let $\Gamma \in C^1_{pw}([0,1],\C \setminus\{\zeta_0\})$ (the set of continuous and piecewise $C^1$ path) be such that $\Gamma(0) = -1$ and $\Gamma(1)=1$ and $\Gamma([0,1])$ is {\em strictly homotopic} (in the sense that the homotopy $(\Gamma_s(t))_{s,t\in [0,1]}$ satisfies $\Gamma_s(0) = -1$ and $\Gamma_s(1)=1$ for all values of the homotopy parameter $s \in [0,1]$) to $[-1,1]$ in $\C \setminus\{\zeta_0\}$. 
Then, the function $\mathsf{F}_{\gamma}$ admits a (unique) analytic continuation from $\zeta_0$ to $\C \setminus \Gamma([0,1])$ (in the sense of Definition~\ref{d:analytic-path}), explicitly given by $$\mathsf{F}_\gamma(\zeta) = \int_\Gamma \frac{(1-t^2)^\gamma}{\zeta-t} dt.$$
\item \label{i:Nilsson-F} For any $d \in \N, d \geq 2$, $\zeta_0 \in \C\setminus [-1,1]$, the function $\mathsf{F}_{\frac{d-3}{2}}$ admits a (unique) {\em multivalued holomorphic extension} from $\zeta_0$ to $\C\setminus \{- 1,1\} $, as a function of Nilsson class.  
 Moreover, the monodromy group of $\mathsf{F}_{\frac{d-3}{2}}$ is $\mathbb{Z}/2\mathbb{Z} $ if $d$ is even with monodromy relation $$\gamma_{+1}\gamma_{-1}\mathsf{F}_{\frac{d-3}{2}} =\gamma_{\pm 1}^2 \mathsf{F}_{\frac{d-3}{2}}=\mathsf{F}_{\frac{d-3}{2}} $$
and the monodromy group of $\mathsf{F}_{\frac{d-3}{2}}$ is $\mathbb{Z}$ if $d$ is odd with monodromy relation 
$$ \gamma_{\pm 1}\mathsf{F}_{\frac{d-3}{2}}=\mathsf{F}_{\frac{d-3}{2}}\pm 2i\pi(1-\zeta^2)^{\frac{d-3}{2}} ,$$
where $\gamma_{\pm 1}$ is a path in $\mathbb{C}\setminus\{-1,1\}$ of index $1$ around $\pm 1$.
  \end{enumerate}
\end{lemma}

\begin{remark}
\label{r:lemme-pret-a-utiliser}
According to Remark~\ref{r:log-definition} (in which the functions $\ln_\theta$, $_\theta\sqrt{\cdot}$ are defined), we also deduce from~\eqref{e:fonctions-speciales-1}--\eqref{e:fonctions-speciales-2}, that for all $\theta \in [-\pi,\pi]$, for all $\zeta$ belonging to either the upper or the lower half-plane (depending on the cut of $\ln_\theta$ on $e^{-i\theta}\R_-$),
 \begin{align*}
  \mathsf{F}_{\frac{d-3}{2}}(\zeta) & =  (1-\zeta^2)^{\frac{d-3}{2}} \big( \ln_\theta(\zeta+1)-\ln_\theta(\zeta-1) \big)  + P_d(\zeta)  , \quad \text{ for }d\geq 3 \text{ odd}, \\
  \mathsf{F}_{\frac{d-3}{2}}(\zeta)&  =   \pi (-1)^{\frac{d-2}{2}}  {}_\theta\sqrt{\zeta+1}^{ d-3}  {}_\theta\sqrt{\zeta-1}^{d-3}  + P_d(\zeta)  , \quad \text{ for }d\geq 2 \text{ even} .
  \end{align*}
The present remark is used with $\theta=\frac{\pi}{2}$ in the statement and proof of Proposition~\ref{p:main-I}. In this setting, recalling the explicit expression  of  $\ln_{\frac{\pi}{2}}$ and $_{\frac{\pi}{2}}\sqrt{\cdot}$ in Remark~\ref{r:log-definition},
 it states that for all $\zeta\in \C$ with $\Im(\zeta)>0$, we have 
  \begin{align*}
  \mathsf{F}_{\frac{d-3}{2}}(\zeta) & =  (1-\zeta^2)^{\frac{d-3}{2}} \big( \ln(i\zeta+i)-\ln(i\zeta-i) \big)  + P_d(\zeta)  , \quad \text{ for }d\geq 3 \text{ odd}, \\
  \mathsf{F}_{\frac{d-3}{2}}(\zeta)&  =   \pi (-1)^{\frac{d-2}{2}} (-i)^{d-3} \sqrt{i\zeta+i}^{d-3}\sqrt{i\zeta-i}^{d-3}  + P_d(\zeta)  , \quad \text{ for }d\geq 2 \text{ even} .
  \end{align*}
The present remark is also used with $\theta=-\frac{\pi}{2}$ in the statement of Theorem~\ref{t:maintheo-multiplication}.

Note finally from Item~\eqref{i:move-slit} combined with~\eqref{e:fonctions-speciales-1}--\eqref{e:fonctions-speciales-2}, that we can actually consider the functions 
$\zeta \mapsto (1-\zeta^2)^{\frac{d-3}{2}} \big( \ln(\zeta+1)-\ln(\zeta-1) \big)$ for $d$ odd and $\zeta \mapsto(\zeta+1)^{\frac{d-3}{2}}(\zeta-1)^{\frac{d-3}{2}}$ for $d$ even as holomorphic functions in $\C \setminus \Gamma$, where $\Gamma \in C^1_{pw}([0,1],\C \setminus\{\zeta_0\})$ is any path with $\Gamma(0) = -1$ and $\Gamma(1)=1$ and $\Gamma([0,1])$ is {\em strictly homotopic} to $[-1,1]$. 
This furnishes a holomorphic continuation statement of these special functions in the main results of the article.

\end{remark}

\begin{proof}[Proof of Lemma~\ref{l:laplace-bessel}]
The first statement of Item~\ref{i:explicit-012} follows from the fact that for $\zeta\in\R$ such that $\zeta> 1$, we have
$$
 \mathsf{F}_0(\zeta)= \int_{-1}^1\frac{1}{\zeta-t} dt = - \int_{-1}^1\d_t\big(\ln(\zeta-t)\big)  dt  =- \big( \ln (\zeta-1)- \ln(\zeta+1)\big) .
$$
Concerning the second statement, we set $t=\sin\tau$ to obtain, for $ \zeta > 1$,
\begin{align*}
 \mathsf{F}_{-\frac12}(\zeta) &= \int_{-\frac{\pi}{2}}^{\frac{\pi}{2}} \frac{d\tau}{\zeta-\sin(\tau)} = \frac{1}{2} \int_{0}^{2\pi} \frac{d\tau}{\zeta-\sin(\tau)} .
 \end{align*}
 Denoting by $\gamma =\{e^{i\tau}, \tau \in [0,2\pi]\}$ the unit circle and setting $z=e^{i\tau}$ whence $\frac{dz}{z}=id\tau$, we deduce 
\begin{align*}
 \mathsf{F}_{-\frac12}(\zeta) & =\frac{1}{2i} \int_{\gamma} \frac{1}{\zeta-\frac{z-z^{-1}}{2i}}\frac{dz}{z}
=-\int_{\gamma} \frac{1}{z^2-2iz\zeta-1}dz\\
&=-\int_\gamma \frac{dz}{(z-i\zeta + i\sqrt{\zeta^2-1})(z-i\zeta -i\sqrt{\zeta^2-1 })} .
 \end{align*}
Since $\zeta - \sqrt{\zeta^2-1} < 1$ for $\zeta >1$ the Cauchy residue formula then yields
\begin{align*}
 \mathsf{F}_{-\frac12}(\zeta) &=- \frac{2i\pi}{i\zeta -i\sqrt{\zeta^2-1 }-i\zeta -i\sqrt{\zeta^2-1 }  } =\frac{2i\pi}{2i\sqrt{\zeta^2-1 }}=\frac{\pi}{\sqrt{\zeta^2-1}} ,
\end{align*}
which concludes the proof of Item~\ref{i:explicit-012}.

 \medskip
 We now prove Item~\ref{i:recurrence}. 
Recalling~\eqref{e:def-Fgamma} and writing $1-t^2 = (\zeta-t)(\zeta+t)+(1-\zeta^2)$, we obtain
 \begin{align*}
  \mathsf{F}_{\gamma+1}(\zeta) & = \int_{-1}^1(1-t^2)^\gamma \frac{(1-t^2)}{\zeta-t} dt  =  \int_{-1}^1(1-t^2)^\gamma (\zeta+t) dt  + (1-\zeta^2) \int_{-1}^1(1-t^2)^\gamma \frac{1}{\zeta-t} dt \\
  & = C_{\gamma} \zeta  + (1-\zeta^2)  \mathsf{F}_{\gamma}(\zeta) , \quad \text{ with }C_{\gamma}=\int_{-1}^1(1-t^2)^\gamma dt >0  , 
  \end{align*}
 where we have used that $ \int_{-1}^1(1-t^2)^\gamma t dt =0$ since the function $(1-t^2)^\gamma t$ is odd.

  \medskip
 An induction argument in Item~\ref{i:recurrence} shows that for all $k \in \N^*$, 
 $$
 \mathsf{F}_{\gamma+k}(\zeta) = (1-\zeta^2)^k  \mathsf{F}_{\gamma}(\zeta) + (-1)^{k-1} C_\gamma\zeta^{2k-1} + R_k(\zeta),\quad\text{with}\quad R_k\in\R_{2k-2}[X].
 $$ 
 This, together with Item~\ref{i:explicit-012} and the fact that $C_\gamma>0$ proves Item~\ref{i:straight-induct}.
 
 \medskip
As for Item~\ref{i:move-slit}, assume first that $\gamma \in \N$. Then the function $t \mapsto (1-t^2)^\gamma$ is holomorphic in $\C$. 
As a consequence, there exists $\eps>0$ such that for all $\zeta \in \C, |\zeta-\zeta_0| <\eps$, the function $t \mapsto \frac{(1-t^2)^\gamma}{\zeta-t}$ is holomorphic in a neighborhood of $\bigcup_{s\in [0,1]}\Gamma_s([0,1])$. The classical deformation argument (see e.g.~\cite[Theorem~5.1 p93]{SteinShak}) then shows that, for all $s\in[0,1]$,
$$
  \mathsf{F}_{\gamma}(\zeta) = \int_{\Gamma_s} \frac{(1-t^2)^\gamma}{\zeta-t} dt   ,\quad \text{ for all } \zeta \in \C, |\zeta-\zeta_0| <\eps .
$$
The expression in the right-hand side defines a holomorphic function in $\C \setminus \Gamma_s([0,1])$, coinciding with $\mathsf{F}_{\gamma}$ on $\zeta \in \C, |\zeta-\zeta_0| <\eps$. Analytic continuation concludes the proof of the lemma in this case.

If we now assume that $\gamma \notin \N$ but $\gamma \geq 0$, then the same proof works except that $t \mapsto (1-t^2)^\gamma$ is not holomorphic in a neighborhood of $\bigcup_{s\in [0,1]}\Gamma_s([0,1])$.
It is however holomorphic in the interior of $\bigcup_{s\in [0,1]}\Gamma_s([0,1])$, and continuous on $\bigcup_{s\in [0,1]}\Gamma_s([0,1])$. The proof then goes through without any further modification. 

Finally, if $\gamma \in (-1,0)$, $t \mapsto (1-t^2)^\gamma$ is no longer continuous at $\pm 1$. it is however integrable on any path ending at/starting from the points $\pm1$, and thus a regularization argument (if we remove an $\eps$-neighborhood of these two points from the deformation of contour argument, the contribution of the parts removed converges to zero as $\eps\to 0$) shows that the proof goes through as in the previous two cases.

\medskip
Finally, Item~\ref{i:Nilsson-F} is a direct consequence of the explicit expression of $\mathsf{F}_{\frac{d-3}{2}}$ in Item~\ref{i:straight-induct} together with the monodromy of the functions $\ln$ and $\sqrt{\cdot}$ described in Appendix~\ref{ex:monodromyexplicit}. When $d$ is even, the determination of $\mathsf{F}_{\frac{d-3}{2}}$ is equal to $2$ since $P_d(z)\neq 0$. and the monodromy group is equal to $\mathbb{Z}/2\mathbb{Z}$. Similarly, when $d$ is odd, the determination is equal to $2$ and the monodromy is equal to $\mathbb{Z}$.

\end{proof}

\section{A quantitative Morse lemma in the analytic category} 
\label{ss:analyticMorse}

The goal of this appendix is to state and prove the Morse lemma in the analytic category with some attention paid to the parameters.  
Our proof is sketched in~\cite[p.~6-7]{zoladek2006} and also is an analytic adaptation of the proof in \cite[Lemma 6.C.1 p.~502]{hormander2007analysis}.

\begin{theorem}\label{theo:morselemmauniform}
Let $U \subset \R^{d-1}$ be an open set with $0\in U$ and $f_\omega:U\subset \mathbb{R}^{d-1}\to \mathbb{R}$ be a family of real-analytic functions parametrized by $\omega$ in some compact space $K\subset \mathbb{C}^N$ such that $\forall\omega\in K$,
$0$ is a critical point of $f_\omega$ with Hessian $d^2f_\omega(0)$ definite positive.
Then there exist $\rho>0$, an open set $U_\rho$ with $0\in U_\rho \subset U$ and, for any $\omega \in K$, a real-analytic diffeomorphism $\phi_\omega:B_{\R}^{d-1}(0,\rho) \to U_\rho$ depending analytically on $\omega\in K$ such that 
$$f_\omega\circ \phi_\omega(x)=f_\omega(0) + \frac12\sum_{i=1}^{d-1} x_i^2 , \quad \text{ for all } x \in B_{\R}^{d-1}(0,\rho) .$$ 
\end{theorem}
In fact, it might be possible that we only need continuous dependence in $\omega$ for this lemma but since in our problem all the dependence are analytic we choose to keep the analyticity assumption. Notice that this is a real-analytic Morse lemma, in the sense that it is formulated on the real domain only.

\begin{proof}
The Taylor formula with integral remainder writes
 $$f_\omega(y')=f_\omega(0)+y'^{T}\left(\int_{0}^1(1-t)d^2f_\omega(ty')dt\right)y'.$$
 One would like to express the $y'$ dependent symmetric matrix $\left(\int_{0}^1(1-t)d^2f_\omega(ty')dt\right)$ as a matrix product of the form $M^T(y',\omega)Q_{0,\omega} M(y',\omega)$ where the matrix $M$ depends analytically on $(y',\omega)$, $M(0,\omega)=\Id$ and $Q_{0,\omega}=d^2f(0,\omega)$ is the Hessian of $f_\omega$ at $0$. The idea is to use the inverse funtion Theorem in the analytic category.
 Given an invertible symmetrix matrix $Q_{0,\omega}$, the map 
 $$\mathcal{M}_\omega:M\in GL_n(\IR) \mapsto M^T Q_{0,\omega} M\in S_n(\IR)$$
is analytic in a neighborhood of $\Id$. Moreover, the differential $d_{\Id}\mathcal{M}_\omega$ at the identity is given by 
$$ d_{\Id}\mathcal{M}_\omega: H\in M_n(\mathbb{R})\mapsto H^TQ_{0,\omega}+Q_{0,\omega}H \in S_n(\mathbb{R})$$ which is surjective on $S_n(\mathbb{R})$, indeed a solution to $H^TQ_{0,\omega}+Q_{0,\omega}H =C$ reads $H=\frac{1}{2}Q_{0,\omega}^{-1}C$ using the invertibility of $Q_{0,\omega}$. Hence, the map $\mathcal{M}_\omega$ is invertible in a small (uniformly in $\omega\in K$) neighborhood of the identity with an inverse which depends holomorphically/real-analytically on all variables of the problem. By applying the result of~\cite[Prop 6.1 p.~138]{cartanelementary}, there exists $M(C,\omega)$ as a function of $C$ defined near $Q_{0,\omega}$ (uniformly in $\omega\in K$) such that $\mathcal{M}_\omega(M(C,\omega))=C$, $M$ is analytic in all variables $\omega,C$. We can now apply this result to the analytic map 
$$Q_\omega:y'\mapsto 2 \int_{0}^1(1-t)d^2f_\omega(ty')dt \in S_n(\IR),\quad \text{verifying}\quad  Q_\omega(0)=d^2f_\omega(0)\in S_n(\IR)\cap GL_n(\IR),$$
and one finds that
$$f_\omega(y')=f_\omega(0)+\frac{1}{2}\left\langle\ml{M}_\omega\circ Q_\omega(y')y',d^2f_\omega(0)\ml{M}_\omega\circ Q_\omega(y')y'\right\rangle.$$ 
Since we assume that $d^2f_\omega(0)$ is positive definite, taking $P_\omega := \sqrt{d^2f_\omega(0)} \in GL_n(\IR)$ (analytically depending on $\omega$) we have $d^2f_\omega(0)=P_\omega^TP_\omega$ and thus
$$f_\omega(y')=f_\omega(0)+\frac{1}{2}\left\|P_\omega\ml{M}_\omega\circ Q_\omega(y')y'\right\|^2.$$ 
To conclude, we remark that the map $\psi_\omega : y'\mapsto P_\omega\mathcal{M}_\omega\circ Q_\omega(y')y'$ is real-analytic and satisfies $\psi_\omega(0)=0$ and $d\psi_\omega(0)=P_\omega\ml{M}_\omega\circ Q_\omega(0) = P_\omega$ invertible. Using the local inversion theorem in the real-analytic/holomorphic category~\cite{cartanelementary}, $\psi_\omega$ is a local real-analytic diffeomorphism fixing $0$ and such that $f_\omega(y')=f_\omega(0)+\frac{1}{2}\left\|\psi_\omega(y')\right\|^2.$ Denoting by $\phi_\omega : = \psi_\omega^{-1}$ its local inverse $B_{\R^{d-1}}(0,\rho) \to U_\rho$ (up to reducing its domain of definition), which is still a real-anaytic diffeomorphism, we obtain that 
 $f_\omega(\phi_\omega(x))=f_\omega(0)+\frac{1}{2}\left\|x\right\|^2,$ which is the sought result.
\end{proof}

\bibliographystyle{alpha}
\bibliography{allbiblio}

\end{document}